\documentclass[a4paper,10pt,final ]{amsart}
\usepackage[utf8]{inputenc}
\pdfoutput=1
\usepackage{amsmath, amsthm, amssymb, enumerate, scalerel, tikz,  color, accents,
  fancybox, rotating, comment, colortbl, stmaryrd, hyperref,
  cleveref, setspace, tikz-3dplot, xlop}
\title{MOY calculus in type \texorpdfstring{$\mathsf{D}$}{D}}
\author{Elijah Bodish}
\address{Department of Mathematics, Massachusetts Institute of Technology
77 Massachusetts Ave. Cambridge, MA USA 02139}
\email{\href{mailto:ebodish@mit.edu}{ebodish@mit.edu}}
\author{Louis-Hadrien Robert}
 \address{Université Clermont Auvergne, LMBP, Campus des Cézeaux, 3 place Vasarely, TSA 60026, CS 60026, 63178 Aubière Cedex, France}
 \email{\href{mailto:louis-hadrien.robert@uca.fr}{louis-hadrien.robert@uca.fr}}
 \author{Emmanuel Wagner}
 \address{Univ Paris Cité, IMJ-PRG, Univ Paris Sorbonne, UMR 7586 CNRS,
   F-75013, Paris, France} 
 \email{\href{mailto:emmanuel.wagner@imj-prg.fr}{emmanuel.wagner@imj-prg.fr}}
\makeatletter
\@namedef{subjclassname@2020}{\textup{2020} Mathematics Subject Classification}
\makeatother
\hypersetup{
    pdftoolbar=true,        
    pdfmenubar=false,        
    pdffitwindow=true,     
    pdfstartview={FitH},    
    pdftitle={\shorttitle},    
    pdfauthor={\shortauthors},     
    pdfsubject={\shorttitle},   
    pdfcreator={\shortauthors},   
    pdfproducer={\shortauthors}, 
    pdfkeywords={}, 
    pdfnewwindow=true,      
    colorlinks=true,       
    linkcolor=darkgray,          
    citecolor=teal,        
    filecolor=magenta,      
    urlcolor=violet,          
    linkbordercolor=red,
    citebordercolor=teal,
    urlbordercolor=violet,  
    linktocpage=true
    }
\usetikzlibrary{arrows, decorations.markings, decorations, patterns,
positioning, decorations.pathreplacing, decorations.pathmorphing,
intersections, fit}

\let\oldtocsubsection\tocsubsection
\renewcommand\tocsubsection[3]{\hspace{0.5cm}\oldtocsubsection{#1}{#2}{#3}}
\let\oldtocsubsubsection\tocsubsubsection
\renewcommand\tocsubsubsection[3]{\hspace{1cm}\oldtocsubsubsection{#1}{#2}{#3}}
\newcounter{res}[section]
\numberwithin{res}{section}
\newtheorem{thm}[res]{Theorem}
\newtheorem{maintheo}{Theorem}

\newtheorem{lem}[res]{Lemma}
\newtheorem{lem-dfn}[res]{Lemma-Definition}

\newtheorem{prop}[res]{Proposition}
\newtheorem{cor}[res]{Corollary}

\theoremstyle{definition}
\newtheorem{notation}[res]{Notation}
\newtheorem{dfn}[res]{Definition}
\newtheorem{rmk}[res]{Remark}
\newtheorem{exa}[res]{Example}

\newtheorem{conv}[res]{Convention}
\newcommand{\web}{\ensuremath{\Gamma}}
\newcommand{\webspin}[1][\web]{\ensuremath{#1}_{\mathrm{spin}}}
\newcommand{\gweb}{generalized web}

\newcommand{\NB}[1]{\ensuremath{\vcenter{\hbox{#1}}}}

\usetikzlibrary{math, arrows, decorations.markings, decorations, patterns, positioning, decorations.pathreplacing,
  decorations.pathmorphing, decorations.text, decorations.shapes, fpu}

\tikzset{set style/.style={#1},
  expand style/.style={set style/.expanded={#1}},
  expand style once/.style={set style/.expand once={#1}},
  expand style twice/.style={set style/.expand twice={#1}}
}

 \tikzset{zxplane/.style={canvas is zx plane at y=#1}}
 \tikzset{yxplane/.style={canvas is yx plane at z=#1}}
 \tikzset{zyplane/.style={canvas is zy plane at x=#1}}

\newcommand{\coloreddiagrams}
{
\definecolor{colorvect}{RGB}{10,10,10}
\definecolor{colorodd}{RGB}{0,0,255}
\definecolor{coloreven}{RGB}{0,127,0}
\definecolor{colorcycle}{RGB}{200,120,0}
}

\newcommand{\graydiagrams}
{
\definecolor{colorvect}{RGB}{100,100,100}
\definecolor{colorodd}{RGB}{100,100,100}
\definecolor{coloreven}{RGB}{100,100,100}
\definecolor{colorspin}{RGB}{100,100,100}
}

\newcommand{\stylespin}{colorspin, line width=1mm}
\newcommand{\stylecycle}{colorcycle, line width=2mm, opacity=0.3, rounded corners}
\newcommand{\stylevect}{colorvect, thin}
\newcommand{\styleodd}{colorodd, thin, decorate, decoration={snake, segment
    length=1mm,amplitude=0.5mm }}
\newcommand{\styleeven}{coloreven, thin, double}
\newcommand{\midarrow}[2]{node[opacity= 0.7, pos=0.5, #1, sloped, font=\normalsize,
  scale=0.7, rotate={#2}] {$\blacktriangleright$}}
\newcommand{\midarrowno}[1]{node[opacity= 0.7, pos=0.5,  sloped, font=\normalsize,
  scale=0.7, rotate={#1}] {$\blacktriangleright$}}

\newcommand{\arr}{\midarrowno{0}}
\newcommand{\rarr}{\midarrowno{180}}
\newcommand{\arreven}{\midarrow{coloreven}{0}}
\newcommand{\arrodd}{\midarrow{colorodd}{0}}

\newcommand{\rarreven}{\midarrow{coloreven}{180}}
\newcommand{\rarrodd}{\midarrow{colorodd}{180}}

\newcommand{\singvertex}{{\tiny $\bullet$}}
\tikzset{->-/.style={decoration={markings, mark=at position .5 with {\arrow{>}}},postaction={decorate}}}
\tikzset{-<-/.style={decoration={markings, mark=at position .5 with {\arrow{<}}},postaction={decorate}}}

\newcommand{\singvertextikz}[1][]{
  \begin{scope}[#1]
    \coordinate (A) at ( 135:1) ;
    \coordinate (B) at (  45:1) ;
    \coordinate (C) at ( -45:1) ;
    \coordinate (D) at (-135:1) ;
    \node (O) at (0,0) {\singvertex};
    \draw (A) --(C);
    \draw (B) --(D);
  \end{scope}
}

\newcommand{\myor}{\ensuremath{\bigcirc}}
\newcommand{\slN}[1][\myN]{\ensuremath{\mathfrak{sl}_{#1}}}
\newcommand{\glN}[1][\myN]{\ensuremath{\mathfrak{gl}_{#1}}}
\newcommand{\so}{\ensuremath{\mathfrak{so}}}
\newcommand{\myN}{\ensuremath{N}}
\newcommand{\sotN}[1][2\myN]{\ensuremath{\so_{#1}}}
\newcommand{\RT}[2][\sotN]{\ensuremath{\mathrm{RT}_{#1}\left(#2\right)}}
\newcommand{\Ashift}[1]{\ensuremath{\mathrm{sh}^{\A}\left(#1\right)}}
\newcommand{\shiftN}[2][\myN]{\ensuremath{\mathrm{sh}_{#1}\left(#2\right)}}
\newcommand{\gshiftN}[2][\myN]{\ensuremath{\mathrm{gsh}_{#1}\left(#2\right)}}
\newcommand{\rotational}[2][]{\ensuremath{\rho_{#1}\left(#2\right)}}

\newcommand{\setN}[1][\myN]{\ensuremath{\llbracket{#1}\rrbracket}}
\newcommand{\powerset}[1]{\ensuremath{\mathcal{P}}}
\newcommand{\ev}[1]{\ensuremath{\left\langle #1 \right\rangle}} \newcommand{\evN}[2][\myN]{\ensuremath{\ev{#2}_{#1}}}
\newcommand{\evA}[2][\glN]{\ensuremath{\left\langle #2 \right\rangle_{#1}}}
\newcommand{\cycle}{\ensuremath{Z}}
\newcommand{\col}[2][]{\ensuremath{\mathrm{col}\ifstrempty{#1}{}{_{#1}}\left(#2\right)}}

\renewcommand{\deg}[1][]{\mathrm{deg}\ifstrempty{#1}{}{_{#1}}}
\newcommand{\NN}{\ensuremath{\mathbb{N}}}
\newcommand{\ZZ}{\ensuremath{\mathbb{Z}}}
\newcommand{\RR}{\ensuremath{\mathbb{R}}}

\newcommand{\A}{\ensuremath{\mathsf{A}}}
\newcommand{\webA}{\ensuremath{G}}

\newcommand{\ii}{{\rotatebox{90}{=}}}

\newcommand{\wt}{\ensuremath{\mathrm{wt}}}
\newcommand{\diag}{\ensuremath{D}}

\newcommand{\ttkzAshift}[5][]{\NB{\tikz[#1]{
\begin{scope}[scale=0.6, font =\tiny]
  \draw[expand style =\stylevect, #2] (0,0) -- (1,0);
  \draw[expand style =\styleeven] (0,0) -- (0,1) \arreven node[pos=0.85, #3] {$#4$};
  \draw[expand style =\styleodd]  (0,0) -- (0,-1) \rarrodd node[pos= 0.85, #3] {$#5$};
\end{scope}
    }}}

\newcommand{\eval}{\mathrm{eval}}
\newcommand{\coeval}{\mathrm{coev}}
\newcommand{\capp}{\mathrm{cap}}
\newcommand{\cupp}{\mathrm{cup}}
\newcommand{\id}{\mathrm{id}}
\newcommand{\Ya}{\mathrm{Y}}
\newcommand{\Pa}{\reflectbox{\rotatebox[origin=c]{180}{$\mathrm{Y}$}}}
\newcommand{\HH}{\mathrm{H}}
\newcommand{\II}{\mathrm{I}}
\newcommand{\XX}{\mathrm{X}}
\newcommand{\Hom}{\mathrm{Hom}}
\newcommand{\End}{\mathrm{End}}

\newcommand{\mymacrotwo}[2]{\ensuremath{C_{{#1}^\star{}{#2}}}}
\newcommand{\mymacroone}[2]{\ensuremath{C_{{#1}'{#2}}}}

 \newcommand{\imagesfolder}{.}

\begin{document}
\begin{abstract}We define a positive state sum for webs ``of type $\mathsf{D}$''. These webs are
graphs which mimic morphisms in the category of finite-dimensional
$U_q(\sotN)$-modules. From the state sum, we derive an invariant of
framed unoriented links. After giving explicit details about
some intertwiners in the category of $U_q(\sotN)$-modules, we relate
our state-sum link invariant with Reshetikhin--Turaev's invariant associated
with $U_q(\sotN)$.
\end{abstract}
\maketitle
\setcounter{tocdepth}{1}

\tableofcontents

\coloreddiagrams
\section{Introduction}
\label{sec:intro}

\subsection{History}
\label{sec:background}

Quantum topology started with Jones' discovery of the eponymous
polynomial \cite{MR766964}. Inspired by this definition, Reshetikhin
and Turaev \cite{MR1036112} gave a rather general framework to
construct polynomial knot invariants given a choice of ``color'', i.e.{} a finite dimensional representations of a ribbon Hopf
algebra. 

A rich source of such colors are finite dimensional representations of Drinfeld-Jimbo quantum groups. These quantum groups are $q$-analogues of universal
enveloping algebras of semisimple Lie algebras. Moreover, the finite dimensional representations of these quantum groups are $q$-analogues of finite dimensional representations of semisimple Lie algebras. 

When the semisimple Lie algebra
is $\mathfrak{g}= \slN[2]$ and the choice of color is the $q$-analogue
of $\mathbb{C}^2$, the Reshetikhin--Turaev construction recovers the
Jones polynomial. Although the Reshetikhin--Turaev construction is quite general, the
majority of constructions and in-depth studies following their work focus on colored $\slN[n]$-polynomials, which are
the closest generalizations of the Jones polynomial. 

One reason these invariants are the most studied is
that explicit computations are
easier. For example, Murakami--Othsuki--Yamada \cite{MOY}
give a very elementary and self-contained combinatorial description of Reshetikhin--Turaev invariants colored by the $q$-analogues of $\Lambda^k(\mathbb{C}^\myN)$. In fact they do even more, and also define an invariant of knotted $\mathfrak{gl}_\myN$ web diagrams, which are certain trivalent graphs colored by the quantum exterior powers. 

A key feature of Murakami--Othsuki--Yamada's approach is integrality and
positivity of the quantity associated with planar diagrams, i.e. given a planar $\mathfrak{gl}_\myN$ web diagram $W$ we have $\left\langle W\right\rangle_{\mathfrak{gl}_\myN} \in \mathbb{Z}_{\ge 0}[q,q^{-1}]$. This feature was essential for defining Khovanov--Rozansky $\glN[\myN]$
link homology theories \cite{MR2391017,MR2421131}.

This wonderful paper by Murakami--Othsuki--Yamada could have originally been thought of as just a diagrammatic and combinatorial reformulation of the algebraic framework of $R$-matrices. But their work has had a large impact in the development of categorified link invariants \cite{MR4164001}, especially the key functorial properties \cite{FunctorialitySLN}.

By leveraging the functorial properties, the algebraic and combinatorial approach to categorification has resulted in breakthroughs for low-dimensional topology and geometry, for example Picirillo's work on the Conway knot \cite{MR4076631} and defining the skein-Lasagna invariants of oriented smooth four manifolds \cite{MR4562565}. It is also amazing to see the combinatorics and algebra approach to categorification replacing analytic gauge theory techniques, e.g. Rasmussen's reproof of the Milnor conjecture \cite{MR2729272}.

\subsection{Our work}
\label{sec:ourwork}

Inspired by the seminal work of Murakami--Othsuki--Yamada and the
representation theory of $U_q(\sotN)$, we define a family of
invariants $\left(\evN{\cdot}\right)_{\myN\geq 3}$ of
unoriented framed links. The details of the construction are analogous to the
rephrasing of \cite{MOY} by the second author \cite{LHRmoy}. The invariants possess integrality and positivity properties that one seeks for well-behaved categorifications.

A consequence of our work is a new proof of the following existence theorem. \begin{maintheo}
  [Theorem \ref{thm:restated-thm-A}]\label{theo:A}
    There is of a family of link invariants satisfying the following skein relations:
\begin{gather}
     \evN{\,\,\NB{\tikz[rotate=90]{\begin{scope}[scale=0.6]
  \draw[expand style =\stylevect] (45:1) -- (-135:1);
  \fill[white] (0,0) circle (0.2cm);
  \draw[expand style =\stylevect] (-45:1) -- (135:1);
  %\node at (0,0) {\singvertex};
\end{scope}}}\,\,} -
   \evN{\,\,\NB{\tikz[rotate=00]{}}\,\,}
   = (q-q^{-1})\left(\evN{\,\,\NB{\tikz[rotate=90]{\begin{scope}[scale=0.6]
  \draw[expand style =\stylevect] (45:1) .. controls (45:0.5) and (-45:0.5) .. (-45:1);
  \draw[expand style =\stylevect] (135:1) .. controls (135:0.5) and (-135:0.5) .. (-135:1);
\end{scope}}}\,\,} -
     \evN{\,\,\NB{\tikz[rotate=00]{}}\,\,} \right), \\
       \evN{\NB{\tikz[scale=0.5]{\begin{scope}[expand style = \stylevect]
  \coordinate (a) at (-0.3,0.3);
  \coordinate (b) at (0.3,0.3);
  \coordinate (c) at (0.3,-0.3);
  \coordinate (d) at (-0.3,-0.3);
  \coordinate (O) at (0,0);
  \coordinate (T) at (-0.3, 0.9);
  \coordinate (B) at (-0.3,-0.9);
  \draw (T) .. controls +(0, -0.2) and +(-0.2, 0.2) .. (a);
  \draw (B) .. controls +(0, +0.2) and +(-0.2, -0.2) .. (d);  
  \draw (a) -- (c);
  \fill[white] (O) circle (1mm); 
  \draw (b) -- (d);
  \draw (b) .. controls +(0.4, 0.4) and +(0.4, -0.4) .. (c);
\end{scope}}}} = -q^{2\myN -1} \evN{\,\NB{\tikz[scale=0.5]{\begin{scope}[expand style = \stylevect]
  \coordinate (T) at (-0.3, 0.9);
  \coordinate (B) at (-0.3,-0.9);
  \draw (T) -- (B);
\end{scope}}}\,}, \qquad \qquad
    \evN{\NB{\tikz[scale=0.5, yscale=-1]{}}} = -q^{1-2\myN}
    \evN{\,\NB{\tikz[scale=0.5]{}}\,}, \\ \text{and}\qquad  \evN{\NB{\tikz[]{  \begin{scope}[scale = 0.6]
    \draw[expand style =\stylevect] (0,0) circle (0.5cm); 
  \end{scope}

}}} = \left([2\myN-1] +1\right).  
\end{gather}
\end{maintheo} 
The first existence proof is due to Kauffman \cite{MR958895}. In fact, the invariants we consider are 
specializations of the two-variable Kauffman polynomial, the existence and uniqueness of which is proven in the place cited. Another proof of the existence of such a family of invariants, which we recall in detail in Section \ref{sec:intertwiners}, uses quantum groups, while yet another existence proof uses the BMW algebras \cite{MR992598, MR927059}.

To define the family of link invariants, we first define a family of invariants of $4$-valent non-oriented planar graphs. The relationship between the invariants of links are then obtained from the invariants of graphs by using the following relationship:
\begin{equation}\label{eqn:typeDskein}
    \ev{\,\,\NB{\tikz[rotate=90]{}}\,\,} = -q\ev{\,\,\NB{\tikz[]{}}\,\,} + \ev{\,\,\NB{\tikz[]{\begin{scope}[scale=0.6]
  \draw[expand style =\stylevect] (45:1) -- (-135:1);
  \draw[expand style =\stylevect] (-45:1) -- (135:1);
  \node at (0,0) {\singvertex};
\end{scope}}}\,\,} -q^{-1}\ev{\,\,\NB{\tikz[rotate=90]{}}\,\,} .
\end{equation}
In fact, ideas of Kauffman and Vogel \cite{KV} imply that defining the invariant for non-oriented planar $4$-valent graphs is equivalent to defining the invariant for links.

To define the invariant for planar $4$-valent graphs, we actually first define the invariant for planar trivalent graphs, equipped with a labeling of edges and a ``semi-orientation", for example:
  \begin{gather*}
    \NB{\tikz[scale=0.6]{\begin{scope}
  \draw[expand style =\stylevect] (0,0) -- (90:1);
  \draw[expand style =\styleeven] (0,0) -- (210:1) \arreven;
  \draw[expand style =\styleodd] (0,0) -- (-30:1) \arrodd;
\end{scope}
}} \qquad\qquad
    \NB{\tikz[scale=0.6]{\begin{scope}[xshift=0cm]
  \draw[expand style =\stylevect] (0,0) -- (90:1);
  \draw[expand style =\styleeven] (0,0) -- (-30:1) \arreven;
  \draw[expand style =\styleodd] (0,0) -- (210:1) \arrodd;
\end{scope}
}} \qquad\qquad
    \NB{\tikz[scale=0.6]{\begin{scope}[xshift=0cm]
  \draw[expand style =\stylevect] (0,0) -- (90:1);
  \draw[expand style =\styleeven] (0,0) -- (210:1) \rarreven;
  \draw[expand style =\styleodd] (0,0) -- (-30:1) \rarrodd;
\end{scope}
}} \qquad\qquad
    \NB{\tikz[scale=0.6]{\begin{scope}[xshift=0cm]
  \draw[expand style =\stylevect] (0,0) -- (90:1);
  \draw[expand style =\styleeven] (0,0) -- (-30:1) \rarreven;
  \draw[expand style =\styleodd] (0,0) -- (210:1) \rarrodd;
\end{scope}
}}
  \end{gather*}
Then, we use the following relationship between trivalent and $4$-valent vertices to define the invariant for $4$-valent graphs.
\begin{gather}
  \label{eq:4v-def-int}
     [2]^{[N-3]} \ev{\NB{\tikz[]{}} } :=  \ev{\NB{\tikz[]{  \begin{scope}[scale = 0.6]
    \coordinate (A) at ( 135:1) ;
    \coordinate (B) at (  45:1) ;
    \coordinate (C) at ( -45:1) ;
    \coordinate (D) at (-135:1) ;
    \coordinate (a) at ( 135:0.6) ;
    \coordinate (b) at (  45:0.6) ;
    \coordinate (c) at ( -45:0.6) ;
    \coordinate (d) at (-135:0.6) ;
    \draw (A) --(a);
    \draw (B) --(b);
    \draw (C) --(c);
    \draw (D) --(d);
    \draw[expand style =\styleeven] (a) -- (b) \arreven;
    \draw[expand style =\styleodd] (b) -- (c) \arrodd;
    \draw[expand style =\styleeven] (c) -- (d) \rarreven;
    \draw[expand style =\styleodd] (d) -- (a) \arrodd;
  \end{scope}

}} }
     =  \ev{\NB{\tikz[]{  \begin{scope}[scale = 0.6]
    \coordinate (A) at ( 135:1) ;
    \coordinate (B) at (  45:1) ;
    \coordinate (C) at ( -45:1) ;
    \coordinate (D) at (-135:1) ;
    \coordinate (a) at ( 135:0.6) ;
    \coordinate (b) at (  45:0.6) ;
    \coordinate (c) at ( -45:0.6) ;
    \coordinate (d) at (-135:0.6) ;
    \draw (A) --(a);
    \draw (B) --(b);
    \draw (C) --(c);
    \draw (D) --(d);
    \draw[expand style =\styleeven] (a) -- (b) \rarreven;
    \draw[expand style =\styleodd] (b) -- (c) \rarrodd;
    \draw[expand style =\styleeven] (c) -- (d) \arreven;
    \draw[expand style =\styleodd] (d) -- (a) \rarrodd;
  \end{scope}

}} }
     =  \ev{\NB{\tikz[]{  \begin{scope}[scale = 0.6, rotate=90]
    \coordinate (A) at ( 135:1) ;
    \coordinate (B) at (  45:1) ;
    \coordinate (C) at ( -45:1) ;
    \coordinate (D) at (-135:1) ;
    \coordinate (a) at ( 135:0.6) ;
    \coordinate (b) at (  45:0.6) ;
    \coordinate (c) at ( -45:0.6) ;
    \coordinate (d) at (-135:0.6) ;
    \draw (A) --(a);
    \draw (B) --(b);
    \draw (C) --(c);
    \draw (D) --(d);
    \draw[expand style =\styleeven] (a) -- (b) \midarrow{coloreven}{90};
    \draw[expand style =\styleodd] (b) -- (c)  \midarrow{colorodd}{90};
    \draw[expand style =\styleeven] (c) -- (d) \midarrow{coloreven}{-90};
    \draw[expand style =\styleodd] (d) -- (a)  \midarrow{colorodd}{90};
  \end{scope}

}} }
     =  \ev{\NB{\tikz[]{  \begin{scope}[scale = 0.6, rotate=90]
    \coordinate (A) at ( 135:1) ;
    \coordinate (B) at (  45:1) ;
    \coordinate (C) at ( -45:1) ;
    \coordinate (D) at (-135:1) ;
    \coordinate (a) at ( 135:0.6) ;
    \coordinate (b) at (  45:0.6) ;
    \coordinate (c) at ( -45:0.6) ;
    \coordinate (d) at (-135:0.6) ;
    \draw (A) --(a);
    \draw (B) --(b);
    \draw (C) --(c);
    \draw (D) --(d);
    \draw[expand style =\styleeven] (a) -- (b) \midarrow{coloreven}{-90};
    \draw[expand style =\styleodd] (b) -- (c)  \midarrow{colorodd}{-90};
    \draw[expand style =\styleeven] (c) -- (d) \midarrow{coloreven}{90};
    \draw[expand style =\styleodd] (d) -- (a)  \midarrow{colorodd}{-90};
  \end{scope}

}} }.
   \end{gather}
The trivalent graphs we consider are inspired by their interpretation as intertwiners between representations of $U_q(\sotN)$, the Drinfeld-Jimbo quantum group associated to $\sotN$, the simple Lie algebra of Dynkin type $D$. 

We define $\mathsf{D}$-webs to be trivalent planar graphs, equipped with a labeling of edges and semi-orientation, see
Definition~\ref{dfn:webs}, and then define $\sotN$-colorings of $\mathsf{D}$-webs. For each $\mathsf{D}$-web $\web$, the set $\col[\sotN]{\web}$ of
$\sotN$-colorings is finite. Using the embedding in the oriented plane
and counting the algebraic number of oriented cycles arising from a
coloring $c$, we can define the degree $\deg(c)$ of $c$. Consequently, we define a state-sum (indexed by colorings) associated with the web $\web$:
\begin{equation}
  \label{eq:37}
  \evN{\web} := \sum_{c \in \col[\sotN]{\web}} q^{\deg(c)}.
\end{equation}
Thus, the invariants of $\mathsf{D}$-webs are $\ZZ_{\ge 0}[q, q^{-1}]$-valued and
defined in purely combinatorial terms. 

We can also extend our evaluation to the 4-valent setting, which implies the following.\begin{maintheo}[Theorem \ref{thm:B}]
  \label{theo:B}
    If $\Gamma$ is a $4$-valent non-oriented planar graph, then $\evN{\Gamma}\in \mathbb{Z}_{\ge 0}[q, q^{-1}]$.
\end{maintheo}

\begin{rmk}
  Theorem \ref{theo:B} does not follow from the definition of evaluation provided by equation \eqref{eq:4v-def-int}: it seems hard to ensure divisibility of $\ev{\web}$ by $\left([2]^{[\myN-3]}\right)^{v}$ for $v$ the number of $4$-valent vertices of $\web$. We are not aware of a proof that uses equation \eqref{eqn:typeDskein} either: it seems hard to ensure positivity of the coefficients. Instead the proof of Thereom \ref{theo:B} relies on an extension of the state sum to webs with $4$-valent vertices.
\end{rmk}

\begin{rmk}
    The centralizer algebra of a tensor power of the $2N$ dimensional representation of $U_q(\sotN)$ has a canonical basis \cite[27.3.10]{MR1227098}. There is no known algebraic characterization of this basis, although an attempt for the BMW algebra\footnote{This algebra depends on two parameters, $r$ and $q$. After specializing $r$ to a power of $q$, there is a surjective but not in general faithful map from the BMW algebra to the centralizer algebra.} appears in \cite{MR1312974}.

    Often canonical bases arise from categorification. One might dream that the positivity in Theorem \ref{theo:B} is a shadow of categorification, which in turn is related to some canonical basis for these centralizer algebras. 
\end{rmk}

To precisely relate our combinatorially defined link invariants to the Reshetikhin--Turaev link invariants associated with $U_q(\sotN)$ (colored by the $2\myN$ dimensional vector representation) we provide explicit formulas for specific intertwiners between tensor products of minuscule $U_q(\sotN)$ representations. In particular for intertwiners corresponding to the trivalent vertices above. After unraveling the general construction of Reshetikhin--Turaev in the case of $U_q(\sotN)$, we conclude with the following theorem.

\begin{maintheo}[Theorem \ref{thm:ev-vs-alg}]
    \label{theo:C}
        For any framed link $L$, the vector representation colored $\sotN$ Reshetikhin-Turaev invariant of $L$ is equal to our combinatorially defined invariant of $L$, after changing $q$ to $-q$.
      \end{maintheo}

Strikingly, one must change
$q$ to $-q$ in order to precisely go from one invariant to the other. In particular, the positivity, which is essential for categorification, is not
reached directly when using the more familiar $R$-matrix framework.

\begin{rmk}
    In this paper, we only define a combinatorial link invariant when the color of the link corresponds to the vector representation of $U_q(\sotN)$. There is a Reshetikhin--Turaev invariant of links with components colored by arbitrary irreducible representations of $U_q(\sotN)$, in particular the half-spin
    representations. For the present work half-spin
    representations are necessary for our combinatorial evaluation of webs, but we do not provide a combinatorial definition of a link invariant colored by half-spin representations.
\end{rmk}

\subsection{Results and Organization}
\label{sec:results}

The main contributions of this paper are the definition of colorings of
webs and of the degree of colorings. These definitions and some
examples of computation are given in Section~\ref{sec:webs}.

In Section~\ref{sec:global-relations}, we state and prove
\emph{global} relations satisfied by $\evN{\cdot}$. We consider
three flavors of global relations: the first globale relations say that
$\evN{\cdot}$ is invariant under certain global modification of
webs, such as changing all the orientations at once
(Corollary~\ref{cor:ev-all-equal}). The second global relations explain how
$\evN{\cdot}$ relates to
$\evN[N-1]{\cdot}$ (Proposition~\ref{prop:BR-typeD}). The third global relations relate $\evN{\cdot}$ and
$\ev{\cdot}_{\glN}$. Philosophically at least, the latter two global relations should reflect the branching rules from $\sotN$ to
$\sotN[2\myN-2]$ and from $\sotN$ to $\slN$ (Proposition~\ref{prop:BR-typeA}).

In Section~\ref{sec:skein}, we start by extending the
notion of webs to include a new kind of vertex (called a
\emph{singular} vertex). As shown in Lemma~\ref{lem:singular-square},
this new kind of vertex can be thought of as a shortcut for an actual
web (divided by a Laurent polynomial). Then we establish various local
``skein" relations satisfied by $\evN{\cdot}$.

Section~\ref{sec:linkinvariants} starts with a generalization of
$\evN{\cdot}$ to (some) knotted web diagrams, by expressing crossings in the knot diagram as a linear combinations of planar diagrams. In particular
this gives a definition of $\evN{\cdot}$ on link
diagrams. Theorem~\ref{thm:restated-thm-A} states that this is an
invariant of unoriented framed links. The rest of the section is
devoted to the proof of that statement.

The paper ends with Section~\ref{sec:intertwiners} which gives details
about $U_q(\sotN)$. We provide explicit
formulas for some morphisms in the category of finite dimensional
$U_q(\sotN)$-modules. Finally we relate the link invariant provided by
the pivotal and braided structure of this category to the invariant
$\evN{\cdot}$ defined in Section~\ref{sec:linkinvariants}.

\subsection*{Acknowledgments}
\label{sec:acknoledgements}
We wish to thank William Ballinger for enlignening
discussion. EB was partially supported by the NSF MSPRF-2202897. LHR
and EW are not supported by any project related funding.

\section{Webs}
\label{sec:webs}
In this paper we will consider some trivalent graphs with additional
combinatorial data which are inspired by the representation theory of
the quantum group $U_q(\so_{2\myN})$.  These graph are called
\emph{$\mathsf{D}$-webs} (or simply \emph{webs}).  In
Section~\ref{sec:webs} and \ref{sec:global-relations}, $\myN$ is an
integer greater than or equal to $1$, in the remaining sections, we
assume $\myN\geq 3$. The definitions provided in
Section~\ref{sec:webs} and \ref{sec:global-relations} make sense for
all positive $\myN$. Some technical points in the remaining sections
(especially Definition~\ref{dfn:gweb} and all constructions that
follows) are easier to phrase in the case $\myN\geq 3$. However, we
believe that the rest of the paper can be adapted for $\myN=1,2$.

\begin{conv}\label{conv:conv}
The symbol $\setN$ denotes the ordered set
  $\{1, \dots, \myN\}$ and $q$ is a formal variable. The set of
  non-negative integers is denoted by $\NN$ and for $k \in \NN$, we set
  $[k]=\frac{q^{k} - q^{-k}}{q-q^{-1}}$, and
  $[2]^{[k]} = \prod_{i=1}^k(q^i+q^{-i})$.
\end{conv}

\begin{dfn}\label{dfn:webs}
  A \emph{type $\mathsf{D}$ web} (or simply \emph{web})  is a trivalent semi-oriented planar
  graph $\web$ with three types of edges:
  \begin{itemize}
  \item \emph{vectorial} unoriented edges, depicted
    $\vcenter{\tikz{\draw[expand style =\stylevect] (0,0) -- (1,0);}}$,
  \item \emph{spin-even} oriented edges, depicted
    $\vcenter{\tikz{\draw[expand style =\styleeven] (0,0) --
        (1,0) \arreven;}}$,
  \item and \emph{spin-odd} oriented edges, depicted
    $\vcenter{\tikz{\draw[expand style =\styleodd] (0,0) -- (1,0)
        \arrodd;}}$.
  \end{itemize}
  The vertices of $\web$ are required to follow one of the four local
  models:
  \begin{equation}\label{eq:local-models-webs}
    \NB
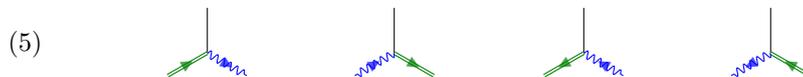
{\tikz[scale=0.6]{}} \qquad\qquad
    \NB{\tikz[scale=0.6]{}} \qquad\qquad
    \NB{\tikz[scale=0.6]{}} \qquad\qquad
    \NB{\tikz[scale=0.6]{}}
  \end{equation}
  Vertex-less loops are allowed and we do not
  require $\web$ to be connected. A \emph{spin} edge is an edge which is either spin-odd or spin-even. The \emph{spin part} of a web $\web$ is the subgraph of $\web$ obtained by forgetting the vectorial edges. It is a disjoint union of edge-bicolored oriented cycles. It is denoted $\webspin$.
\end{dfn}
Note that when one forgets about embedding and orientations of the web
in the plane, the four local models of vertices are identical.

\begin{dfn}\label{dfn:dual-web}
  If $\web$ is a web, the \emph{dual of $\web$} is the web
  $\web^\star$ obtained by reversing the orientations of all spin
  edges and, if $\myN$ is odd, by changing all spin-even edges for
  spin-odd edges and vice-versa. The \emph{opposite of $\web$} is the
  web $\web'$ obtained by reversing the orientation of all spin edges
  (regardless of the parity of $\myN$).
\end{dfn}

\begin{dfn}
  Given a web $\web$, an \emph{$\sotN$-coloring} (or simply coloring)
  of $\web$ is a pair $(o, c)$ with $o$ an orientation of the
  vectorial edges of $\web$ and
  $c:E(\Gamma) \to \mathcal{P}\left({\setN}\right)$ such that:
  \begin{itemize}
  \item if $e$ is a vectorial edge of $\web$, $\#c(e) = 1$,
  \item if $e$ is a spin-even edge of $\web$, $\#c(e)$ is even,
  \item if $e$ is a spin-odd edge of $\web$, $\#c(e)$ is odd,
  \item at each vertex a flow condition (taking into account $o$) is satisfied, namely
    (disregarding the embedding in the plane) the coloring follow one
    of the four local models:
    \[
      \NB
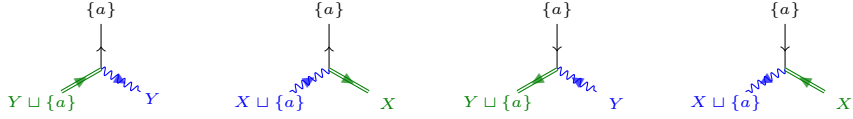
{\tikz[]{\input{\imagesfolder/td-webs-local-col}}}
    \]
    with $a$ an element of $\setN$ and $X$ and $Y$ subsets of $\setN$
    of respectively even and odd cardinality. 
    
    The elements of $\setN$ are \emph{pigments} and for an edge $e$, the set
    $c(e)$ is called the \emph{color of $e$}, so that a color is a set of pigments. The set of colorings of a
    web $\web$ is denoted $\col[\myN]{\web}$. 
  \end{itemize}
\end{dfn}

\begin{figure}[ht]
  \centering
  \NB{\tikz[scale=2.2]{\input{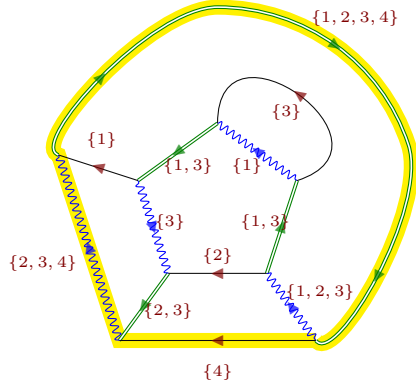}}}
  \caption[Example of a colored web.]{Example of a web endowed with a
    $\sotN$-coloring (for $N\geq 4$). For more readibility, in this
    picture, and only in this one, the coloring (including the
    orientations of vectorial edges) is indicated in wine. The
    $4$-colored cycle is depiced in yellow (see Definition~\ref{dfn:col-cyc}).}
  \label{fig:example-col-web}
\end{figure}

Now that we have colorings, which might be thought of as
\emph{states}, we need to define the degree of such a coloring, so
that we can define the state sum. To this end, we need to consider
various cycles in the webs.

\begin{dfn}\label{dfn:col-cyc}
  Let $\web$ be a web, $(o, c)$ a coloring of $\web$, and $a \in \setN$ a pigment. The \emph{$a$-colored cycle of $(\web, (o,c))$} is
  the union of oriented edges containing $a$ in their color. 
\end{dfn}

After a quick inspection of the local models of colorings, one obtains:
\begin{lem}
  For any colored web $(\web, (o,c))$ and any pigment $a \in \setN$, the
  $a$-colored cycle of $(\web, (o,c))$ is a collection of oriented
  simple closed cycles of the graph underlying $\web$. 
\end{lem}
\begin{dfn} \label{dfn:dual-col}
  Let $(\web, (o,c))$ be a colored web. The \emph{dual of $(o,c)$} is
  the coloring $(o^\star, c^\star)$ of the web $\web^\star$ obtained by  replacing the color of spin edges by their complement in
  $\setN$ while keeping the color and orientation of vectorial edges the same.

  The \emph{opposite of $(o,c)$} is the coloring $(o', c')$ of the web $\web'$
  obtained by keeping the color of spin edges the same while keeping the color and reversing the orientation of the vectorial edges.
\end{dfn}

Once more, a quick inspection of local models gives:
\begin{lem}
  The operation
  $\col[\myN]{\web}\ni(o,c)) \mapsto  (o^\star, c^\star) \in
  \col[\myN]{\web^\star}$  and $\col[\myN]{\web}\ni (o,c) \mapsto (o',
  c') \in \col[\myN]{\Gamma'}$  are 
  well-defined, i.e.{} $(o^\star, c^\star)$ and $(o', c')$ are indeed
  a coloring of   $\web^\star$ and  $\web'$ respectively. 
\end{lem}

Note that the underlying unoriented graphs of $\web^\star$ (resp. $\web'$) is the same as that of
$\web$, so that a cycle in $\web^\star$ (resp. $\web'$) can be seen as a cycle in $\web$.

\begin{lem}
  If $N\geq 2$, then any webs admit a $\sotN$-coloring.
\end{lem}
\begin{rmk}
  One can easily check that the  web \NB{\tikz[scale=0.8]{\input{\imagesfolder/td-ce-coloring}}}  does not admit an $\sotN[2]$-coloring.
  
\end{rmk}

\begin{proof}
  Let $\web$ be a web. It is enough to prove that it admits an $\sotN[4]$ coloring, since any $\sotN$-coloring can be viewed as an $\sotN[2M]$ coloring for $M\geq N$.

  First note that changing the orientations of a spin cycle does not
  change the $\sotN[4]$-colorability, since the first operation
described in Definition~\ref{dfn:dual-col} (and in Definition~\ref{dfn:dual-web}) can be actually performed on each spin cycle independently.

  Color the connected components of $\RR^2\setminus \webspin$ by
  $\{1,2\}$ such that adjacent regions are colored differently. We
  momentarily call these regions \emph{domains}. We may suppose that
  spin cycles are oriented as the boundary of the $1$-colored domains
  (whose orientations are induced by that of $\RR^2$).

  Each domain is partitioned by vectorial edges. We momentarily call these smaller regions \emph{sectors}. We will color sectors $\{\emptyset, \bullet\}$ so that within a given domain $\{\emptyset, \bullet\}$ such that adjacent sectors (in the same domain) are colored differently. There are $2^{\#\textrm{domains}}$ such colorings.
  Starting with the outer domain one can inductively choose a coloring of sectors such that if two sectors share a spin-edge in their boundaries (and hence are in different domain), then:
  \begin{itemize}
  \item If this spin-edge is even, then the two sectors have the same color.
  \item If this spin-edge is odd, then the two sectors have different colors.
  \end{itemize}
There are exactly two such colorings, which corresponds to the two possibilities to color the outer domain.
  
  Now orient vectorial edges as boundary of $\bullet$-colored sectors for those which are in $1$-colored domain and color them by $1$. Orient them as boundary of $\emptyset$-colored sectors for those which are in $2$-colored sectors and color them by $2$. This can be extended to a $\sotN[4]$ coloring as follows: each spin edge is at the interface of two sectors, one in a domain colored by $1$ the other the other one in a domain colored by $2$. Color spin edges by the subset of $\{1,2\}$ which described which adjacent sector is colored by $\{\bullet\}$.

  A local inspection of the situation at vertices shows  that this indeed give a $\sotN[4]$-coloring of $\web$.
  \end{proof}

\begin{dfn}
  Let $\web$ be a web, $(o, c)$ a coloring of $\web$ and $a \in \setN$
  a pigment. The \emph{$a^\star$-colored cycle of $(\web, (o,c))$}
  (resp.{} \emph{$a'$-colored cycle of $(\web', (o',c')$}). 
  is the $a$-colored cycle of $(\web^\star, (o^\star,c^\star))$
  (resp.{} $(\web', (o',c'))$).
\end{dfn}

\begin{rmk}
  Note in particular, on spin edges the orientations of the
  $a^\star$-colored  and $a'$-colored cycles of $(\web, (o,c))$ are opposite to what is  prescribed by $\web$. 
\end{rmk}

Recall that the symmetric difference of two sets $X$ and $Y$, denoted by
$X\Delta Y$, is defined to be $(A\setminus B) \cup (B\setminus A)$, which is the same as $(A\cup B) \setminus (A\cap B)$, i.e. the union without the intersection.

\begin{dfn}
  Let $(\web,(o,c))$ be a colored web and let $a < b$ be two distinct pigments in $\setN$.
  \begin{enumerate}
  \item The \emph{$a'b$-bicolored cycle} is the symmetric difference
    of the $b$-colored cycle of $(\web, (o,c))$ and the $a'$-colored
    cycle of $(\web, (o,c))$. It is denoted by
    $\mymacroone{a}{b}(\web, (o,c))$. (or simply by
    $\mymacroone{a}{b}$ when $\web$ and $(o,c)$ are clear from the
    context).\item The \emph{$a^\star b$-bicolored cycle} is the symmetric difference of
  the $a^\star$-colored cycle (with opposite orientation) of $(\web, (o,c))$ and the $b$-colored cycle
  of $(\web, (o,c))$. It is denoted by $\mymacrotwo{a}{ b}(\web,
  (o,c))$ (or simply by $\mymacrotwo{a}{b}$ when $\web$ and $(o,c)$
  are clear from the context).
\end{enumerate}
If $a>b$, the
\emph{$a'b$-bicolored} cycle is defined to be the $b'a$-bicolored
cycle. Similarly the \emph{$a^\star b$-bicolored} cycle is defined to
be the $b^\star{}a$-bicolored cycle.

We say $a'b$-bicolored cycles are \emph{of the first type}, and
$a^\star b$-bicolored cycles are \emph{of the second type}.  These
notions are illustrated on Figure~\ref{fig:example-bicolored}
\end{dfn}

\begin{rmk}
  Note that $\mymacrotwo{a}{b}(\web, (o,c))$ is more symmetric in
  $a$ and $b$ than $\mymacroone{a}{ b}(\web, (o,c))$: For defining $\mymacrotwo{a}{ b}(\web, (o,c))$
knowing the $a<b$ is not necessary.   \end{rmk}

\begin{figure}[ht]
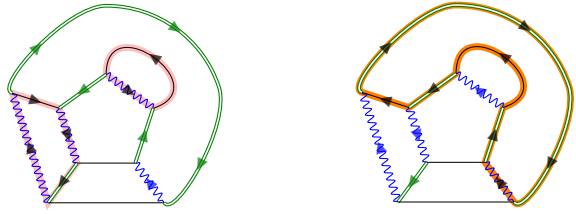

  \centering
  \NB{\tikz[scale=1.3]{\input{\imagesfolder/td-example-1}}}\qquad  \qquad   \NB{\tikz[scale=1.3]{\input{\imagesfolder/td-example-2}}}
  \caption[Example of a bicolored cycles.]{The $1'3$-bicolored (in
    pink, on the
    left) and $1^\star3$-bicolored (in orange, on the right) cycles for the web and the
    coloring given in Figure~\ref{fig:example-col-web}.}
  \label{fig:example-bicolored}
\end{figure}

\begin{dfn}
  The \emph{rotational} of a collection of oriented cycles $\mathcal{C} \subset \mathbb{R}^2$, is
  the number of positively oriented (i.e.{} counter-clockwise) simple cycles in $\mathcal{C}$ minus the number of
  negatively oriented (i.e.{} clockwise) simple cycles. It is denoted $\rotational{\mathcal{C}}$.

  Given a colored web $(\web, (o,c))$, the \emph{degree} of $(o,c)$ is
  given by the following formula:
  \begin{equation}\label{eq:def-degree}
    \deg(o,c) = \sum_{1\leq a < b \leq \myN} \bigg(\rotational{\mymacroone{a}{b}(\web,(o,c))} + \rotational{\mymacrotwo{a}{ b}(\web, (o,c))} \bigg).
  \end{equation}

  The \emph{evaluation} of a web $\web$ is the element of $\NN[q, q^{-1}]$ given by the following formula:
  \begin{equation}
    \label{eq:1}
    \evN{\web} = \sum_{(o,c) \in \col{\web}} q^{\deg(o,c)}. 
  \end{equation}
  In most places, we will make the dependency on $\myN$ implicit and
  simply write $\ev{\web}$ for $\evN{\web}$ (we will keep the
  dependency in Section~\ref{sec:changing-myn}, since we'll relate
  $\evN{\cdot}$ and $\evN[\myN-1]{\cdot}$).
\end{dfn}
\begin{exa}\label{ex:vect}  Let $\web$ be the web consisting of one vectorial circle. A coloring
  of $\web$ consists of an orientation $o$ of the circle and a
  one-element subset $\{a\}$ of $\setN$. Let $b\neq d$ be two pigments in $\setN\setminus\{a\}$. In Table~\ref{table:deg-bic-vectorial-circle}, we gather the degree of the bicolored cycles of $\web$ colored with $(o,\{a\})$. 
  \begin{table}[h!]
    \centering
    \begin{tabular}{|c|c|c|c|c|}
      \hline
                                       & $o=\circlearrowleft, a<b$ & $o=\circlearrowleft, a>b$ & $o=\circlearrowright, a<b$ & $o=\circlearrowright, a>b$ \\ \hline
      $\rotational{\mymacroone{a}{b}}$            & $-1$                        & $+1$                        & $+1$                         & $-1$                         \\ \hline
      $\rotational{\mymacrotwo{a}{b}}$ & $+1$                        & $+1$                        & $-1$                         & $-1$                         \\ \hline
      $\rotational{\mymacroone{b}{d}}$            & 0                         & 0                         & 0                          & 0                        \\ \hline
      $\rotational{\mymacrotwo{b}{d}}$ & 0                         & 0                         & 0                          & 0                   \\ \hline
    \end{tabular}
    \caption{Degrees of bicolored cycles on the vectorial circle web colored with $(o, \{a\})$.}\label{table:deg-bic-vectorial-circle}
  \end{table}
  From Table~\ref{table:deg-bic-vectorial-circle}, it follows that:\[
    \deg((\circlearrowleft, \{a\}))= 2a-2 \qquad \textrm{and} \qquad
    \deg((\circlearrowright, \{a\}))= -2a+2. 
  \]
  So that $\ev{\web}= [N] (q^{N-1} + q^{1-N}) =[2N-1] + 1$.
\end{exa}
\begin{exa}
\label{ex:even} Let us now inspect the case where $\web$ is a
  positively oriented even-spin circle. Colorings of $\Gamma$ are in
  canonical one-to-one correspondence with even\footnote{An
    \emph{even} (resp.{} \emph{odd}) set, is a set with even
    (resp. odd) cardinality.} subsets of $\setN$.
  Let us denote by
  $\myor$ the orientation given to $\Gamma$: the orientation part
  of each coloring is always equal to $\myor$.
Let $a<b$ be two
  distinct pigments in $\setN$. We gather the degree of bicolored
  cycles of $\web$ colored by $X\subseteq \setN$ in Table
  \ref{table:deg-bic-circle}.

  \begin{table}[h!]
    \centering
    \begin{tabular}{|c|c|c|c|c|}
      \hline
                                       & $a\notin X, b\notin X $ & $a\in X, b \notin X$ & $a\notin X, b \in X$ & $a\in X, b\in X$ \\ \hline
      $\rotational{\mymacroone{a}{ b}}$            & $0$                     & $-1$                 & $+1$                 & $0$              \\ \hline
      $\rotational{\mymacrotwo{a}{ b}}$ & $-1$                    & $0$                  & $0$                  & $+1$             \\ \hline
    \end{tabular}
    \caption{Degrees of bicolored cycles on the positively oriented even-spin circle web colored with $(o, X)$.}\label{table:deg-bic-circle}
  \end{table}
  Now, since $a<b$, we have:
  \begin{equation}
\rotational{\mymacroone{a}{ b}} + \rotational{\mymacrotwo{a}{b}} =\begin{cases} +1 & \textrm{if $b \in X$,} \\ -1 &\textrm{if $b\notin X$.}\end{cases}
    \end{equation}
    so that
    \begin{equation}
      \label{eq:2}
      \sum_{a=1}^{b-1}(\rotational{\mymacroone{a}{ b}} + \rotational{\mymacrotwo{a}{b}}) = \begin{cases} b-1 & \textrm{if $b \in X$,} \\ 1-b &\textrm{if $b\notin X$.}\end{cases}
    \end{equation}
    From equation \eqref{eq:def-degree}, we have
    \begin{equation}
        \deg(\myor, X) = \sum_{b=1}^N\sum_{a=1}^{b-1}(\rotational{\mymacroone{a}{ b}} + \rotational{\mymacrotwo{a}{b}}) =\sum_{\substack{1\le b\le N \\ b\in X}}(b-1) + \sum_{\substack{1\le b\le N \\ b\notin X}} (1-b).
      \end{equation}
    Since we know the parity of $X$, we can completely describe $X$ by its intersection with $\{2,\dots, \myN \}$, hence we have:
  \begin{equation}
    \ev{\Gamma} = \prod_{b=2}^\myN (q^{b-1}+ q^{1-b}) = [2]^{[N-1]}.
  \end{equation}
  Here, as decided in Convention~\ref{conv:conv}, $[2]^{[N-1]} := \prod_{i=1}^{N-1}(q^i+q^{-i})$.
\end{exa}
\begin{exa}
\label{ex:odd} For a positively oriented odd-spin circle, the previous argument is almost the same ($X$ being odd and not even). Negatively oriented spin circles are treated analogously.
\end{exa}
\begin{exa}
\label{ex:theta} Consider the theta web
  \[
    \web:= \NB
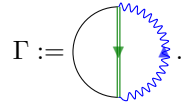
{\tikz[]{\begin{scope}[scale=0.6]
    \draw[expand style =\stylevect] (90:1) arc (90:270:1);
  \draw[expand style =\styleeven] (90:1) -- (270:1) \arreven;
  \draw[expand style =\styleodd] (90:1) arc (90:-90:1) \rarrodd;
\end{scope}}}. 
  \]
  Colorings of $\web$ can be partitioned into two subsets: the ones for which the vectorial edge is upward-oriented and the ones for which it is downward-oriented. We'll first focus on the ones which are downward-oriented. Such colorings are in a canonical one-to-one correspondence with pairs $(X, a)$, where $a\in \setN$ is a pigment and $X$ is an even subset of $\setN$ not containing $a$. This correspondence is made explicit in Figure~\ref{fig:col-theta}.
  \begin{figure}[h!]
    \centering
    \NB{\tikz[]{\input{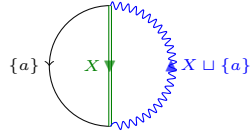}}}
    \caption{Coloring of $\Gamma$}\label{fig:col-theta}
  \end{figure}
  Bicolored cycles involving $a$ will never be empty. With a similar analysis to that for the even-spin circle, one obtains that:
  \begin{align}
    \deg(X,a) &= N-1 + ((a-1) -(N-a)) + \sum_{b \in X, b<a} (b-1) \\ & \quad + \sum_{b \in X, b>a} (b-2)  + \sum_{b \notin X, b<a} (1-b) + \sum_{b \in X, b>a} (2-b).
  \end{align}
  So if one denotes by $\col{\Gamma}_{\uparrow a}$, the subset of colorings with the vectorial edge upward-oriented and colored by $a$, one has:
  \begin{equation}
    \label{eq:4}
    \sum_{(X,a) \in \col{\Gamma}_{\uparrow a}} q^{\deg(X,a)}= q^{2a-2}\prod_{b=1}^{N-2}(q^b + q^{-b}). 
  \end{equation}
  Similarly, one has (extending the notation in the obvious manner):
\begin{equation}
  \sum_{(X,a) \in \col{\Gamma}_{\downarrow a}} q^{\deg(X,a)}= q^{2-2a}\prod_{b=1}^{N-2}(q^b + q^{-b}). 
  \end{equation}
  Summing all colorings together, one obtains: \begin{align}
    \label{eq:5}
    \ev{\web}&=([2N-1] +1)\prod_{b=1}^{N-2}(q^b + q^{-b})\\ &=([2N-1] +1)[2]^{[N-2]} = [N][2]^{[N-1]}.
  \end{align}
\end{exa}

\section{Global relations}
\label{sec:global-relations}
This section investigates certain global relations satisfied by the web
evaluation. Section~\ref{sec:chang-parity-orient} leads to a proof that $\ev{\web}$ is
symmetric in $q$ and $q^{-1}$ for every web
$\web$. Section~\ref{sec:changing-myn} relates $\evN{\cdot}$ and
$\evN[N-1]{\cdot}$. We expect that the relation between web evaluation for different values of $\myN$ may be further developed to deduce an inductive formula to evaluate $\evN{\cdot}$.

This section is somewhat independent of the rest of the paper. In particular, the global relations are not used to
prove the skein relations in Section~\ref{sec:skein} or to deduce invariance under Reidemeister moves in Section~\ref{sec:linkinvariants}.

\subsection{Changing parities and orientations}
\label{sec:chang-parity-orient}

\begin{lem}
Let $\web$ be a web and $\web^\star$ its dual (see Defintion~\ref{dfn:dual-web}), then
\[\ev{\web} = \ev{\web^\star}.\]
\end{lem}

\begin{proof}This follows from the fact for any coloring $(o,c)$ of $\web$ and
  any pigments $a<b$ in $\setN$, the $a'b$-bicolored cycle
  (resp.{} $a^\star b$-bicolored cycle) of $(\web, (o,c))$ is equal to
  the $a'b$-bicolored cycle (resp.{} $a^\star b$-bicolored cycle) of $(\web^\star, (o^\star,c^\star))$, hence the colorings $(o,c)$ and $(o^\star,c^\star)$ have the same degree.
\end{proof}

\begin{lem}
Let $\web$ be a web and  $\web'$ its opposite, then  
\[\ev{\web} = \ev{\web'}_{q\mapsto q^{-1}}.\]
\end{lem}

\begin{proof}
  To any coloring $(o,c)$ of $\web$, one associates the coloring
  $(o',c')$ of $\web'$.  For any pigments $a<b \in \setN$, the
  $a'b$-bicolored cycle (resp.{} $a^\star b$-bicolored cycle) of
  $(\web, (o,c))$ is equal to $a'b$-bicolored cycle (resp.{}
  $a^\star b$-bicolored cycle) of $(\web', (o',c'))$ with reversed orientations.
\end{proof}

\begin{cor}
  If $N$ is even, $\ev{\web}$ is symmetric under $q \leftrightarrow q^{-1}$. 
\end{cor}

\begin{proof}
  If $N$ is even, $\web^\star = \web'$, so that $\ev{\web}= \ev{\web}_{q\mapsto q^{-1}}.$
\end{proof}

\begin{cor}
  If $N$ is odd, and $\overline{\web}$ is the web obtained from $\web$ by swapping odd and even spin-edges (but keep their orientations), then $\ev{\web} = \ev{\overline{\web}}_{q\mapsto q^{-1}}$.
\end{cor}
\begin{proof}
  If $N$ is odd, $\overline{\web} = (\web^\star)'$.
\end{proof}

\begin{lem}
  For any web, $\ev{\web} = \ev{\overline{\web}}$. 
\end{lem}

\begin{proof}
  Let $(o,c)$ be a coloring of $\web$. Let us consider the coloring $(\overline{o}, \overline{c})$ of $\overline{\web}$ obtained as follow: change the orientations of all the vectorial edges colored by $\{1\}$ and for each spin edge, replace its color by its symmetric difference with $\{1\}$. Note that this changes the parities of the colors and that $(\overline{o}, \overline{c})$ is indeed a coloring of $\overline{\web}$.
  This operation is of course involutive and therefore induces a bijection between $\col{\web}$ and $\col{\overline{\web}}$. 
  Let $a<b$ be two pigments in $\setN$. If $a\neq 1$, the
  $a'b$-bicolored cycle (resp.{} $a^\star b$-bicolored cycle) of
  $(\web, (o,c))$ is equal to the $a'b$-bicolored cycle (resp.{}
  $a^\star b$-bicolored cycle) of $(\overline{w}, (\overline{o},
  \overline{c}))$. If $a=1$, $1'b$-bicolored cycle (resp.{} $1^\star
  b$-bicolored cycle) is equal to the $1^\star b$-bicolored cycle (resp.{}
  $1'b$-bicolored cycle) of $(\overline{w}, (\overline{o}, \overline{c}))$. So that $\deg(\overline{o}, \overline{c}) = \deg(o,c)$.
\end{proof}

\begin{cor}\label{cor:ev-all-equal}
  For all webs $\web$,  $\ev{\web}$ is invariant under $q\mapsto q^{-1}$, and
  \begin{equation} \ev{\web} = \ev{\web^\star} = \ev{\web'} = \ev{\overline{\web}}.\end{equation}
\end{cor}

\begin{rmk}
  Corollary~\ref{cor:ev-all-equal}, says as far as $\ev{\cdot}$ is concerned, we can globaly reverse orientation and swap parity of spin edges. In particular this will reduce the amount of cases to consider when proving skein relations. 
  
  Corollary~\ref{cor:ev-all-equal} does not imply that we can perform this operation locally. However, we do believe that locally reversing orientation and swapping parity of spin edges leaves $\ev{\cdot}$ invariant. See Lemma~\ref{lem:singular-square} for an example. But we do not have a proof of such a statement in general.
\end{rmk}

\subsection{Changing $\myN$}
\label{sec:changing-myn}

In this section, we aim to express $\ev{\web}=\evN{\web}$ as a linear combination of $\evN[\myN-1]{\web^\cycle}$ where $\web_\cycle$ are webs constructed from $\web$. Heuristically, this should reflect the branching rules of $\sotN$ to $\sotN[2\myN-2]$.  Before stating a precise result, we need a couple of definitions and notations.

\begin{dfn}
  Let $\web$ be a web. An \emph{oriented cycle} in $\web$ is a union (maybe empty) of disjoint simple cycles in $\web$ endowed with an orientation which is compatible with the orientation (given by $\web$) of the spin edges it contains.

  If $\cycle$ is an oriented cycle in $\web$, then $\web^\cycle$ is obtained from $\web$ be removing all vectorial edges contained in $\cycle$ and changing the parity (but keeping the orientations) of all spin edges contained in $\cycle$.

  The \emph{global $\myN$-shift} of a web $\web$ is the integer denoted $\gshiftN{\web}$ given by the following formula which counts local configurations in $\web$:
  \begin{align}
    \begin{split}
    \label{eq:dfn-global-shift}
    \gshiftN{\web} =& -(N-1)\rotational{\webspin}\\ &+\frac{1}{2}\left(
                        + \#\left\{ \NB{\tikz[]{\begin{scope}[scale=0.6]
  \draw[expand style =\stylevect] (0,0) -- (0.7,0);
  \draw[expand style =\styleeven] (0,0) -- (0,1) \arreven;
  \draw[expand style =\styleodd]  (0,0) -- (0,-1) \rarrodd;
\end{scope}}} \right\} + \#\left\{\NB{\tikz[yscale=-1]{}} \right\}
                        - \#\left\{ \NB{\tikz[xscale=-1]{}} \right\} - \#\left\{\NB{\tikz[scale=-1]{}} \right\} \right),
                      \end{split}
  \end{align}
    where $\rotational{\webspin}$ is the rotational of the collection of oriented cycles given by spin edges of $\web$. The other terms are a weighted count of the vertices of $\web$, counting $+\frac12$ (resp.{} $-\frac12$) when the vectorial edge is on the right-and side (resp.{} left-hand side) of the spins.

  The \emph{$\myN$-shift} of a cycle $\cycle$ in $\web$ is the integer denoted $\shiftN{\cycle}$ given by the following formula which counts local configurations of $\cycle$: 
  \begin{align}
  \begin{split}
    \label{eq:dfn-shift}
    \shiftN{\cycle} =  &2(N-1)\rotational{\cycle} 
    - \#\left\{ \NB{\tikz[]{\begin{scope}[scale=0.6]
  \draw[expand style =\stylecycle] (0,-1) -- (0,1);% \arrcycle;
  \draw[expand style =\stylevect] (0,0) -- (0.7,0);
  \draw[expand style =\styleeven] (0,0) -- (0,1) \arreven;
  \draw[expand style =\styleodd]  (0,0) -- (0,-1) \rarrodd;
\end{scope}}} \right\} + \#\left\{\NB{\tikz[yscale=-1]{}} \right\}
    - \#\left\{ \NB{\tikz[xscale=-1]{}} \right\} - \#\left\{\NB{\tikz[scale=-1]{}} \right\} \\
    & \quad+ \frac{N-2}{2}\left( \#\left\{ \NB{\tikz[]{\begin{scope}[scale=0.6]
  \draw[expand style =\stylecycle] (0,-1) -- (0,0) -- (0.7,0);% \arrcycle;
  \draw[expand style =\stylevect] (0,0) -- (0.7,0);
  \draw[expand style =\styleeven] (0,0) -- (0,1) \arreven;
  \draw[expand style =\styleodd]  (0,0) -- (0,-1) \rarrodd;
\end{scope}}} \right\} + \#\left\{\NB{\tikz[yscale=-1]{\begin{scope}[scale=0.6]
  \draw[expand style =\stylecycle] (0,-1) -- (0,0) -- (0.7,0);% \arrcycle;
  \draw[expand style =\stylevect] (0,0) -- (0.7,0);
  \draw[expand style =\styleeven] (0,0) -- (0,-1) \rarreven;
  \draw[expand style =\styleodd]  (0,0) -- (0,1) \arrodd;
\end{scope}}} \right\} + \#\left\{\NB{\tikz[]{}} \right\}
      + \#\left\{ \NB{\tikz[yscale=-1]{}} \right\}  \right.  \\
    &  \qquad \qquad \left. - \#\left\{ \NB{\tikz[xscale=-1]{}} \right\}  - \#\left\{\NB{\tikz[scale=-1]{}} \right\} - \#\left\{\NB{\tikz[xscale=-1]{}} \right\}
      - \#\left\{ \NB{\tikz[scale=-1]{}} \right\} \right),      
  \end{split}
  \end{align}
  where the cycle $\cycle$ is represented by $\NB{\tikz[scale=0.8]{\draw[expand style =\stylecycle] (0,0) -- (1,0);}}$. The second and third lines count algebraically (with a factor $\frac{N-2}{2}$) the number of right-angle turns in the cycle when the web is drawn so that vectorial edges are orthogonal to spin ones at vertices. 
\end{dfn}

\begin{prop}  \label{prop:BR-typeD}
  Suppose that $N\geq 2$. Then for any web $\web$, one has:
  \begin{equation}
  \evN[\myN]{\web}= q^{\gshiftN{\web}}\sum_{\substack{\text{$\cycle$ oriented} \\ \text{cycle in $\web$}}}q^{\shiftN{\cycle}}\evN[N-1]{\web^\cycle}. \label{eq:branching-rule-TD}
\end{equation}
\end{prop}

\begin{proof}
  Let $(o,c)$ be an $\sotN$-coloring of $\web$. The union of edges
  containing $N$ is their colors is an oriented cycle in $\web$,
  denote it $\cycle(o,c)$. The $\sotN$-coloring $(o,c)$ induces a
  $\sotN[2(N-1)]$-coloring of $\web^{\cycle(o,c)}$. This construction
  induces a one-to-one correspondence between the set of $\sotN$-colorings of $\web$ and the union of the sets of $\sotN[2N-2]$-colorings of $\web^\cycle$ for $\cycle$ running through oriented cycles in $\web$. In order to prove \eqref{eq:branching-rule-TD}, one needs to understand how this correspondence behaves with respect to the degrees of colorings.

  Let $(o,c)$ be a $\sotN$-coloring of $\web$ and let $(o',c')$ be the corresponding $\sotN[2(N-1)]$ coloring. Following the definition of degrees, one has:
  \begin{equation}
    \label{eq:3}
    \deg{(o,c)}= \deg{(o',c')} + \sum_{1\leq a < \myN} \left(\deg(\mymacroone{a}{\myN}(\web,(o,c)) + \deg(\mymacrotwo{a}{ \myN}(\web, (o,c)) \right).
  \end{equation}
  The aim of what follows is to show that
  \begin{equation}
    \label{eq:6}
 \gshiftN{\web} + \shiftN{\web} =\sum_{1\leq a < \myN} \left(\deg(\mymacroone{a}{\myN}(\web,(o,c))) + \deg(\mymacrotwo{a}{ \myN}(\web, (o,c))) \right)
  \end{equation}
  from which \eqref{eq:branching-rule-TD} follows.
Let $a \in \setN[\myN-1]$ be a pigment, and let us compute
$D:=\deg(\mymacroone{a}{ \myN}(\web,(o,c))) + \deg(\mymacrotwo{a}{ \myN}(\web, (o,c)))$. For computing rotational of planar curves we may use (normalized) curvature. However since webs are not smooth at vertices we should add a combinatorial contribution to curvature at corners: at a vertex $v$ the curve has two unit tangent vectors, the combinatorial (normalized) curvature is the directed angle between these two vectors (the direction being given by the orientation) divided by $2\pi$.  We may suppose that the web $\web$ is embedded so that at vertices the union of spin edges forms a smooth closed 1-manifold and that at vertices vectorial edges are orthogonal to this 1-manifold, i.e.{} vertices look locally as follows:
\begin{equation}
  \label{eq:orthogonal-models}
  \NB{\tikz[]{}} \qquad \qquad \NB{\tikz[yscale=-1]{}} \qquad \qquad \NB{\tikz[xscale=-1]{}} \qquad \qquad \NB{\tikz[scale=-1]{}}
\end{equation}

With these conditions, one has:
\begin{align}
\begin{split}
\gshiftN{\web} + \shiftN{\cycle} = & -(N-1)\rotational{\webspin} + 2(N-1)\rotational[\mathrm{smooth}]{\cycle}  \\
&+\frac{1}{2}\left(\#\left\{ \NB{\tikz[]{}} \right\} + \#\left\{\NB{\tikz[yscale=-1]{}} \right\} - \#\left\{ \NB{\tikz[xscale=-1]{}} \right\} - \#\left\{\NB{\tikz[scale=-1]{}} \right\} \right) \\
&-\frac{1}{2}\left(\#\left\{ \NB{\tikz[]{}} \right\} + \#\left\{\NB{\tikz[yscale=-1]{}} \right\} - \#\left\{ \NB{\tikz[xscale=-1]{}} \right\} - \#\left\{\NB{\tikz[scale=-1]{}} \right\} \right),
\end{split}\label{eq:gshsh}
\end{align}
where:
\begin{itemize}
\item The diagrams on the first line have slightly different meaning than in \eqref{eq:dfn-global-shift}. Here they represent vertices which are away from the cycle $\cycle$, while in \eqref{eq:dfn-global-shift} they represented all vertices in $\web$. 
\item The quantity $\rotational[\mathrm{smooth}]{\cycle}$ is the sum of the contribution of the (normalized) curvature of the edges belonging to $\cycle$. 
\end{itemize}
This follows from the fact that the quantity
\begin{equation}
  \begin{split}
\rotational[\mathrm{comb.}]{\cycle} = \frac{1}{4} =  &\#\left\{ \NB{\tikz[]{}} \right\} + \#\left\{\NB{\tikz[yscale=-1]{}} \right\} + \#\left\{\NB{\tikz[]{}} \right\}
      + \#\left\{ \NB{\tikz[yscale=-1]{}} \right\}   \\
    &  \quad - \#\left\{ \NB{\tikz[xscale=-1]{}} \right\}  - \#\left\{\NB{\tikz[scale=-1]{}} \right\} - \#\left\{\NB{\tikz[xscale=-1]{}} \right\}
      - \#\left\{ \NB{\tikz[scale=-1]{}} \right\}
  \end{split}
  \label{eq:rhocomb}
\end{equation}
with $\rotational{\cycle}=\rotational[\mathrm{comb.}]{\cycle} + \rotational[\mathrm{smooth}]{\cycle}$.

Let us write $D= D_{\mathrm{smooth}} + D_{\mathrm{comb.}}$ with:
\begin{align}
  D_{\mathrm{smooth}} &=
 \sum_{1\leq a < \myN} \rotational[\mathrm{smooth}]{\mymacroone{a }{ \myN}(\web, (o,c))}  +
                        \rotational[\mathrm{smooth}]{\mymacrotwo{a}{
     \myN}(\web, (o,c))}
\end{align}
and \begin{align}
D_{\mathrm{comb.}} &=
  \sum_{1\leq a < \myN} \rotational[\mathrm{comb.}]{\mymacroone{a }{ \myN}(\web, (o,c))}  +
                        \rotational[\mathrm{comb.}]{\mymacrotwo{a}{ \myN}(\web, (o,c))}
\end{align}

Let $e$ be a vectorial edge. It is colored and oriented by $(o,c)$. We will distinguish three cases: $c(e) = \{\myN\}$, $c(e)=\{a\}$ or $c(e)=\{b\}$ for $b\notin \{a , \myN\}$.
\begin{itemize}
\item If $c(e) =\{\myN\}$, the edge $e$ appears in both $\mymacroone{a }{ \myN}(\web, (o,c))$ and $\mymacrotwo{a}{ \myN}(\web, (o,c))$ with orientation given by $o$. 
\item If $c(e) =\{a\}$, the edge $e$ appears in
  $\mymacrotwo{a}{ \myN}(\web, (o,c))$ with orientation given by
  $o$ and in $\mymacroone{a }{ \myN}(\web, (o,c))$  with the opposite orientation.
\item If $c(e) =\{b\}$, the edge $e$ appears neither in  $\mymacroone{a}{ \myN}(\web, (o,c))$ nor in $\mymacrotwo{a}{ \myN}(\web, (o,c))$.
\end{itemize}
Hence the curvature of $e$ contributes to $D_{\mathrm{smooth}}$ if and only if $c(e)= \{\myN\}$ and in this case, it contributes with twice its (normalized) curvature computed with the orientation given by $o$.

Now let $e$ be a spin edge. It is colored by $c$. Independently, both $a$ and $\myN$ can belong to $c(e)$, hence we have four cases to inspect:
\begin{itemize}
\item If both $a$ and $\myN$ belong to $c(e)$, then $e$ appears in $\mymacrotwo{a}{ \myN}(\web, (o,c))$ with its natural orientation and does not appear in $\mymacroone{a}{ \myN}(\web, (o,c))$.
\item If $\myN$ belongs to $c(e)$ but not $a$, then $e$ appears in $\mymacroone{a}{ \myN}(\web, (o,c))$ with its natural orientation and does not appear in $\mymacrotwo{a}{ \myN}(\web, (o,c))$.
\item If $a$ belongs to $c(e)$ but not $a$, then $e$ appears in $\mymacroone{a}{ \myN}(\web, (o,c))$ with its opposite orientation and does not appear in $\mymacrotwo{a}{ \myN}(\web, (o,c))$.
\item If neither $a$ nor $\myN$ belong to $c(e)$, then $e$ appears in $\mymacrotwo{a}{ \myN}(\web, (o,c))$ with its opposite orientation and does not appear in $\mymacroone{a}{ \myN}(\web, (o,c))$.
\end{itemize}
Finally, the curvature or $e$ contributes to $D_{\mathrm{smooth}}$ with its natural orientation if $\myN \in c(e)$ and with its opposite orientation otherwise.

Since $a$ runs between $1$ and $\myN-1$, one has:
\begin{equation}
  \label{eq:6}
  D_{\mathrm{smooth}}= -(N-1)\rotational{\webspin} + 2(N-1)\rotational[\mathrm{smooth}]{\cycle}.
\end{equation}

We now inspect the contribution of the combinatorial curvature at a
vertex $v$ to $D_{\mathrm{comb.}}$. Forgetting about the embedding in
the plane and the orientation of spin edges, there are four different
local configuration for $\cycle(c)$ at a vertex:
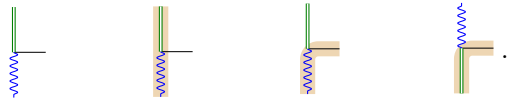
\[
  \NB{\tikz[]{\begin{scope}[scale=0.6]
  \draw[expand style =\stylevect] (0,0) -- (0.7,0);
  \draw[expand style =\styleeven] (0,0) -- (0,1); %\arreven;
  \draw[expand style =\styleodd]  (0,0) -- (0,-1); %\rarrodd;
\end{scope}}} \qquad \qquad \NB{\tikz[]{\begin{scope}[scale=0.6]
  \draw[expand style =\stylecycle] (0,-1) -- (0,1);% \arrcycle;
  \draw[expand style =\stylevect] (0,0) -- (0.7,0);
  \draw[expand style =\styleeven] (0,0) -- (0,1); %\arreven;
  \draw[expand style =\styleodd]  (0,0) -- (0,-1);% \rarrodd;
\end{scope}}} \qquad \qquad
  \NB{\tikz[]{\begin{scope}[scale=0.6]
  \draw[expand style =\stylecycle] (0,-1) -- (0,0) -- (0.7,0);% \arrcycle;
  \draw[expand style =\stylevect] (0,0) -- (0.7,0);
  \draw[expand style =\styleeven] (0,0) -- (0,1);% \arreven;
  \draw[expand style =\styleodd]  (0,0) -- (0,-1);% \rarrodd;
\end{scope}}} \qquad \qquad\NB{\tikz[]{\begin{scope}[scale=0.6]
  \draw[expand style =\stylecycle] (0,-1) -- (0,0) -- (0.7,0);% \arrcycle;
  \draw[expand style =\stylevect] (0,0) -- (0.7,0);
  \draw[expand style =\styleeven] (0,0) -- (0,-1); %\rarreven;
  \draw[expand style =\styleodd]  (0,0) -- (0,1); %\arrodd;
\end{scope}}}.
\]
The two last models are actually similar since in our argument the parity of spin edges will not play any role.

Suppose first that the vertex $v$ is  of type \NB{\tikz[]{}}. If the vectorial edge is colored by $\{b\}$ with $b\neq a$, then there is no contribution of $v$ to $D_{\mathrm{comb.}}$ since both $\mymacroone{a}{ \myN}(\web, (o,c))$ and $\mymacrotwo{a}{ \myN}(\web, (o,c))$ are locally smooth. If it is colored by $\{a\}$, then both  $\mymacroone{a}{ \myN}(\web, (o,c))$ and $\mymacrotwo{a}{ \myN}(\web, (o,c))$ have a right-angle at $v$ (both with normalized combinatorial curvature equal to $\pm \frac14$) and the contribution of $v$ to $D_{\mathrm{comb.}}$ depends on the orientation and is summarized in Table~\ref{tab:BR1}.

  \renewcommand{\arraystretch}{2.5}
\begin{table}[ht]
  \centering
  \begin{tabular}{|p{2.7cm}|c|c|c|c|}\hline
    orientation& \NB{\tikz[]{\input{\imagesfolder/td-sh0}}}  & \NB{\tikz[yscale=-1]{\input{\imagesfolder/td-sh0}}} & \NB{\tikz[xscale=-1]{\input{\imagesfolder/td-sh0}}} &\NB{\tikz[scale=-1]{\input{\imagesfolder/td-sh0}}} \\ \hline
    contribution to $D_{\mathrm{comb.}}$ &$\displaystyle{\frac12}$& $\displaystyle{\frac12}$ &$\displaystyle{-\frac12}$&$\displaystyle{-\frac12}$\\ \hline
  \end{tabular}
  \label{tab:BR1}
\end{table}

If the vertex  $v$ is of type \NB{\tikz[]{}}, the situation is quite similar:  If the vectorial edge is colored by $\{b\}$ with $b\neq a$, then there is no contribution of $v$ to $D_{\mathrm{comb.}}$ since both $\mymacroone{a}{ \myN}(\web, (o,c))$ and $\mymacrotwo{a}{ \myN}(\web, (o,c))$ are locally smooth. If it is colored by $\{a\}$, then both  $\mymacroone{a}{ \myN}(\web, (o,c))$ and $\mymacrotwo{a}{ \myN}(\web, (o,c))$ have a right-angle at $v$ (both with normalized combinatorial curvature equal to $\pm \frac14$) and the contribution of $v$ to $D_{\mathrm{comb.}}$ depends on the orientation and is summarized in Table~\ref{tab:BR2}.

\begin{table}[ht]
  \centering
  \begin{tabular}{|p{2.7cm}|c|c|c|c|}\hline
    orientation and embedding & \NB{\tikz[]{\input{\imagesfolder/td-sh1}}}  & \NB{\tikz[yscale=-1]{\input{\imagesfolder/td-sh1}}} & \NB{\tikz[xscale=-1]{\input{\imagesfolder/td-sh1}}} &\NB{\tikz[scale=-1]{\input{\imagesfolder/td-sh1}}} \\ \hline
    contribution to $D_{\mathrm{comb.}}$ &$\displaystyle{-\frac12}$& $\displaystyle{-\frac12}$ &$\displaystyle{\frac12}$&$\displaystyle{\frac12}$\\ \hline
  \end{tabular}
  \label{tab:BR2}
\end{table}

If the vertex $v$ is of type \NB{\tikz[]{}} or \NB{\tikz[]{}}, then both $\mymacroone{a}{ \myN}(\web, (o,c))$ and $\mymacrotwo{a}{ \myN}(\web, (o,c))$ have a right-angle at $v$, however they are of opposite signs so that the contribution of $v$ to $D_{\mathrm{comb.}}$ is $0$.

Finally we have:
\begin{equation}\label{eq:Dcomb}
  \begin{split}
    D_{\mathrm{comb.}} = &\frac{1}{2}\left( + \#\left\{ \NB{\tikz[]{}} \right\} + \#\left\{\NB{\tikz[yscale=-1]{}} \right\} - \#\left\{ \NB{\tikz[xscale=-1]{}} \right\} - \#\left\{\NB{\tikz[scale=-1]{}} \right\} \right) \\
&-\frac{1}{2}\left( + \#\left\{ \NB{\tikz[]{}} \right\} + \#\left\{\NB{\tikz[yscale=-1]{}} \right\} - \#\left\{ \NB{\tikz[xscale=-1]{}} \right\} - \#\left\{\NB{\tikz[scale=-1]{}} \right\} \right).
  \end{split}
\end{equation}
and it follows from \eqref{eq:rhocomb}, \eqref{eq:gshsh} and \eqref{eq:Dcomb}
\begin{equation}
  \label{eq:7}
  D= D_{\mathrm{comb.}} + D_{\mathrm{smooth}} = \gshiftN{\web} + \shiftN{\cycle}. \qedhere
\end{equation}
\end{proof}

\subsection{Relation with type $\texorpdfstring{\mathsf{A}}{A}$}
\label{sec:relation-with-type}

In this section, we aim to express $\ev{\web}=\evN{\web}$ as a linear
combination of $\evA{\web^\ell}$, where $\web^\ell$ are webs of type
$\mathsf{A}$ and $\evA{\cdot}$ is the evaluation of such
webs. Heuristically, this should correspond to the
branching rules of $\mathfrak{sl}_N\subset \sotN$. The restriction of the vector representation of $\sotN$ decomposes as the direct sum of the vector representation of $\mathfrak{sl}_N$ and its dual, while the even/odd spin representations decompose as the direct sum of even/odd exterior powers of the vector representation of $\mathfrak{sl}_N$.
We will not make this precise, but these branching rules are exactly  the motivation for the combinatorial colorings introduced earlier. For instance, the choice of orientations for the vectorial unoriented edges.

  Before stating a
precise result, we need a pair of definitions and notations.

\begin{dfn}
  An \emph{$\A$-labeling} of a web $\web$ is a pair $\ell = (\ell_o, \ell_t)$, where $\ell_o$ gives an orientation to vectorial edges of $\web$ and $\ell_t$ associated a non-negative integers to each spin edge of $\web$, respecting their parities, and such that the flow is preserved at each vertex.

  Let $\ell$ be an $\A$-labeling of a $\web$ and $v$ be a vertex of $\web$. Define the $\A$-shift of $(\ell, v)$ to be the integer $\Ashift{\ell,v}$ given the following rules depending on the local configuration\footnote{In \eqref{eq:Ashift-begin}--\eqref{eq:Ashift-end}, in the column on the left-hand side, $k$ is an even non-negative integer, and $k$ is an odd non-negative integer in the column on the right-hand side.}:
  \begin{align}\label{eq:Ashift-begin}
    \Ashift{\ttkzAshift[xscale=-1]{->}{left}{k}{k+1}} &=\frac{ k}{2}, &
    \Ashift{\ttkzAshift[xscale=-1]{-<-}{left}{k+1}{k}} &=\frac{ k}{2}, \\
    \Ashift{\ttkzAshift[scale=-1]{-<-}{left}{k}{k+1}} &=\frac{ k}{2}, &
    \Ashift{\ttkzAshift[scale=-1]{->}{left}{k+1}{k}} &=\frac{ k}{2}  \\
    \Ashift{\ttkzAshift[]{->}{right}{k}{k+1}} &=\frac{ -k}{2}, & 
    \Ashift{\ttkzAshift[]{-<-}{right}{k+1}{k}} &=\frac{ -k}{2}, \\
    \Ashift{\ttkzAshift[yscale=-1]{-<-}{right}{k}{k+1}} &=\frac{ -k}{2}, &
    \Ashift{\ttkzAshift[yscale=-1]{->}{right}{k+1}{k}} &=\frac{ -k}{2}. \label{eq:Ashift-end}
  \end{align}
  Finally define the $\A$-shift of $\ell$ to be the integer $\Ashift{\ell}$ defined by:
  \begin{equation}
    \label{eq:10}
    \Ashift{\ell} =  \sum_{v\in V(\web)}\Ashift{\ell,v}.
  \end{equation}
\end{dfn}
If $\web$ is a web and $\ell$ is a labeling of $\web$, then $\web^\ell$
denotes the web $\web$ endowed with the extra combinatorial features
provided by $\ell$. Such a labeled web is\footnote{Not all $\A$-webs
  can be obtained as labeled web.} then a plane trivalent oriented
graph with a \emph{thickness} function (denoted $t$) from the edges to
$\NN$ which preserves the flow. Such a graph is often called an
\emph{$\A$-web} or a MOY graph. These graphs are well-studied since
they encodes the finite dimensional representation theory of $\glN$
and $U_q(\glN)$ \cite{MOY, MR3263166, 1701.02932,RoseTub}.

  \begin{dfn}
    Let $\webA$ be an $\A$-web, a \emph{$\glN$-coloring} is a map $c$ from the set of edges of $\webA$ to $\powerset(\setN)$, such that:
    \begin{itemize}
    \item For each edge, $\#c(e) = t(e)$.   
    \item If $e_1, e_2, e_3$ meet at a vertex with
      $t(e_1)= t(e_2)+t(e_3)$, then $c(e_1)=c(e_2)\sqcup c(e_3)$. In
      other words, the coloring satisfies itself some kind of flow
      condition.
    \end{itemize}
    The set of $\glN$-coloring of an $\A$-web $\webA$ is denoted by
    $\col[\glN]{\webA}$.
  \end{dfn}

  Note that a $\glN$-colored $\A$-web and a $\sotN$-colored web carry
  the same type of information. Hence we can import a lot of
  definitions of Section~\ref{sec:webs} in this new context.

  \begin{dfn}[{\cite{LHRmoy}}]
    The \emph{degree} of a $\glN$-coloring $c$ of an $\A$-web $\webA$ is the integer given by the following formula:
    \begin{equation}
      \label{eq:8}
      \deg[\glN](c)=\sum_{1\leq a < b \leq \myN} \rotational{\mymacroone{a}{ b}(\webA,c)}. 
    \end{equation}
    The \emph{$\glN$-evaluation} of an $\A$-web is the element of $\NN[q, q^{-1}]$ given by the following formula:
    \begin{equation}
      \label{eq:9}
      \evA{\webA} = \sum_{c \in \col[\glN]{\webA}} q^{\deg(c)}.
    \end{equation}
    If $\webA$ is an $\A$-web, its \emph{rotational}
    $\rotational{\webA}$ is the rotational of the collection of plane
    cycles obtained when replacing each edge of $\webA$ by a number of
    parallel arcs given by its thickness and joining these arcs at
    vertices in the only crossing-less manner.
  \end{dfn}
  
  \begin{prop} \label{prop:BR-typeA}
    Let $\web$ be a web, then, one has:
    \begin{equation}
      \label{eq:11}
      \ev{\web} = q^{\frac{-\myN(\myN-1)}{2}\rotational{\webspin}}\sum_{\substack{\text{$\ell$ labeling} \\ \text{of $\web$}}} q^{\Ashift{\ell}}q^{(\myN-1)\rotational{\web^\ell}}\evA{\Gamma^\ell}.
    \end{equation}
  \end{prop}

  \begin{rmk}
In  \eqref{eq:11}, the sum is a priori infinite, however, only finitely many labelings $\ell$ are such 
$\evA{\Gamma^\ell}\neq 0$, these are precisely labelings for which all edges have thickness less than or equal to $\myN$.
\end{rmk}

  \begin{proof}
    The proof goes along the same strategy as that of
    Proposition~\ref{prop:BR-typeD}. First note that there is a
    canonical one-to-one correspondence between $\sotN$-colorings of
    $\web$ and the union of $\glN$-colorings of $\web^\ell$ for $\ell$
    running through labelings of $\web$. Let $(o,c_{\sotN})$ be a
    $\sotN$-coloring of $\web$. To fix notation, say it corresponds to
    the coloring $c_{\glN}$ of $\web^\ell$.  Following the definition
    of degrees, one has:
    \begin{equation}
      \label{eq:12}
      \deg((o,c_{\sotN})) - \deg[\glN](c_{\glN}) = \sum_{1\leq a < b\leq \myN} \rotational{\mymacrotwo{a}{b}(\web,(o, c_{\sotN}))}. 
    \end{equation}
    Hence it is enough to show that:
    \begin{equation}
      \label{eq:13}
      \sum_{1\leq a < b\leq \myN} \rotational{\mymacrotwo{a}{b}(\web,(o, c_{\sotN}))} =
      \frac{-\myN(\myN-1)}{2}\rotational{\webspin} +{\Ashift{\ell}} +{(\myN-1)\rotational{\web^\ell}}.
    \end{equation}
    As in the proof of Proposition~\ref{prop:BR-typeA}, we suppose
    that neighborhood of vertices are locally embedded as depicted in
    \eqref{eq:orthogonal-models}. Define
    $D = \sum_{1\leq a < b\leq \myN}
    \rotational{\mymacrotwo{a}{b}(\web,(o, c_{\sotN}))}$.  We will
    inspect how edges and vertices contribute to $D$.

    Let $e$ be a vectorial edge, the pair $(o, c_{\sotN})$ orients and
    colors it , by say $\{a \}$ with $a$ in $\setN$. This edge $e$ belongs to
    $\mymacrotwo{i}{j}(\web,(o, c_{\sotN}))$ if and only if
    $a \in\{i,j\}$, and when this is the case it is endowed with its
    orientation given by $o$. Hence the (normalized) curvature of $e$
    contributes $\myN-1$ times to $D$.

    Let $e$ be a spin edge (odd or even). It is given a color
    $X\subset \setN$ by $c_{\sotN}$, let $k$ be the cardinal of
    $X$. The edge $e$ belongs to
    $\mymacrotwo{i}{j}(\web,(o, c_{\sotN}))$ in two cases:
    \begin{itemize}
    \item If $i$ and $j$ belong to $X$, in this case the edge $e$
      appears with its natural orientation. This occurs $k(k-1)/2$
      times.
    \item If neither $i$ nor $j$ belong to $X$, in this case the edge
      $e$ appears with the opposite orientation. This occurs
      $(\myN-k)(\myN-k-1)/2$ times.
    \end{itemize}
    Since
    \[
      \frac{k(k-1)}{2} - \frac{(\myN -k)(\myN-k-1)}{2} =
      -\frac{\myN(\myN -1)}2 + k(\myN-1)
    \]
    Hence the curvature of $e$ contributes
    $\left(-\frac{\myN(\myN -1)}2 + k(\myN-1)\right)$ times to $D$. On
    the other hand it contributes $k$ times to
    $\rotational{\web^\ell}$ and once to $\rotational{\webspin}$.

    Let now $v$ be a vertex. First $v$ contributes to
    $\rotational{\web^\ell}$ with $\pm\frac{1}{4}$, let us denote
    $\epsilon_v$ this sign. One easily checks that the contribution of
    $v$ (via combinatorial curvature) to $D$ is precisely given by
    $\Ashift{\ell, v} - \epsilon_v\frac{N-1}{4}$.

    Summing all contribution, we obtain that:
    \begin{equation}
      \label{eq:14}
      D = \Ashift{\ell} + (N-1)\rotational{\web^\ell} - \frac{N(N-1)}{2}\rotational{\webspin}.
      \qedhere
    \end{equation} 
  \end{proof}

\section{Skein relations}
\label{sec:skein}
This section is devoted to stating and proving skein relations
satisfied by the web evaluation. Note that we only consider closed
webs. To state skein relations, we draw webs with boundaries (as
usually done in the literature), these relations are meant to be
understood in a skein theoretical way. In this section, we assume that
$\myN$ is an integer greater than or equal to $3$.

\subsection{A new kind of vertex}
\label{sec:new-kind-vertex}

\begin{dfn}\label{dfn:gweb}
  A \emph{\gweb{}} is a planar graph, with the same kind of
  combinatorial data as a web, but with the extra flexibility of
  having 4-valent vertices (all adjacent edges being vectorial)
  called \emph{singular vertices}. In other words, we extend our local
  models to vertices of the form:
  \begin{equation}
    \label{eq:15}
    \NB{\tikz[]{}}
  \end{equation}
  
  A \emph{coloring} of a \gweb{} is (essentially) defined in the same way as a
  coloring of a web, but with two modifications. One is a flow-preserving condition at singular
  vertices. The other is that if all edges adjacent to a singular vertex are
  colored with the same one-element subset of $\setN$, then their
  orientations should alternate, and the coloring carries the extra
  information of an element of $\{-, +\}$ attached to this
  vertex. Hence, up to symmetry, a coloring of a singular vertex
  follows one of the six local models:
  \begin{equation}
    \label{eq:15}
    \NB{\tikz[]{\input{\imagesfolder/td-singvertex-colorings-1}}} \quad
    \NB{\tikz[]{\input{\imagesfolder/td-singvertex-colorings-2}}} \quad
    \NB{\tikz[]{\begin{scope}[scale=0.6, font = \tiny]
  \draw[expand style =\stylevect, ->] (0,0)--(45:1) node[above] {$\{b\}$};
  \draw[expand style =\stylevect, -<-] (0,0)--(135:1) node[above] {$\{a\}$};;
  \draw[expand style =\stylevect, ->] (0,0)--(-135:1) node[below] {$\{a\}$};;
  \draw[expand style =\stylevect, -<-] (0,0)--(-45:1)node[below] {$\{b\}$};;
  \node at (0,0) {\singvertex};
\end{scope}
}} \quad
    \NB{\tikz[]{\begin{scope}[scale=0.6, font = \tiny]
  \draw[expand style =\stylevect, ->] (0,0)--(45:1) node[above] {$\{a\}$};
  \draw[expand style =\stylevect, -<-] (0,0)--(135:1) node[above] {$\{a\}$};;
  \draw[expand style =\stylevect, ->] (0,0)--(-135:1) node[below] {$\{b\}$};;
  \draw[expand style =\stylevect, -<-] (0,0)--(-45:1)node[below] {$\{b\}$};;
  \node at (0,0) {\singvertex};
\end{scope}
}} \quad
    \NB{\tikz[]{\begin{scope}[scale=0.6, font = \tiny]
  \draw[expand style =\stylevect, -<-] (0,0)--(45:1) node[above] {$\{b\}$};
  \draw[expand style =\stylevect, -<-] (0,0)--(135:1) node[above] {$\{a\}$};;
  \draw[expand style =\stylevect, ->] (0,0)--(-135:1) node[below] {$\{a\}$};;
  \draw[expand style =\stylevect, ->] (0,0)--(-45:1)node[below] {$\{b\}$};;
  \node at (0,0) {\singvertex};
\end{scope}
}} \quad
    \NB{\tikz[]{\begin{scope}[scale=0.6, font = \tiny]
  \draw[expand style =\stylevect, -<-] (0,0)--(45:1) node[above] {$\{b\}$};
  \draw[expand style =\stylevect, -<-] (0,0)--(135:1) node[above] {$\{a\}$};;
  \draw[expand style =\stylevect, ->] (0,0)--(-135:1) node[below] {$\{b\}$};;
  \draw[expand style =\stylevect, ->] (0,0)--(-45:1)node[below] {$\{a\}$};;
  \node at (0,0) {\singvertex};
\end{scope}
}} 
  \end{equation}  
  for $a\neq b$ pigments in
  $\setN$.
\end{dfn}

In order to define the evaluation of a \gweb{}, one needs to define the
degree of a coloring of such a web. We will first modify
formula~\eqref{eq:def-degree} by setting:
\begin{equation}\label{eq:degree-gcol}
\deg((o,c)) = {\#\{-\} - \#\{+\}} + \sum_{1\leq a < b\leq \myN}
  \rotational{\mymacroone{a }{ b}(\web, (o,c))}+ \rotational{\mymacrotwo{a}{ b}(\web, (o,c))} 
\end{equation}
where $\#\{+\}$ (resp.{} $\#\{-\}$) represent the number of $+$
(resp. $-$) on mono-colored singular vertices.  For formula
\eqref{eq:degree-gcol} to make sense, we need to specify what are
$\mymacroone{a}{ b}(\web, (o,c))$ and
$\mymacrotwo{a}{ b}(\web, (o,c))$ in this new context. More precisely,
we need to describe how to define cycles locally around singular
vertices. For this we refer to Tables \ref{tab:singvertex1} and
\ref{tab:singvertex2}.

\begin{table}[ht]
  \begin{tabular}{|c|c|c|c|c|c|c|} \hline
    &$\begin{array}{c}\mymacroone{a}{ b}\\(a<b)\end{array}$&$\begin{array}{c}\mymacroone{a}{ b} \\  (b<a)\end{array}$&$\begin{array}{c}\mymacroone{a}{ c} \\  (a<c) \end{array}$&$\begin{array}{c}\mymacroone{a}{ c} \\  (c<a)\end{array}$&$\begin{array}{c}\mymacroone{b}{ c} \\  (b<c)\end{array}$&$\begin{array}{c}\mymacroone{b}{ c} \\  (c<b)\end{array}$ \\ \hline
    \NB{\tikz[scale=1]{\input{\imagesfolder/td-singvertex-colorings-1}}}&\NB{\tikz[xscale=-1, scale=0.6]{\input{\imagesfolder/td-singvertex-cycle-smoothing}}}&\NB{\tikz[scale=0.6]{\input{\imagesfolder/td-singvertex-cycle-smoothing}}}&\NB{\tikz[xscale=-1, scale=0.6]{\input{\imagesfolder/td-singvertex-cycle-smoothing}}}&\NB{\tikz[scale=0.6]{\input{\imagesfolder/td-singvertex-cycle-smoothing}}}&\NB{\tikz[scale=0.6]{\singvertextikz[opacity=0.2]
}}& \NB{\tikz[scale=0.6]{}} \\ \hline
    \NB{\tikz[]{\input{\imagesfolder/td-singvertex-colorings-2}}}&\NB{\tikz[ rotate=90, scale=0.6]{\input{\imagesfolder/td-singvertex-cycle-smoothing}}}&\NB{\tikz[xscale=-1, rotate=90, scale=0.6]{\input{\imagesfolder/td-singvertex-cycle-smoothing}}}&\NB{\tikz[rotate=90, scale=0.6]{\input{\imagesfolder/td-singvertex-cycle-smoothing}}}&\NB{\tikz[xscale=-1, rotate=90, scale=0.6]{\input{\imagesfolder/td-singvertex-cycle-smoothing}}}&\NB{\tikz[scale=0.6]{}}& \NB{\tikz[scale=0.6]{}} \\ \hline
    \NB{\tikz[]{\input{\imagesfolder/td-singvertex-colorings-3}}}&\NB{\tikz[rotate=90, scale=0.6]{\input{\imagesfolder/td-singvertex-cycle-smoothing-2}}}&\NB{\tikz[rotate=90, xscale=-1, scale=0.6]{\input{\imagesfolder/td-singvertex-cycle-smoothing-2}}}&\NB{\tikz[rotate=90, scale=0.6]{\singvertextikz[opacity=0.2]
\begin{scope}[red, thick]
  \draw[->] (A) .. controls (O) .. (B);
\end{scope}}}&\NB{\tikz[yscale=-1, rotate=90,scale=0.6]{}}&\NB{\tikz[rotate=90, scale=-0.6]{}}&\NB{\tikz[yscale= -1, rotate=90, scale=-0.6]{}} \\ \hline
    \NB{\tikz[]{\input{\imagesfolder/td-singvertex-colorings-4}}}&\NB{\tikz[xscale=-1, scale=0.6]{\input{\imagesfolder/td-singvertex-cycle-smoothing}}}&\NB{\tikz[scale=0.6]{\input{\imagesfolder/td-singvertex-cycle-smoothing}}}&\NB{\tikz[xscale=-1, scale=0.6]{}}&\NB{\tikz[scale=0.6]{}}&\NB{\tikz[yscale=-1, scale=0.6]{}}&\NB{\tikz[scale=-0.6]{}} \\ \hline
    \NB{\tikz[]{\input{\imagesfolder/td-singvertex-colorings-5}}}&\NB{\tikz[xscale=-1, scale=0.6]{\input{\imagesfolder/td-singvertex-cycle-smoothing}}}&\NB{\tikz[scale=0.6]{\input{\imagesfolder/td-singvertex-cycle-smoothing}}}&\NB{\tikz[rotate=90, scale=0.6]{}}&\NB{\tikz[yscale=-1, rotate=90,scale=0.6]{}}&\NB{\tikz[yscale= -1, rotate=90, scale=-0.6]{}}&\NB{\tikz[rotate=90, scale=-0.6]{}} \\ \hline
    \NB{\tikz[]{\input{\imagesfolder/td-singvertex-colorings-6}}}&\NB{\tikz[xscale=-1, scale=0.6]{\input{\imagesfolder/td-singvertex-cycle-smoothing-2}}}&\NB{\tikz[scale=0.6]{\input{\imagesfolder/td-singvertex-cycle-smoothing-2}}}&\NB{\tikz[rotate=180, scale=0.6]{\singvertextikz[opacity=0.2]
\begin{scope}[red, thick]
  \draw[->] (A) .. controls (O) .. (C);
\end{scope}}}&\NB{\tikz[scale=0.6]{}}&\NB{\tikz[rotate=90, scale=0.6]{}}&\NB{\tikz[rotate= 270, scale=0.6]{}} \\ \hline
  \end{tabular}
  \caption{Local definition bicolored cycles (of first type) at singular vertices}\label{tab:singvertex1}
 \end{table}
 
 \begin{table}[ht]
  \begin{tabular}{|c|c|c|c|} \hline
    &\!\!$\mymacrotwo{a}{ b}$\!\!&\!\!$\mymacrotwo{a}{ c}$\!\!&\!\!$\mymacrotwo{b}{ c}$ \\ \hline
    \NB{\tikz[scale=1]{\input{\imagesfolder/td-singvertex-colorings-1}}}&

                                                                                 \NB{\tikz[xscale=1, scale=0.6]{\input{\imagesfolder/td-singvertex-cycle-smoothing}}}&\NB{\tikz[xscale=1, scale=0.6]{\input{\imagesfolder/td-singvertex-cycle-smoothing}}}&\NB{\tikz[scale=0.6]{}}\\ \hline \NB{\tikz[]{\input{\imagesfolder/td-singvertex-colorings-2}}}&\NB{\tikz[ xscale= -1, rotate=90, scale=0.6]{\input{\imagesfolder/td-singvertex-cycle-smoothing}}}&\NB{\tikz[xscale= -1, rotate=90, scale=0.6]{\input{\imagesfolder/td-singvertex-cycle-smoothing}}}&\NB{\tikz[scale=0.6]{}}\\ \hline
    \NB{\tikz[]{\input{\imagesfolder/td-singvertex-colorings-3}}}&\NB{\tikz[scale=0.6]{\input{\imagesfolder/td-singvertex-cycle-smoothing}}}&\NB{\tikz[yscale= -1, rotate=90, scale=0.6]{}}&\NB{\tikz[yscale= -1, rotate=90, scale=-0.6]{}}\\ \hline
    \NB{\tikz[]{\input{\imagesfolder/td-singvertex-colorings-4}}}&\NB{\tikz[rotate=90, xscale=-1, scale=0.6]{\input{\imagesfolder/td-singvertex-cycle-smoothing}}}&\NB{\tikz[scale=0.6]{}}&\NB{\tikz[xscale=-1,yscale=-1, scale=0.6]{}}\\ \hline
    \NB{\tikz[]{\input{\imagesfolder/td-singvertex-colorings-5}}}&\NB{\tikz[rotate=90, xscale=-1, scale=0.6]{\input{\imagesfolder/td-singvertex-cycle-smoothing-2}}}&\NB{\tikz[yscale=-1, rotate=90, scale=0.6]{}}&\NB{\tikz[yscale= 1, rotate=90, scale=-0.6]{}}\\ \hline
    \NB{\tikz[]{\input{\imagesfolder/td-singvertex-colorings-6}}}&\NB{\tikz[rotate=90, xscale=-1, scale=0.6]{\input{\imagesfolder/td-singvertex-cycle-smoothing-2}}}&\NB{\tikz[rotate=0, scale=0.6]{}}&\NB{\tikz[rotate=-90, scale=0.6]{}}\\ \hline
  \end{tabular}
        \caption{Local definition bicolored cycles (of second type) at singular vertices}\label{tab:singvertex2}
 \end{table}

We now can prove Theorem \ref{theo:B} from the introduction.

\begin{thm}\label{thm:B}
    If $\Gamma$ is a $4$-valent non-oriented planar graph, then $\evN{\Gamma}\in \mathbb{Z}_{\ge 0}[q^{\pm 1}]$.
\end{thm}
\begin{proof}
The evaluation of \gweb{} is defined combinatorially, so naturally is an element of $\mathbb{Z}_{\ge 0}[q^{\pm 1}]$.
\end{proof}

The evaluation of a \gweb{} is defined without using a the following Lemma which relates evaluation of generalized webs to evaluation of $\mathsf{D}$-webs. 
 
 \begin{lem}\label{lem:singular-square}
   For $\myN \geq 3$, the web evaluation $\ev{\cdot}$ satisfies the
   following skein relations:
   \begin{equation}
        \ev{\NB{\tikz[]{}} }
     =  \ev{\NB{\tikz[]{}} }
     =  \ev{\NB{\tikz[]{}} }
     =  \ev{\NB{\tikz[]{}} }
     =    [2]^{[N-3]} \ev{\NB{\tikz[]{}} }.
   \end{equation}
\end{lem}

 \begin{proof}
We only prove
   $\ev{\NB{\tikz[]{}}} = [2]^{[N-3]} \ev{\NB{\tikz[]{}}}$,
   the other ones are similar. Let us denote $\web_\square$ (resp.{}
   $\web_\times$) the web on the LHS (resp.{} RHS) of the desired equality.

   Let us define a surjective $2^{N-3}$-to-one map from the set of
   $\sotN$-colorings of $\web_\square$ to that of $\web_\times$. This
   map is canonical if the four ends of the square are not all
   colored by the same pigments (note that in this case, there are
   exactly two pigments appearing at the ends of the square and the
   flow condition is satisfied, so that this indeed gives a coloring of
   $\web_\times$). A careful but easy inspection of degrees for the
   various $\sotN$-colorings of $\web_\square$ corresponding to a
   given $\sotN$-coloring of $\web_\times$ shows that indeed
   \begin{equation}
     \label{eq:16}
     \sum_{(o,c)\in \col[\neq]{\web_\square}} \deg((o,c)) = [2]^{[\myN-3]}
     \sum_{(o,c)\in \col[\neq]{\web_\times}} \deg((o,c)).
   \end{equation}
   Here $\col[\neq]{\web_\square}$ (resp.{}
   $\col[\neq]{\web_\times}$) denotes the set of colorings of
   $\web_\square$ (resp. ${\web_\times}$)  which do not associate the same pigments to the four
   vectorial edges touching the square.

   Hence we need to focus on $\sotN$-colorings of $\Gamma_\square$ and
   $\Gamma_\times$ with the same pigments on all four ends of the square. To fix notations, let $(o,c)$ be such a coloring and
   $a \in \setN$ be the pigment. Let $b$ the maximum of
   $\setN \setminus \{a\}$ and suppose that the orientations given by
   the coloring are locally given by:
   \[
\NB
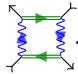
{\tikz[]{  \begin{scope}[scale = 0.6]
    \coordinate (A) at ( 135:1) ;
    \coordinate (B) at (  45:1) ;
    \coordinate (C) at ( -45:1) ;
    \coordinate (D) at (-135:1) ;
    \coordinate (a) at ( 135:0.6) ;
    \coordinate (b) at (  45:0.6) ;
    \coordinate (c) at ( -45:0.6) ;
    \coordinate (d) at (-135:0.6) ;
    \draw[<-] (A) --(a);
    \draw[>-] (B) --(b);
    \draw[<-] (C) --(c);
    \draw[>-] (D) --(d);
    \draw[expand style =\styleeven] (a) -- (b) \arreven;
    \draw[expand style =\styleodd] (b) -- (c) \arrodd;
    \draw[expand style =\styleeven] (c) -- (d) \rarreven;
    \draw[expand style =\styleodd] (d) -- (a) \arrodd;
  \end{scope}

}}.
   \]
   Denote by $X$ the color given to the vertical edges in the square
   (in particular $a \notin X$).  Let us rename $(o,c)=(o,c_X)$ and
   denote $(o, c_\emptyset)$ the coloring identical to $(o,c_X)$ but
   with $\emptyset$ in place of $X$ coloring the vertical edges. An
    inspection of bicolored cycles shows that:
   \begin{equation}
     \deg(o, c_X) - \deg(o,c_\emptyset) = -\sum_{\substack{i \in X \\ i < a}}  2(i-1)  - \sum_{\substack{i \in X \\ i > a}}  2(i-2) - \sum_{i \in X \\ i > a} 2.
   \end{equation}
   The two first terms on the RHS accounts for bicolored cycles of
   either type with both pigments different from $a$. The last one for
   bicolored cycles with one pigment equal to $a$. So that if one sums
   over all possible $X$'s, one gets:
   \begin{equation}
     \sum_{X} q^{\deg(o, c_X) - \deg(o,c_\emptyset)}= q^{-(\myN-1)(\myN-2)/2 +(\myN-a)} (q^{a-2}+q^{2-a}) \prod_{i=1}^{\myN-3}(q^i + q^{-i}).
     \end{equation}
     On the other hand, there are two colorings of $\web_\times$ which
     corresponds to this situation. We may denote them $(o,c_+)$ and
     $(o, c_-)$. One has:
     \begin{equation}
       \deg(o, c_+) - \deg(o,c_\emptyset) = q^{-1 -(\myN-1)(\myN-2)/2
         +(a-1)-(\myN-a) }
     \end{equation}
     {and}
     \begin{equation}                                   
       \deg(o, c_-) - \deg(o,c_\emptyset) = q^{1-(\myN-1)(\myN-2)/2 - (\myN-1)}.
     \end{equation}
     The $\pm 1$ term corresponds to the extra weight coming from the
     $\pm$-decoration included in $c_\pm$. The $-(\myN-1)(\myN-2)/2$
     accounts for the bicolored cycles (pigments different from $a$)
     in the square of $(\web_\square,(o, c_\emptyset))$ which do not
     appear in $(\web_\times ,(o,c_\pm))$. The remaining terms account
     for the difference between bicolored cycles (one of the pigment
     being $a$) in $(\web_\square,(o, c_\emptyset))$ and
     $(\web_\times ,(o,c_\pm))$ of the first type for $c_+$ and of the
     second type for $c_-$.  Hence, we do have:
     \begin{equation}
       \sum_{X} q^{\deg((o, c_X))} = [2]^{[\myN-3]}(q^{\deg((o,c_+))} +  q^{\deg((o,c_-))})\qedhere.
     \end{equation}
 \end{proof}
 
\subsection{Local relations}
\label{sec:local-relations}
From the definition of $\ev{\cdot}$ and
Examples~\ref{ex:vect}, \ref{ex:even}, \ref{ex:odd}, we get:

 \begin{prop}\label{prop:disjoint-union}
   The web evaluation $\ev{\cdot}$ is multiplicative under disjoint
   unions. Moreover, we have:
   \begin{align}
     \label{eq:17}
     \ev{\NB{\tikz[]{}}} &= [2\myN-1] +1 \\ \ev{\NB{\tikz[]{  \begin{scope}[scale = 0.6]
    \draw[expand style =\styleodd] (0,0) arc (0:360:0.5cm) \arrodd; 
  \end{scope}

}}}&=\ev{\NB{\tikz[]{  \begin{scope}[scale = 0.6]
    \draw[expand style =\styleodd] (0,0) arc (0:360:0.5cm) \rarrodd; 
  \end{scope}

}}}=\ev{\NB{\tikz[]{  \begin{scope}[scale = 0.6]
    \draw[expand style =\styleeven] (0,0) arc (0:360:.5cm) \arreven; 
  \end{scope}

}}}=\ev{\NB{\tikz[]{  \begin{scope}[scale = 0.6]
    \draw[expand style =\styleeven] (0,0) arc (0:360:.5cm) \rarreven; 
  \end{scope}

}}}= [2]^{[\myN-1]}.
   \end{align}
 \end{prop}

 An adaptation of Example~\ref{ex:theta} proves:
 \begin{prop}\label{p:combinatorial-skein-bigon}
   The following local relations hold for $\myN\geq 2$:
   \begin{align}
     \ev{\NB{\tikz[]{\begin{scope}
  \draw[expand style= \stylevect] (-0.5,0) -- (-0.3,0);
  \draw[expand style= \stylevect] ( 0.5,0) -- ( 0.3,0);
  \draw[expand style= \styleodd] (-0.3,0) ..  controls +(0.2, 0.3) and +(-0.2, 0.3) .. (0.3,0) \arrodd;
  \draw[expand style= \styleeven] (-0.3,0) ..  controls +(0.2, -0.3) and +(-0.2, -0.3) .. (0.3,0) \rarreven;
\end{scope}
}}} =
     \ev{\NB{\tikz[yscale=-1]{}}}=&
     [2]^{[\myN-2]}
     \ev{\NB{\tikz[]{\begin{scope}
  \draw[expand style= \stylevect] (-0.5,0) -- ( 0.5,0);
\end{scope}}}}, \\
     \ev{\NB{\tikz[]{\begin{scope}
  \draw[expand style= \styleeven] (-0.5,0) -- (-0.3,0) \arreven;
  \draw[expand style= \styleeven] ( 0.5,0) -- ( 0.3,0) \arreven;
  \draw[expand style= \styleodd] (-0.3,0) ..  controls +(0.2, 0.3) and +(-0.2, 0.3) .. (0.3,0) \arrodd;
  \draw[expand style= \stylevect] (-0.3,0) ..  controls +(0.2, -0.3) and +(-0.2, -0.3) .. (0.3,0);
\end{scope}
}}} =
     \ev{\NB{\tikz[yscale=-1]{}}}=&
     [\myN]
     \ev{\NB{\tikz[]{\begin{scope}
  \draw[expand style= \styleeven] (-0.5,0) -- ( 0.5,0) \arreven;
  % \draw[expand style= \styleeven] ( 0.5,0) -- ( 0.3,0) \arreven;
  % \draw[expand style= \styleodd] (-0.3,0) ..  controls +(0.2, 0.3) and +(-0.2, 0.3) .. (0.3,0) \arrodd;
  % \draw[expand style= \stylevect] (-0.3,0) ..  controls +(0.2, -0.3) and +(-0.2, -0.3) .. (0.3,0);
\end{scope}
}}}, \\
     \ev{\NB{\tikz[]{\begin{scope}
  \draw[expand style= \styleodd] (-0.5,0) -- (-0.3,0) \arrodd;
  \draw[expand style= \styleodd] ( 0.5,0) -- ( 0.3,0) \arrodd;
  \draw[expand style= \styleeven] (-0.3,0) ..  controls +(0.2, 0.3) and +(-0.2, 0.3) .. (0.3,0) \arreven;
  \draw[expand style= \stylevect] (-0.3,0) ..  controls +(0.2, -0.3) and +(-0.2, -0.3) .. (0.3,0);
\end{scope}
}}} =
     \ev{\NB{\tikz[yscale=-1]{}}}=&
     [\myN]
     \ev{\NB{\tikz[]{\begin{scope}
  \draw[expand style= \styleodd] (-0.5,0) -- ( 0.5,0) \arrodd;
  % \draw[expand style= \styleeven] ( 0.5,0) -- ( 0.3,0) \arreven;
  % \draw[expand style= \styleodd] (-0.3,0) ..  controls +(0.2, 0.3) and +(-0.2, 0.3) .. (0.3,0) \arrodd;
  % \draw[expand style= \stylevect] (-0.3,0) ..  controls +(0.2, -0.3) and +(-0.2, -0.3) .. (0.3,0);
\end{scope}
}}}.
 \end{align}
 \end{prop}

 \begin{cor}\label{cor:curl-sing}
   The following local relation holds:
   \begin{equation}
     \label{eq:19}
     \ev{\NB{\tikz[]{\begin{scope}
  \clip (-0.5, -0.21) rectangle (0.5, 0.31);
  \draw (-0.5, -0.2) .. controls +(0.1,0) and +(-0.2, -0.2) .. (0,0);
  \draw (0.5, -0.2) .. controls +(-0.1,0) and +( 0.2, -0.2) .. (0,0);
  \draw (0,0) .. controls +(0.4, 0.4) and +(-0.4, 0.4) .. (0,0);
  \node at (0,0) {\singvertex};
\end{scope}}}} = ([2\myN-2] +[2]) \ev{\NB{\tikz[]{}}}.
   \end{equation}   
 \end{cor}
 \begin{proof}
   This follows from Lemma \ref{lem:singular-square}, Proposition \ref{p:combinatorial-skein-bigon}, and the fact that $[2]^{[\myN-2]}\cdot [\myN] /[2]^{[\myN-3]} = [\myN]\cdot(q^{\myN-2}+ q^{2-\myN}) = [2\myN -2] + [2]$.
 \end{proof}
 
 \begin{prop}\label{prop:triangle}
   The following local relations hold:
   \begin{align}
     \ev{\NB{\tikz[scale=0.6]{\begin{scope}
  \draw[expand style =\stylevect] (90:.5) -- (90:1);
  \draw[expand style =\styleeven] (210:.5) -- (210:1) \arreven;
  \draw[expand style =\styleodd] (-30:0.5) -- (-30:1) \arrodd;
  \draw[expand style =\stylevect] (210:.5) -- (-30:.5);
  \draw[expand style =\styleeven] (90:.5) -- (-30:.5) \arreven;
  \draw[expand style =\styleodd] (90:0.5) -- (210:.5) \arrodd;  
\end{scope}
}}} &= [\myN-1] \ev{\NB{\tikz[scale=0.6]{}}}, \qquad 
                                 \ev{\NB{\tikz[scale=0.6]{\begin{scope}
  \draw[expand style =\stylevect] (90:.5) -- (90:1);
  \draw[expand style =\styleodd] (210:.5) -- (210:1) \arrodd;
  \draw[expand style =\styleeven] (-30:0.5) -- (-30:1) \arreven;
  \draw[expand style =\stylevect] (210:.5) -- (-30:.5);
  \draw[expand style =\styleodd] (90:.5) -- (-30:.5) \arrodd;
  \draw[expand style =\styleeven] (90:0.5) -- (210:.5) \arreven;  
\end{scope}
}}} &= [\myN-1] \ev{\NB{\tikz[scale=0.6]{}}}, \\                            
     \ev{\NB{\tikz[scale=0.6]{\begin{scope}
  \draw[expand style =\stylevect] (90:.5) -- (90:1);
  \draw[expand style =\styleeven] (210:.5) -- (210:1) \rarreven;
  \draw[expand style =\styleodd] (-30:0.5) -- (-30:1) \rarrodd;
  \draw[expand style =\stylevect] (210:.5) -- (-30:.5);
  \draw[expand style =\styleeven] (90:.5) -- (-30:.5) \rarreven;
  \draw[expand style =\styleodd] (90:0.5) -- (210:.5) \rarrodd;  
\end{scope}
}}} &= [\myN-1] \ev{\NB{\tikz[scale=0.6]{}}}, \qquad 
                                 \ev{\NB{\tikz[scale=0.6]{\begin{scope}
  \draw[expand style =\stylevect] (90:.5) -- (90:1);
  \draw[expand style =\styleodd] (210:.5) -- (210:1) \rarrodd;
  \draw[expand style =\styleeven] (-30:0.5) -- (-30:1) \rarreven;
  \draw[expand style =\stylevect] (210:.5) -- (-30:.5);
  \draw[expand style =\styleodd] (90:.5) -- (-30:.5) \rarrodd;
  \draw[expand style =\styleeven] (90:0.5) -- (210:.5) \rarreven;  
\end{scope}
}}} &= [\myN-1] \ev{\NB{\tikz[scale=0.6]{}}}.
   \end{align}
 \end{prop}
 \begin{proof}
   Thanks to Corollary~\ref{cor:ev-all-equal}, it is enough to prove
   the first local relation. Let us write $\web_\triangle$ (resp.{}
   $\web_\lambda$) to denote the webs on the LHS (resp.{} RHS) of the identity.
   Note that there is a canonical $\myN$-to-one map from $\col{\web_\triangle}$
   to $\col{\web_\lambda}$. Fix $(o,c)$ a coloring of
   $\web_\lambda$.

   Suppose that the orientation of the vectorial edge
   at the distinguished vertex of $\web_\lambda$ points upward (i.e.{}
   out of the vertex) and denote by $\{a\}\subseteq\setN$ its
   color. Denote by $X$ the color of the even spin edge at the
   distinguished vertex. One has $a \in X$ and the odd spin edge is
   colored by $X\setminus \{a \}$. The colorings of $\web_\triangle$
   corresponding to $(o,c)$ can be parameterized as
   $(o,c_b)_{b\in \setN\setminus\{a\}}$, where $\{b\}$ is the color of
   the horizontal vectorial edge in the triangle. If $b \in X$
   (resp.{} $b\notin X$), then this edge is oriented rightward (resp.{}
   leftward). If $b \notin \{i,j\}\subseteq \setN$, then
   \begin{align}
     \rotational{\mymacroone{i }{ j} (\web_\triangle, (o,c_b))} &=
     \rotational{\mymacroone{i }{ j} (\web_\lambda, (o,c))} \qquad \text{and}\\
     \rotational{\mymacrotwo{i}{j} (\web_\triangle,(o,c_b))}&=
     \rotational{\mymacrotwo{i}{j} (\web_\lambda,(o,c))}.
     \end{align}
     Note also, that
\begin{align}
     \rotational{\mymacroone{a }{ b} (\web_\triangle, (o,c_b))} &=
     \rotational{\mymacroone{a }{ b} (\web_\lambda, (o,c))}      \qquad \text{and}  \\
     \rotational{\mymacrotwo{a}{b} (\web_\triangle,(o,c_b))}& =
     \rotational{\mymacrotwo{a}{b} (\web_\lambda,(o,c))}.
 \end{align}
 Suppose that $b \in X\setminus\{a\}$, for $i\in X \setminus \{a, b\}$, one has:
\begin{align}
  \rotational{\mymacroone{i }{ b} (\web_\triangle, (o,c_b))} &=
                                                           \begin{cases}
                                   \rotational{\mymacroone{i }{ b} (\web_\lambda, (o,c))}  +1 & \text{if $i<b$,}\\
                                                             \rotational{\mymacroone{i }{ b} (\web_\lambda, (o,c))}  -1 & \text{if $b<i$,}
                                                           \end{cases}
                                                                                                                      \qquad \text{and}  \\
  \rotational{\mymacrotwo{i}{b} (\web_\triangle,(o,c_b))}& =
  \rotational{\mymacrotwo{a}{b} (\web_\lambda,(o,c))};
 \end{align}
 for $i \in \setN \setminus X$, then:
\begin{align}
  \rotational{\mymacroone{i }{ b} (\web_\triangle, (o,c_b))} &= \rotational{\mymacroone{i }{ b} (\web_\lambda, (o,c))}
                                                                                                                      \qquad \text{and}  \\
  \rotational{\mymacrotwo{i}{b} (\web_\triangle,(o,c_b))}& =   \rotational{\mymacrotwo{a}{b} (\web_\lambda,(o,c))} + 1, 
\end{align}
Hence we obtain that:
\begin{equation}
  \label{eq:20}
    \sum_{b\in X\setminus \{a\}} q^{\deg(o,c_b)} = q^{\myN-\#X}[\#X-1]q^{\deg(o,c)}.
\end{equation}
 Suppose that $b \notin X$. For $i\in X \setminus \{a\}$, one has:
\begin{align}
  \rotational{\mymacroone{i }{ b} (\web_\triangle, (o,c_b))} &= \rotational{\mymacroone{i}{ b} (\web_\lambda, (o,c))}
                                                                                                                      \qquad \text{and}  \\
  \rotational{\mymacrotwo{i}{b} (\web_\triangle,(o,c_b))}& =   \rotational{\mymacrotwo{a}{b} (\web_\lambda,(o,c))} - 1;
 \end{align}
 for $i \in \setN \setminus X$, then:
 \begin{align}
  \rotational{\mymacroone{i }{ b} (\web_\triangle, (o,c_b))} &=
                                                           \begin{cases}
              \rotational{\mymacroone{i }{ b} (\web_\lambda, (o,c))}  -1 & \text{if $i<b$,}\\
              \rotational{\mymacroone{i }{ b} (\web_\lambda, (o,c))}  +1 & \text{if $b<i$,}
                                                           \end{cases}
                                                                                                                      \qquad \text{and}  \\
  \rotational{\mymacrotwo{i}{b} (\web_\triangle,(o,c_b))}& =
  \rotational{\mymacrotwo{a}{b} (\web_\lambda,(o,c))}.
 \end{align}
Hence we obtain that:
\begin{equation}
  \label{eq:21}
    \sum_{b\in \setN \setminus X} q^{\deg(o,c_b)} = q^{-\#X-1}[\myN-\#X]q^{\deg(o,c)}.
\end{equation}
Altogether, this gives:
\begin{equation}
  \label{eq:22}
    \sum_{b\in \setN \setminus \{a\}} q^{\deg(o,c_b)} = [\myN-1]q^{\deg(o,c)}.
\end{equation}

The case where the vectorial edge is downward oriented (i.e. into the vertex), and therefore $a\notin X$, is similar and we also get~\eqref{eq:22} in that case. This implies that $\ev{\web_\triangle} = [\myN-1]\ev{\web_\lambda}$.
 \end{proof}

 \begin{prop}\label{prop:XIHII}
   The following relation holds:
   \begin{align} \label{eq:XH=1}
     \ev{\NB{\tikz[scale =0.6]{\begin{scope}
  \draw[expand style =\stylevect] (135:1) --  (-45:0.6);
  \draw[expand style =\stylevect] (45:1) --  (-135:0.6);
  \node at (0,0) {\singvertex};
  \draw[expand style =\styleeven] (-135:1) -- (-135:.6) \arreven;
  \draw[expand style =\styleodd] (-135:.6) -- (-45:.6) \arrodd;
  \draw[expand style =\styleeven] (-45:.6) -- (-45:1) \arreven;
\end{scope}
}}} = [2]\ev{\NB{\tikz[scale =0.6]{\input{\imagesfolder/td-H-1}}}} + [\myN -1 ] \ev{\NB{\tikz[scale =0.6]{\begin{scope}
  \draw[expand style =\stylevect] (135:1) .. controls (135:0.7) and (45:0.7) ..  (45:1);
  \draw[expand style =\styleeven] (-135:1) .. controls (-135:.7) and (-45:0.7) .. (-45:1) \arreven;
\end{scope}
}}}
   \end{align}
 \end{prop}
 \begin{proof}  
   This is once more a careful inspection of what happens
   locally. There are essentially two different cases: the one for
   which the two outer vectorial edges are colored by the same
   pigments (necessarily with opposite orientations) or the case where
   these vectorial edges are colored by different pigments
   (orientations are then arbitrary). We start with the latter. Let us
   denote by $\web_\times$, $\web_H$ and $\web_=$ the webs in
   \eqref{eq:XH=1}, read from left to right.

   Suppose that the two outer vectorial edges are inwardly oriented. Let us fix
   notations and let $i\neq j$ be two elements of $\setN$ and $X$ be
   an even subset of $\setN$ with $\{i,j\}\cap X = \emptyset$. One can
   consider colorings of $\web_\times$, $\web_H$ and $\web_=$, such
   that the restriction to the border of the distinguished subdiagram from \eqref{eq:XH=1} is
   given by:
   \[
     \NB
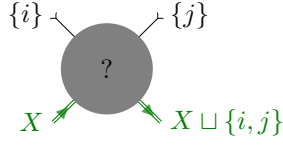
{\tikz[]{\input{\imagesfolder/td-XH-col-1}}}
   \]
   The set of such colorings for $\web_=$ is clearly empty. There is a canonical two-to-one correspondance between colorings of $\web_\times$ and colorings of $\web_H$ satisfying this constraint.
   Let $(o,c)$ denote such a coloring of $\web_H$. On the distinguished piece of $\web_H$, $(o,c)$ is entirely determined by the coloring of the boundary:
   \[
     \NB
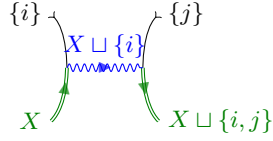
{\tikz[]{\input{\imagesfolder/td-H-1-col}}}
   \]
   The corresponding two colorings of $\web_\times$ are given locally by:
   \[
     (o_1,c_1) :=  \NB
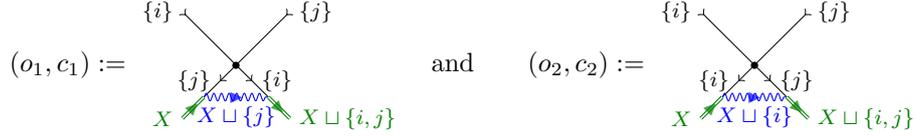
{\tikz[]{\input{\imagesfolder/td-XI-1-col-1}}} \quad \text{and} \qquad (o_2,c_2) :=  \NB{\tikz[]{\input{\imagesfolder/td-XI-1-col-2}}}.
   \]
   A careful inspection of bicolored circles shows that for any $a\neq b \in \setN$,
   \begin{equation}\rotational{\mymacrotwo{a}{b} (\web_H,(o,c))} = \rotational{\mymacrotwo{a}{b} (\web_\times,(o_1,c_1))} = \rotational{\mymacrotwo{a}{b} (\web_\times,(o_2,c_2))},\end{equation}
   and if $\{a,b\}\neq \{i,j\}$:
   \begin{equation}\rotational{\mymacroone{a }{ b} (\web_H,(o,c))} =
     \rotational{\mymacroone{a }{ b} (\web_\times,(o_1,c_1))} =
     \rotational{\mymacroone{a }{ b} (\web_\times,(o_2,c_2))}.\end{equation}
   Finally, 
   \begin{align}
     \rotational{\mymacroone{i }{ j} (\web_\times,(o_1,c_1))}
     &=\rotational{\mymacroone{i }{ j} (\web_H,(o,c))} + \mathrm{sign}(j-i)\\
     \rotational{\mymacroone{i }{ j} (\web_\times,(o_2,c_2))}
                  &=\rotational{\mymacroone{i}{ j} (\web_H,(o,c))} +  \mathrm{sign}(i-j). 
   \end{align}
   This implies that $q^{\deg(o_1,c_1)} + q^{\deg(o_2,c_2)} = [2] q^{\deg(o,c)}$.  

   The other cases (\emph{i.e.{}} other possible orientations of the two outer vectorial edges) can be treated similarly and are left to the reader.

   So we need to consider the case were the two outer vectorial edges
   are colored by the same pigment (necessarily one points inward, the
   other one outward). To fix conventions, consider the following coloring:
   \begin{equation}\label{eq:XH-col-2}
\NB
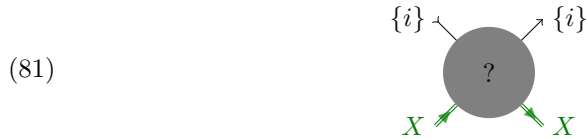
{\tikz[]{\input{\imagesfolder/td-XH-col-2}}}     
\end{equation}

There are two subcases to consider: either $i \in X$ or $i \notin X$. Let us inspect the first one first which is easier than the second. The set of such colorings for $\web_H$ is empty. On the other hand, on the distinguished piece of $\web_=$, such a coloring is completely determined by the condition. There is a canonical $(\myN-1)$-to-one correspondence between these colorings of $\web_\times$ and those of $\web_=$, locally this correspondence is given by:
\begin{equation}
  \label{eq:21}
  \NB{\tikz[]{\begin{scope}
  \draw[->-, expand style =\stylevect] (135:1) .. controls (135:0.7) and (45:0.7) ..  (45:1) node[above, pos =0.5] {$\{i\}$};
  \draw[expand style =\styleeven] (-135:1) .. controls (-135:.7) and (-45:0.7) .. (-45:1) \arreven node[below, pos =0.5] {$X$};
\end{scope}
}} \qquad \leftrightsquigarrow \qquad
  \NB{\tikz[]{\input{\imagesfolder/td-XI-1-col-3}}} \qquad \text{or} \qquad \NB{\tikz[]{\input{\imagesfolder/td-XI-1-col-4}}},
\end{equation}
where the orientation of the two possible colorings of
$\web_\times$ is determined by whether or not $j$ belongs to $X$.
Fix $(o,c)$ such a coloring of $\web_{=}$ and denote by $C$ the set of corresponding colorings of $\web_{\times}$. A careful inspection of rotationals of bicolored cycles, very similar to that in the proof of Proposition~\ref{prop:triangle}, shows that:
\begin{equation}
  \label{eq:23}
  \sum_{(\tilde{o},\tilde{c}) \in C} q^{\deg(\tilde{o},\tilde{c})} = [\myN-1] q^{\deg(o,c)}.  
\end{equation}

It remains to treat the last subcase which locally looks like \eqref{eq:XH-col-2} but with $i \notin X$. The correspondence between various colorings of webs $\web_{\times}$, $\web_{H}$ and $\web_=$ is a bit more subtle: locally there are $\myN+1$ colorings of $\web_\times$:
\begin{equation}
  \NB{\tikz[]{\input{\imagesfolder/td-XI-1-col-3}}} \quad \text{or} \quad \NB{\tikz[]{\input{\imagesfolder/td-XI-1-col-4}}}, \quad \NB{\tikz[]{\input{\imagesfolder/td-XI-1-col-5}}} \quad \text{and} \quad \NB{\tikz[]{\input{\imagesfolder/td-XI-1-col-6}}}.  \label{eq:24}
\end{equation}
While for $\web_{H}$ and $\web_{=}$ the colorings are completly determined by the boundary condition. Locally they are given by:
\begin{equation}
  \label{eq:25}
  \NB{\tikz[]{}} \qquad \text{and} \qquad 
  \NB{\tikz[]{\input{\imagesfolder/td-H-1-col-2}}}
\end{equation}

For $j \in \setN \setminus\{i\} \cup \{i^+, i^-\}$, let us denote by $(o_j,c_j)$ the colorings of $\web_\times$ (given locally in \eqref{eq:24}) and by $(o_{H},c_{H})$ and $(o_{=}, c_{=})$ the colorings of $\web_H$ and $\web_{=}$ locally given on \eqref{eq:25}.

One has for $j \in X$ and $k \in \setN \setminus (X \sqcup\{i \})=: Y$:
\begin{align}
  \deg(o_H, c_H) =& \deg(o_=, c_=) - (\myN - \#X -1) -  \#X_{<i}+ \#X_{>i}, \\
  \deg(o_j, c_j) =& \deg(o_=, c_=) - (\myN - \#X) + \mathrm{sign}(j-i) -\#X_{<j} +\#X_{>j}\\
  \deg(o_k, c_k) =& \deg(o_=, c_=) +\#X -1 + \mathrm{sign}(k-i) +\#Y_{<k} -\#Y_{>k} \\
  \deg(o_{i^-}, c_{i^-}) =& \deg(o_=, c_=) +\#X -1 +  \#Y_{<i} - \#Y_{>i}  \\
  \deg(o_{i^+}, c_{i^+}) =& \deg(o_=, c_=)  -(\myN -\#X - 2)   -  \#X_{<i}+ \#X_{>i},
\end{align}
where  $Z_{<\ell} = \{z \in  Z | z<\ell\}$ and $Z_{>\ell} = \{z \in  Z | z>\ell\}$,
so that:
\begin{align}
  q^{\deg(o_{i^-}, c_{i^-})}+\sum_{k \in Y} q^{\deg(o_k,c_k)} &= q^{\#X-1}[\myN - \#X]
\end{align} and
\begin{align}
  \sum_{j \in X} q^{\deg(o_j,c_j)} &= q^{-(\myN+\#X-2)}(q^{-1+\#X_{>i} }[\#X_{<i}] + q^{1-\#X_{<i} }[\#X_{>i}])\\  &= q^{\myN - \#X}([\# X -1] + q^{\#X_{>i}-\#X_{<i}}),
\end{align}
Finally this gives:
\begin{align*}
q^{\deg(o_{i^-}, c_{i^-})}+ q^{\deg(o_{i^+}, c_{i^+})}+ \sum_{j \in Y \cup X} q^{\deg(o_j,c_j)} = [\myN-1] q^{\deg(o_=, c_=) } + [2]q^{\deg(o_H, c_H)}. 
\end{align*}
Which is precisely what was needed. 
\end{proof}

In view of the global relations given in Corollary~\ref{cor:ev-all-equal},
Proposition~\label{prop:XIHII} immediately gives additional skein relations:

\begin{cor}
     The following relations hold:
   \begin{align} \label{eq:XH=1}
     \ev{\NB{\tikz[scale =0.6]{\begin{scope}
  \draw[expand style =\stylevect] (135:1) --  (-45:0.6);
  \draw[expand style =\stylevect] (45:1) --  (-135:0.6);
  \node at (0,0) {\singvertex};
  \draw[expand style =\styleodd] (-135:1) -- (-135:.6) \arrodd;
  \draw[expand style =\styleeven] (-135:.6) -- (-45:.6) \arreven;
  \draw[expand style =\styleodd] (-45:.6) -- (-45:1) \arrodd;
\end{scope}
}}} &= [2]\ev{\NB{\tikz[scale =0.6]{\input{\imagesfolder/td-H-1-2}}}} + [\myN -1 ] \ev{\NB{\tikz[scale =0.6]{\begin{scope}
  \draw[expand style =\stylevect] (135:1) .. controls (135:0.7) and (45:0.7) ..  (45:1);
  \draw[expand style =\styleodd] (-135:1) .. controls (-135:.7) and (-45:0.7) .. (-45:1) \arrodd;
\end{scope}
}}} \\
     \ev{\NB{\tikz[scale =0.6]{\begin{scope}
  \draw[expand style =\stylevect] (135:1) --  (-45:0.6);
  \draw[expand style =\stylevect] (45:1) --  (-135:0.6);
  \node at (0,0) {\singvertex};
  \draw[expand style =\styleeven] (-135:1) -- (-135:.6) \rarreven;
  \draw[expand style =\styleodd] (-135:.6) -- (-45:.6) \rarrodd;
  \draw[expand style =\styleeven] (-45:.6) -- (-45:1) \rarreven;
\end{scope}
}}} &= [2]\ev{\NB{\tikz[scale =0.6]{\input{\imagesfolder/td-H-1-3}}}} + [\myN -1 ] \ev{\NB{\tikz[scale =0.6]{\begin{scope}
  \draw[expand style =\stylevect] (135:1) .. controls (135:0.7) and (45:0.7) ..  (45:1);
  \draw[expand style =\styleeven] (-135:1) .. controls (-135:.7) and (-45:0.7) .. (-45:1) \rarreven;
\end{scope}
}}} \\
     \ev{\NB{\tikz[scale =0.6]{\begin{scope}
  \draw[expand style =\stylevect] (135:1) --  (-45:0.6);
  \draw[expand style =\stylevect] (45:1) --  (-135:0.6);
  \node at (0,0) {\singvertex};
  \draw[expand style =\styleodd] (-135:1) -- (-135:.6) \rarrodd;
  \draw[expand style =\styleeven] (-135:.6) -- (-45:.6) \rarreven;
  \draw[expand style =\styleodd] (-45:.6) -- (-45:1) \rarrodd;
\end{scope}
}}} &= [2]\ev{\NB{\tikz[scale =0.6]{\input{\imagesfolder/td-H-1-4}}}} + [\myN -1 ] \ev{\NB{\tikz[scale =0.6]{\begin{scope}
  \draw[expand style =\stylevect] (135:1) .. controls (135:0.7) and (45:0.7) ..  (45:1);
  \draw[expand style =\styleodd] (-135:1) .. controls (-135:.7) and (-45:0.7) .. (-45:1) \rarrodd;
\end{scope}
}}} 
   \end{align}
\end{cor}

\begin{cor}\label{cor:2sing}
  The following relation holds:
  \begin{equation}
    \ev{\NB{\tikz[rotate = 90, scale= 0.7]{\begin{scope}[expand style =\stylevect]
  \draw[name path= p] ( 0, 0.7) .. controls +( 0.5, -0.5) and +( 0.5, 0.5) .. (0,-0.7);
  \draw[name path= q] ( 0.5, 0.7) .. controls +(-0.5, -0.5) and +(-0.5, 0.5) .. ( 0.5,-0.7);
  \path[name intersections={of=p and q}] (intersection-1) node {\singvertex};
  \path[name intersections={of=p and q}] (intersection-2) node {\singvertex};
\end{scope}}}}= [2]\ev{\NB{\tikz[rotate = 90, scale= 0.7]{\begin{scope}[expand style =\stylevect]
  \draw[name path= p] ( 0, 0.7) .. controls +( 0.3, -0.3) and +( -0.3, 0.3) .. (0.5,-0.7);
  \draw[name path= q] ( 0.5, 0.7) .. controls +(-0.3, -0.3) and +(0.3, 0.3) .. ( 0,-0.7);
  \path[name intersections={of=p and q}] (intersection-1) node {\singvertex};
\end{scope}
%%% Local Variables:
%%% mode: latex
%%% TeX-master: t
%%% End:
}}}+ ([2\myN -3] + 1)\ev{\NB{\tikz[rotate = 90, scale= 0.7]{\begin{scope}[expand style =\stylevect]
  \draw[name path= p] ( 0, 0.7) .. controls +( 0.3, -0.3) and +( -0.3, -0.3) .. (0.5,0.7);
  \draw[name path= q] ( 0.5, -0.7) .. controls +(-0.3, 0.3) and +(0.3, 0.3) .. ( 0,-0.7);
\end{scope}
%%% Local Variables:
%%% mode: latex
%%% TeX-master: t
%%% End:
}}}.
    \label{eq:2sing}  
\end{equation}
\end{cor}

\begin{prop}
  \label{prop:square31}
   The following relation holds:
   \begin{equation} \label{eq:square31}
          \ev{\NB{\tikz[scale =0.8, rotate=0]{  \begin{scope}%[scale = 0.6]
    \coordinate (A) at ( 135:1) ;
    \coordinate (B) at (  45:1) ;
    \coordinate (C) at ( -45:1) ;
    \coordinate (D) at (-135:1) ;
    \coordinate (a) at ( 135:0.5) ;
    \coordinate (b) at (  45:0.5) ;
    \coordinate (c) at ( -45:0.5) ;
    \coordinate (d) at (-135:0.5) ;
    \draw[expand style =\stylevect] (A) --(a);
    \draw[expand style =\styleeven] (B) --(b) \arreven;
    \draw[expand style =\styleeven] (C) --(c) \rarreven;
    \draw[expand style =\stylevect] (D) --(d);  
    \draw[expand style =\styleodd] (a) -- (b) \arrodd;
    \draw[expand style =\stylevect] (b) -- (c); 
    \draw[expand style =\styleodd] (c) -- (d) \rarrodd;
    \draw[expand style =\styleeven] (d) -- (a) \arreven;
  \end{scope}

}}} = [\myN -2] \ev{\NB{\tikz[scale =0.8, rotate=0]{  \begin{scope}%[scale = 0.6]
    \coordinate (A) at ( 135:1) ;
    \coordinate (B) at (  45:1) ;
    \coordinate (C) at ( -45:1) ;
    \coordinate (D) at (-135:1) ;
    \coordinate (a) at ( 135:0.5) ;
    \coordinate (b) at (  45:0.5) ;
    \coordinate (c) at ( -45:0.5) ;
    \coordinate (d) at (-135:0.5) ;
    \coordinate (l) at (-90:0.4);
    \coordinate (t) at ( 90:0.4);
    
    \draw[expand style =\stylevect] (A) .. controls (135:0.7) .. (t);
    \draw[expand style =\styleeven] (B) .. controls (45:0.7) .. (t) \arreven;
    \draw[expand style =\styleeven] (C) .. controls (-45:0.7) .. (l) \rarreven;
    \draw[expand style =\stylevect] (D) .. controls (-135:0.7) .. (l);  
    \draw[expand style =\styleodd] (t) -- (l) \rarrodd;
  \end{scope}

}}} +
     \ev{\NB{\tikz[scale=0.8, rotate=0]{  \begin{scope}%[scale = 0.6]
    \coordinate (A) at ( 135:1) ;
    \coordinate (B) at (  45:1) ;
    \coordinate (C) at ( -45:1) ;
    \coordinate (D) at (-135:1) ;
    \draw[expand style =\stylevect] (A) .. controls (135:0.3) and (-135:0.3).. (D);
    \draw[expand style =\styleeven] (B) .. controls (45:0.3) and (-45:0.3) .. (C) \rarreven;
  \end{scope}

}}}.
   \end{equation}
\end{prop}

\begin{proof}
  Again, this is a careful inspection of what happens locally. Let us
  denote by $\web_\square$, $\web_I$, and $\web_{\ii}$ the webs in
  \eqref{eq:square31}, read from left to right. There are essentially
  two different cases: the one for which the two outer vectorial edges
  are colored by the same pigments (necessarily, one points inward,
  and the other one outward) or the case where these vectorial edges are colored by different pigments (orientations can be arbitrary). 

   We first consider the latter case when there are two different pigments. We only treat the case where the orientations are inward (the other cases are similar). Let us fix notations and let $i\neq j$ be two elements of $\setN$ and $X$ be an even subset of $\setN$ with $\{i,j\}\cap X = \emptyset$, one can consider colorings of $\web_\times$, $\web_H$ and $\web_=$, such that the restriction to the border of the distinguished pieces are given by:
   \[
\NB
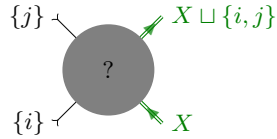
{\tikz[]{\input{\imagesfolder/td-square-31-col-1}}}
   \]

   The set of such colorings for $\web_{\ii}$ is clearly empty. There is a canonical $(\myN -2)$-to-one correspondence between colorings of $\web_\square$ and colorings of $\web_I$ satisfying this constraint.
   Let $(o_I,c_I)$ denote such a coloring of $\web_I$. On the distinguished piece of $\web_I$, $(o_I, c_I)$ is entirely determined by the coloring of the boundary:
   \begin{equation}
     (o_I,c_I) = \NB{\tikz[]{  \begin{scope}%[scale = 0.6]
    \coordinate (A) at ( 135:1) ;
    \coordinate (B) at (  45:1) ;
    \coordinate (C) at ( -45:1) ;
    \coordinate (D) at (-135:1) ;
    \coordinate (a) at ( 135:0.5) ;
    \coordinate (b) at (  45:0.5) ;
    \coordinate (c) at ( -45:0.5) ;
    \coordinate (d) at (-135:0.5) ;
    \coordinate (l) at (-90:0.4);
    \coordinate (t) at ( 90:0.4);
    \draw[expand style =\stylevect, >-] (A) .. controls (135:0.7) .. (t) node[pos=0, left] {$\{j\}$};
    \draw[expand style =\styleeven] (B) .. controls (45:0.7) .. (t) \arreven node[pos=0, right] {$X \cup \{i,j\}$};
    \draw[expand style =\styleeven] (C) .. controls (-45:0.7) .. (l) \rarreven node[pos=0, right] {$X$};
    \draw[expand style =\stylevect, >-] (D) .. controls (-135:0.7) .. (l) node[pos=0, left] {$\{i\}$};
    \draw[expand style =\styleodd] (t) -- (l) \rarrodd node[pos=0.5, right] {$X \cup \{i\}$};
  \end{scope}

}}
   \end{equation}
   The corresponding $\myN-2$ colorings of $\web_\square$ are given locally by:
   \begin{align}
     (o_k,c_k) &:=  \NB{\tikz[font=\tiny, scale =1.5]{\input{\imagesfolder/td-square-other-1-col-1}}}& \quad \text{for $k \in X$, and}  \\ (o_\ell, c_\ell)&:=  \NB{\tikz[font=\tiny, scale =1.5]{\input{\imagesfolder/td-square-other-1-col-2}}}&\quad  \text{for $\ell \in Y$,} 
   \end{align}
   where $Y= \setN\setminus (X \sqcup \{i,j\})$. One has:
   \begin{align}
     \deg(o_k, c_k) =& \deg(o_I, c_I) +\#X + \#Y_{<k} -\#Y_{>k} \quad \text{and} \\
     \deg(o_k, c_k) =& \deg(o_I, c_I) - \#Y - \#X_{<k} +\#X_{>k},
   \end{align}
   which implies that:
   \begin{align}
     \label{eq:18}
     \sum_{k \in X} q^{\deg(o_k,c_k)} + \sum_{\ell \in Y}q^{\deg(o_\ell, c_\ell)} &= q^{\deg(o_I,c_I)}(q^{\#X}[\#Y] + q^{\#Y}[\#X]) \\= &q^{\deg(o_I,c_I)}[\myN-2]
   \end{align}
   which is precisely what is needed.

   Now, we consider the case when the two outer vectorial edges are colored with the same pigments, so orientations are either both upward or both downward. We inspect the upward case, the other one is similar. Fixing notations, we let $i$ be an element of $\setN$ and $X$ be an even subset of $\setN$. One can consider colorings of $\web_\times$, $\web_H$ and $\web_=$, such that the restriction to the border of the distinguished pieces are given by:  
   \begin{equation}\label{eq:col-square31-2}
\NB{\tikz[]{\input{\imagesfolder/td-square-31-col-2}}}
\end{equation}
There are two subcases: either $i \in X$ or $i\notin X$. Let us first suppose that $i \in X$.  The set of such colorings for $\web_{I}$ is empty. There is a canonical one-to-one correspondance between colorings of $\web_\square$ and colorings of $\web_\ii$ satisfying this constraint, it is given below:
\begin{equation}
  \label{eq:30}
  (c_\square, o_\square) := \NB{\tikz[font=\tiny, scale =1.5]{\input{\imagesfolder/td-square-other-1-col-3}}}
\quad  \leftrightsquigarrow  \quad\NB{\tikz[font=\tiny, scale =1.5]{  \begin{scope}%[scale = 0.6]
    \coordinate (A) at ( 135:1) ;
    \coordinate (B) at (  45:1) ;
    \coordinate (C) at ( -45:1) ;
    \coordinate (D) at (-135:1) ;
    \draw[expand style =\stylevect, -<-] (A) .. controls (135:0.3) and (-135:0.3).. (D) node[left, pos=0.5] {$\{i\}$};
    \draw[expand style =\styleeven] (B) .. controls (45:0.3) and (-45:0.3) .. (C) \rarreven node[right, pos=0.5] {$X$};
  \end{scope}

}} =:(c_\ii, o_\ii),
\end{equation}
and one has $\deg(c_\square, o_\square)= \deg(c_\ii, o_\ii)$, so that
\begin{equation}
  \label{eq:31}
  q^{\deg(c_\square, o_\square)}= q^{\deg(c_\ii, o_\ii)},
\end{equation}
which is what is needed in this case.

Let us deal with the other subcase: colorings of the boundary are still given by \eqref{eq:col-square31-2} but now $i \notin X$. There is a canonical one-to-one correspondance between colorings of $\web_I$ and colorings of $\web_\ii$ satisfying this constraint:
\begin{equation}
  \label{eq:col-square-last}
  (c_I, o_I) := \NB{\tikz[font=\tiny, scale =0.8]{  \begin{scope}%[scale = 0.6]
    \coordinate (A) at ( 135:1) ;
    \coordinate (B) at (  45:1) ;
    \coordinate (C) at ( -45:1) ;
    \coordinate (D) at (-135:1) ;
    \coordinate (a) at ( 135:0.5) ;
    \coordinate (b) at (  45:0.5) ;
    \coordinate (c) at ( -45:0.5) ;
    \coordinate (d) at (-135:0.5) ;
    \coordinate (l) at (-90:0.4);
    \coordinate (t) at ( 90:0.4);
    \draw[expand style =\stylevect, <-] (A) .. controls (135:0.7) .. (t) node[pos=0, left] {$\{i\}$};
    \draw[expand style =\styleeven] (B) .. controls (45:0.7) .. (t) \arreven node[pos=0, right] {$X$};
    \draw[expand style =\styleeven] (C) .. controls (-45:0.7) .. (l) \rarreven node[pos=0, right] {$X$};
    \draw[expand style =\stylevect, >-] (D) .. controls (-135:0.7) .. (l) node[pos=0, left] {$\{i\}$};
    \draw[expand style =\styleodd] (t) -- (l) \rarrodd node[pos=0.5, right] {$X \cup \{i\}$};
  \end{scope}

}}
\quad  \leftrightsquigarrow  \quad\NB{\tikz[font=\tiny, scale =0.8]{}} =:(c_\ii, o_\ii).
\end{equation}
There is also a canonical $(\myN -1)$-to-one correspondance between colorings of $\web_\square$ and colorings of $\web_\ii$ satisfying this constraint. The colorings of $\web_\square$ corresponding to $(c_\ii, o_\ii)$ given in \eqref{eq:col-square-last} are given below:
\begin{align}
     (o_k,c_k) &:=  \NB{\tikz[font=\tiny, scale =1.5]{\input{\imagesfolder/td-square-other-1-col-4}}}& \quad \text{for $k \in X$, and}  \\ (o_\ell, c_\ell)&:=  \NB{\tikz[font=\tiny, scale =1.5]{\input{\imagesfolder/td-square-other-1-col-5}}}&\quad  \text{for $\ell \in Y$,} 
\end{align}
with $Y = \setN \setminus (X\sqcup\{i\})$.  One has for $k \in X$ and $\ell \in Y$: 
\begin{align}
  \deg(o_\ii, c_\ii) =& \deg(o_I, c_I) + \#Y -  \#X_{<i}+ \#X_{>i}, \\
  \deg(o_k, c_k) =& \deg(o_I, c_I) +\#Y - \mathrm{sign}(k-i) -\#X_{<k} +\#X_{>k} \\
  \deg(o_\ell, c_\ell) =& \deg(o_I, c_I) +1 -\#X +  \#Y_{<\ell} -\#Y_{>\ell} \end{align}
so that:
\begin{align}
  \label{eq:32}
  \sum_{k \in X} q^{(o_k,c_k)} + \sum_{\ell \in Y}q^{\deg(o_\ell, c_\ell)}&= q^{\deg(o_I, c_I)}(q^{1-\#X} [\#Y] + q^{\#Y}([\#X-1] + q^{\#X_{<i} -\#X_{>i}} ) \\ &= q^{\deg(o_I, c_I)}[\myN-1] + q^{\deg(o_\ii, c_\ii)}
\end{align}
which was precisely what we needed.
\end{proof}

\begin{cor}
  \label{cor:square31}
  The following relations hold:
     \begin{align} 
     \ev{\NB{\tikz[scale =0.8]{  \begin{scope}%[scale = 0.6]
    \coordinate (A) at ( 135:1) ;
    \coordinate (B) at (  45:1) ;
    \coordinate (C) at ( -45:1) ;
    \coordinate (D) at (-135:1) ;
    \coordinate (a) at ( 135:0.5) ;
    \coordinate (b) at (  45:0.5) ;
    \coordinate (c) at ( -45:0.5) ;
    \coordinate (d) at (-135:0.5) ;
    \draw[expand style =\stylevect] (A) --(a);
    \draw[expand style =\styleodd] (B) --(b) \arrodd;
    \draw[expand style =\styleodd] (C) --(c) \rarrodd;
    \draw[expand style =\stylevect] (D) --(d);  
    \draw[expand style =\styleeven] (a) -- (b) \arreven;
    \draw[expand style =\stylevect] (b) -- (c); 
    \draw[expand style =\styleeven] (c) -- (d) \rarreven;
    \draw[expand style =\styleodd] (d) -- (a) \arrodd;
  \end{scope}

}}} &= [\myN -2] \ev{\NB{\tikz[scale =0.8]{  \begin{scope}%[scale = 0.6]
    \coordinate (A) at ( 135:1) ;
    \coordinate (B) at (  45:1) ;
    \coordinate (C) at ( -45:1) ;
    \coordinate (D) at (-135:1) ;
    \coordinate (a) at ( 135:0.5) ;
    \coordinate (b) at (  45:0.5) ;
    \coordinate (c) at ( -45:0.5) ;
    \coordinate (d) at (-135:0.5) ;
    \coordinate (l) at (-90:0.4);
    \coordinate (t) at ( 90:0.4);
    
    \draw[expand style =\stylevect] (A) .. controls (135:0.7) .. (t);
    \draw[expand style =\styleodd] (B) .. controls (45:0.7) .. (t) \arrodd;
    \draw[expand style =\styleodd] (C) .. controls (-45:0.7) .. (l) \rarrodd;
    \draw[expand style =\stylevect] (D) .. controls (-135:0.7) .. (l);  
    \draw[expand style =\styleeven] (t) -- (l) \rarreven;
  \end{scope}

}}} +
     \ev{\NB{\tikz[scale=0.8]{  \begin{scope}%[scale = 0.6]
    \coordinate (A) at ( 135:1) ;
    \coordinate (B) at (  45:1) ;
    \coordinate (C) at ( -45:1) ;
    \coordinate (D) at (-135:1) ;
    \draw[expand style =\stylevect] (A) .. controls (135:0.3) and (-135:0.3).. (D);
    \draw[expand style =\styleodd] (B) .. controls (45:0.3) and (-45:0.3) .. (C) \rarrodd;
  \end{scope}

}}},\\
       \ev{\NB{\tikz[scale =0.8, rotate=180]{}}} &= [\myN -2] \ev{\NB{\tikz[scale =0.8, rotate=180]{}}} +
     \ev{\NB{\tikz[scale=0.8, rotate=180]{}}},\\
     \ev{\NB{\tikz[scale =0.8, rotate=180]{}}} &= [\myN -2] \ev{\NB{\tikz[scale =0.8, rotate=180]{}}} +
     \ev{\NB{\tikz[scale=0.8, rotate=180]{}}}.
     \end{align}
\end{cor}

 \begin{prop}\label{prop:pentagon}
   The following relation holds:
   \begin{equation} \label{eq:pentagon}
     \ev{\NB{\tikz[scale =0.8]{\input{\imagesfolder/td-pentagone-0}}}} = \ev{\NB{\tikz[scale =0.8]{\input{\imagesfolder/td-pentagone-1}}}} + [\myN -3 ] \ev{\NB{\tikz[scale =0.8]{\input{\imagesfolder/td-pentagone-2}}}}.
   \end{equation}
 \end{prop}

 \begin{proof}[Sketch of the proof.]
   This is once again a careful inspection of what happens locally. There are essentially three different cases:
   \begin{enumerate}
   \item \label{it:pent1col} The three outer vectorial edges are colored by the same pigment (necessarily with alternating orientations).
   \item \label{it:pent2col} Two of the three edges outer edges are colored by the same pigment (necessarily with opposite orientations), the third one is colored by another pigment.  
   \item \label{it:pent3col} The three edges are colored by three different pigments. 
   \end{enumerate}

   Instead of giving a full proof (which would go along the same lines as for the previous propositions), we'll inspect for each case an example of a coloring of the boundary which corresponds to that case. Let us denote $\web_P$, $\web_X$ and $\web_T$ (for pentagon, crossing and tree) the webs in \eqref{eq:pentagon}, read from left to right.

   \subsection*{Case (\ref{it:pent1col}):}
   Let us suppose that the boundary is colored as follows:
   \[
     \NB
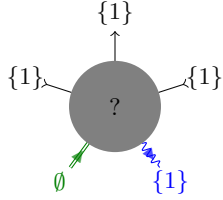
{\tikz[]{\input{\imagesfolder/td-pentagone-col-0}}}
   \]
   The compatible colorings with $\web_P$, $\web_X$ and $\web_T$ are listed below:
   \begin{equation} \label{eq:pentagon-col}
     \NB{\tikz[]{\input{\imagesfolder/td-pentagone-col-i}}}\qquad
     \NB{\tikz[]{\input{\imagesfolder/td-pentagone-col-pm}}}\qquad
     \NB{\tikz[]{\input{\imagesfolder/td-pentagone-col-T}}}\qquad 
   \end{equation}
   for $i \in \setN \setminus\{1\}$ and $\epsilon \in \{+, -\}$.  Let
   $(c_{i})_{i \in \setN\setminus\{1\}}$, $c_+$, $c_-$ and $c_T$ be
   colorings of $\web_P$, $\web_{X}$ and $\web_T$ which are identical
   except where these webs differ and they are given by
   \eqref{eq:pentagon-col}. A careful inspection gives:
\begin{align}
  \deg(c_i)&= 2+ (\myN-i) -(i-2) + \deg(c_T), \\
  \deg(c_+)&= (\myN-1)  -1 + \deg(c_T) \quad \text{and} \\
  \deg(c_-)&= (\myN-1)  +1 + \deg(c_T).
\end{align}
Therefore, one has
\begin{equation}
  \label{eq:26}
  \sum_{i\in \setN \setminus\{1\}} q^{\deg(c_i)} = q^{\deg(c_T) + 2}
  [\myN-1] 
\end{equation}
{and}
\begin{align}
  q^{\deg(c_T)}[\myN-3] + q^{\deg(c_+)} + q^{\deg(c_-)}&= q^{\deg(c_T)}[\myN-3] + q^{\deg(c_T)+\myN-1}[2] \\ &= q^{\deg(c_T) + 2} [\myN-1],
\end{align}
so that:
\begin{equation}
  \sum_{i\in \setN \setminus\{1\}} q^{\deg(c_i)} = q^{\deg(c_T)}[\myN-3] + q^{\deg(c_+)} + q^{\deg(c_-)}.\label{eq:27}
\end{equation}
which is precisely the incarnation of \eqref{eq:pentagon} on that example.

\subsection*{Case (\ref{it:pent2col}):}

   Let us suppose that the boundary is colored as follows:
   \[
     \NB
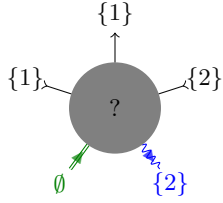
{\tikz[]{\input{\imagesfolder/td-pentagone-col-1}}}
   \]
   The compatible colorings with $\web_P$, $\web_X$ and $\web_T$ are listed below:
   \begin{equation} \label{eq:pentagon-col-1}
     \NB
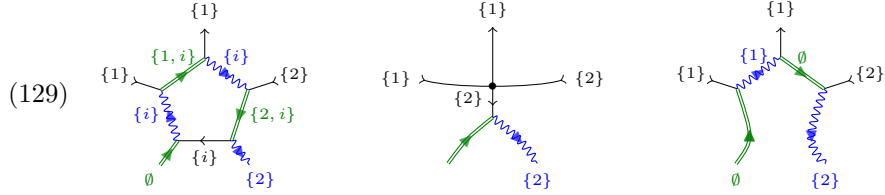
{\tikz[]{\input{\imagesfolder/td-pentagone-col-i-1}}}\qquad
     \NB{\tikz[]{\input{\imagesfolder/td-pentagone-col-pm-1}}}\qquad
     \NB{\tikz[]{\input{\imagesfolder/td-pentagone-col-T-1}}}\qquad 
   \end{equation}
   for $i \in \setN \setminus\{1,2\}$
   Let $(c_{i})_{i \in \setN\setminus\{1,2\}}$, $c_X$ and $c_T$ be colorings of $\web_P$, $\web_{X}$ and $\web_T$ which are identical except where these webs differ and there are given by \eqref{eq:pentagon-col-1}. A careful inspection gives:
\begin{align*}
  \deg(c_i)&= 1+ (\myN-i) -(i-3) + \deg(c_T), \quad \text{and} \\
  \deg(c_X)&= (\myN-2) + \deg(c_T). \\
\end{align*}
Therefore, one has
\begin{align}
  \label{eq:29}
  \sum_{i\in \setN \setminus\{1\}} q^{\deg(c_i)} &= q^{\deg(c_T) + 1} [\myN-2] \quad\text{and} \\
  q^{\deg(c_T)}[\myN-3] + q^{\deg(c_X)}&= q^{\deg(c_T)}[\myN-3] + q^{\deg(c_T)+\myN-2} \\ &= q^{\deg(c_T) + 1} [\myN-2],
\end{align}
so that:
\begin{equation}
  \sum_{i\in \setN \setminus\{1\}} q^{\deg(c_i)} = q^{\deg(c_T)}[\myN-3] + q^{\deg(c_X)} .\label{eq:27}
\end{equation}
which is precisely the incarnation of \eqref{eq:pentagon} on that example.

\subsection*{Case (\ref{it:pent3col}):}

   Let us suppose that the boundary is colored as follows:
   \[
     \NB
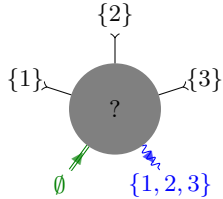
{\tikz[]{\input{\imagesfolder/td-pentagone-col-2}}}
   \]

   There is no compatible coloring for $\web_X$. The compatible colorings with $\web_P$, and $\web_T$ are listed below:
   \begin{equation} \label{eq:pentagon-col-2}
     \NB{\tikz[]{\input{\imagesfolder/td-pentagone-col-i-2}}}\qquad
     \NB{\tikz[]{\input{\imagesfolder/td-pentagone-col-T-2}}}\qquad 
   \end{equation}
   for $i \in \setN \setminus\{1,2\}$
   Let $(c_{i})_{i \in \setN\setminus\{1,2,3\}}$ and $c_T$ be colorings of $\web_P$, and $\web_T$ which are identical except where these webs differ and there are given by \eqref{eq:pentagon-col-2}. A careful inspection gives:
\begin{align*}
  \deg(c_i)&=  (\myN-i) -(i-4) + \deg(c_T), \\
\end{align*}
so that:
\begin{equation}
  \sum_{i\in \setN \setminus\{1\}} q^{\deg(c_i)} = q^{\deg(c_T)}[\myN-3] .\label{eq:27}
\end{equation}
which is precisely the incarnation of \eqref{eq:pentagon} on that example.
 \end{proof}
 
\begin{cor}
  \label{cor:square31}
  The following relations hold:
  \begin{align}
         \ev{\NB{\tikz[scale =0.8]{\input{\imagesfolder/td-pentagone-0-1}}}} &= \ev{\NB{\tikz[scale =0.8]{\input{\imagesfolder/td-pentagone-1-1}}}} + [\myN -3 ] \ev{\NB{\tikz[scale =0.8]{\input{\imagesfolder/td-pentagone-2-1}}}}, \\
    \ev{\NB{\tikz[scale =0.8]{\input{\imagesfolder/td-pentagone-0-2}}}} &= \ev{\NB{\tikz[scale =0.8]{\input{\imagesfolder/td-pentagone-1-2}}}} + [\myN -3 ] \ev{\NB{\tikz[scale =0.8]{\input{\imagesfolder/td-pentagone-2-2}}}}, \\
             \ev{\NB{\tikz[scale =0.8]{\input{\imagesfolder/td-pentagone-0-3}}}} &= \ev{\NB{\tikz[scale =0.8]{\input{\imagesfolder/td-pentagone-1-3}}}} + [\myN -3 ] \ev{\NB{\tikz[scale =0.8]{\input{\imagesfolder/td-pentagone-2-3}}}}.
  \end{align}
\end{cor}

\section{Link invariants}
\label{sec:linkinvariants}
In this section we associate Laurent polynomials to (some) diagrams of
knotted webs and prove that this quantity is invariant under certain
moves. In particular we will be able to define an invariant for
framed unoriented links, see Theorem~\ref{thm:link-invariant}. 

In this section, $\myN$ is an integer greater than or equal to
$3$. This restriction arises because we need to use singular vertices
which behavior is well-behaved when $\myN\geq 3$ (see
Lemma~\ref{lem:singular-square}).

Although our construction is novel, the resulting link invariant is not new. We will identify this invariant as the $\sotN$
quantum invariant associated with the vector representation (with
$q$ changed for $-q$) in Section~\ref{sec:intertwiners}.

We will not consider all (reasonable) diagrams of knotted webs: we exclude the spin-spin
crossings for simplicity. Let us make this precise by defining what
a diagram means:

\begin{dfn}
  A \emph{diagram} is a graph embedded in the plane with the same type
  of edges as webs (vectorial, even-spin and odd-spin) and for which
  each vertex has degree 3 or 4 and follows one of the 10 following
  local models:
  \begin{itemize}
  \item if the vertex has degree 3, it follows one of the four local
    models for vertices of webs (see \eqref{eq:local-models-webs}).
  \item of the vertex has degree 4, it is called a \emph{crossing} and
    follows one of the 6 following local models:
    \begin{equation}
      \label{eq:local-models-diags}
      \NB{\tikz[scale=0.6]{}}\quad\qquad
      \NB{\tikz[scale=0.6]{}}\qquad\quad
      \NB{\tikz[scale=0.6]{\begin{scope}[scale=0.6]
  \draw[expand style =\styleeven] (45:1) -- (-135:1);
  \draw[expand style =\styleeven] (45:0.9) -- (45:1) \arreven;
  \draw[expand style =\styleeven] (-135:1) -- (-135:0.9) \arreven;
  \fill[white] (0,0) circle (0.2cm);
  \draw[expand style =\stylevect] (-45:1) -- (135:1);
  %\node at (0,0) {\singvertex};
\end{scope}}}\quad\qquad
      \NB{\tikz[scale=0.6]{\begin{scope}[scale=0.6]
  \draw[expand style =\stylevect] (-45:1) -- (135:1);
  \fill[white] (0,0) circle (0.2cm);
  \draw[expand style =\styleeven] (45:1) -- (-135:1);
  \draw[expand style =\styleeven] (45:0.9) -- (45:1) \arreven;
  \draw[expand style =\styleeven] (-135:1) -- (-135:0.9) \arreven;
\end{scope}}}\quad\qquad
      \NB{\tikz[scale=0.6]{\begin{scope}[scale=0.6]
  \draw[expand style =\styleodd] (45:1) -- (-135:1);
  \draw[expand style =\styleodd] (45:0.9) -- (45:1) \arrodd;
  \draw[expand style =\styleodd] (-135:1) -- (-135:0.9) \arrodd;
  \fill[white] (0,0) circle (0.2cm);
  \draw[expand style =\stylevect] (-45:1) -- (135:1);
\end{scope}}}\quad\qquad
      \NB{\tikz[scale=0.6]{\begin{scope}[scale=0.6]
  \draw[expand style =\stylevect] (-45:1) -- (135:1);
  \fill[white] (0,0) circle (0.2cm);
  \draw[expand style =\styleodd] (45:1) -- (-135:1);
  \draw[expand style =\styleodd] (45:0.9) -- (45:1) \arrodd;
  \draw[expand style =\styleodd] (-135:1) -- (-135:0.9) \arrodd;
\end{scope}}}      
    \end{equation}
    The first one is called a \emph{singular crossing} and was already present in the definition of \gweb{}, the second one is called \emph{regular crossing} and the other ones \emph{mixed crossings}.   
  \end{itemize}
\end{dfn}

\begin{rmk}
    The local models for crossings in equation \eqref{eq:local-models-diags} are all considered up to rotation, so do not need to be read from the bottom to the top. For example, a regular crossing may also appear as an over-crossing as in the left hand side of \eqref{eq:dfn:xingVV}.
\end{rmk}

\begin{exa}\label{exa:link-diag}
  Unoriented link diagrams, where link components are thought of as vertex-less closed vectorial edges are an example of diagrams. As we shall see, one can think of link diagrams as representing framed links (with the blackboard framing). 
\end{exa}

We already know how to evaluate \gweb{}s. The evaluation of diagrams is given in a skein theoretical way by expressing the evaluation of diagrams with regular and/or mixed crossings as $\ZZ[q, q^{-1}]$ linear combinations of evaluation of webs.

\begin{dfn}\label{dfn:Xings}
  We extend $\ev{\cdot}$ on the set of diagrams as follows: \begin{align}
\label{eq:dfn:xingVV}    \ev{\,\,\NB{\tikz[rotate=90]{}}\,\,} &= -q\ev{\,\,\NB{\tikz[]{}}\,\,} + \ev{\,\,\NB{\tikz[]{}}\,\,} - q^{-1}\ev{\,\,\NB{\tikz[rotate=90]{}}\,\,}, \\
\ev{\NB{\tikz[]{}}} &= q^{1/2}\ev{\NB{\tikz[]{\input{\imagesfolder/td-Hing-even-1}}}} -  q^{-1/2}\ev{\NB{\tikz[]{\input{\imagesfolder/td-Hing-even-0}}}}, \\
    \ev{\NB{\tikz[]{}}} &= - q^{1/2}\ev{\NB{\tikz[]{\input{\imagesfolder/td-Hing-even-0}}}} +  q^{-1/2}\ev{\NB{\tikz[]{\input{\imagesfolder/td-Hing-even-1}}}}, \\
    \ev{\NB{\tikz[]{}}} &= -q^{1/2}\ev{\NB{\tikz[]{\input{\imagesfolder/td-Hing-odd-1}}}} +  q^{-1/2}\ev{\NB{\tikz[]{\input{\imagesfolder/td-Hing-odd-0}}}}, \\    
    \ev{\NB{\tikz[]{}}} &= q^{1/2}\ev{\NB{\tikz[]{\input{\imagesfolder/td-Hing-odd-0}}}} -  q^{-1/2}\ev{\NB{\tikz[]{\input{\imagesfolder/td-Hing-odd-1}}}}.
  \end{align}
  The map $\ev{\cdot}$ takes value in $\ZZ[q^{1/2}, q^{-1/2}]$.
\end{dfn}

\begin{rmk}\label{rmk:value-knot-inv}
 \begin{enumerate} 
  \item \label{it:inv-Laurent}The map $\ev{\cdot}$ on unoriented link diagrams takes value in $\ZZ[q, q^{-1} ]$.

  \item \label{it:mirror-image}  If $\diag$ is a diagram and $\overline{\diag}$ is the mirror image of this diagram, then $\ev{\overline{\diag}} = \ev{{\diag}}_{q^{1/2}\mapsto q^{-1/2}}$.
  \end{enumerate}
\end{rmk}

\begin{prop}\label{prop:link-invariant}
  If $\diag$ and $\diag'$ are two link diagrams representing the same framed unoriented link, then
  $\ev{\diag} = \ev{\diag'}$.
\end{prop}

Thanks to Reidemeister's theorem 
, the proof of this statement follows immediately from Lemmas~\ref{lem:R1}, \ref{lem:R2} and \ref{lem:R3}.

\begin{lem}\label{lem:R1}
  The following relations hold:
  \begin{equation}
    \label{eq:R1s}
    \ev{\NB{\tikz[scale=0.5]{}}} = -q^{2\myN -1} \ev{\,\NB{\tikz[scale=0.5]{}}\,} \qquad \qquad  \text{and} \qquad \qquad
    \ev{\NB{\tikz[scale=0.5, yscale=-1]{}}} = -q^{1-2\myN} \ev{\,\NB{\tikz[scale=0.5]{}}\,}. 
  \end{equation}
\end{lem}

\begin{proof}
  This is a basic computation, we only treat the first identity, the second follows by Remark \ref{rmk:value-knot-inv}.(\ref{it:mirror-image}).
  \begin{align}
    \ev{\NB{\tikz[scale=0.5]{}}}&= -q\ev{\NB{\tikz[scale=0.5]{\begin{scope}[expand style = \stylevect]
  \coordinate (a) at (-0.3,0.3);
  \coordinate (b) at (0.3,0.3);
  \coordinate (c) at (0.3,-0.3);
  \coordinate (d) at (-0.3,-0.3);
  \coordinate (O) at (0,0);
  \coordinate (T) at (-0.3, 0.9);
  \coordinate (B) at (-0.3,-0.9);
  \draw (T) .. controls +(0, -0.2) and +(-0.2, 0.2) .. (a);
  \draw (B) .. controls +(0, +0.2) and +(-0.2, -0.2) .. (d);  
  \draw (a) .. controls (O) .. (d);
  \draw (c) .. controls (O) .. (b);
  %\node at (O) {\singvertex};
  \draw (b) .. controls +(0.4, 0.4) and +(0.4, -0.4) .. (c);
\end{scope}}}} + \ev{\NB{\tikz[scale=0.5]{\begin{scope}[expand style = \stylevect]
  \coordinate (a) at (-0.3,0.3);
  \coordinate (b) at (0.3,0.3);
  \coordinate (c) at (0.3,-0.3);
  \coordinate (d) at (-0.3,-0.3);
  \coordinate (O) at (0,0);
  \coordinate (T) at (-0.3, 0.9);
  \coordinate (B) at (-0.3,-0.9);
  \draw (T) .. controls +(0, -0.2) and +(-0.2, 0.2) .. (a);
  \draw (B) .. controls +(0, +0.2) and +(-0.2, -0.2) .. (d);  
  \draw (a) -- (c);
  \node at (O) {\singvertex};
  \draw (b) -- (d);
  \draw (b) .. controls +(0.4, 0.4) and +(0.4, -0.4) .. (c);
\end{scope}}}} - q^{-1} \ev{\NB{\tikz[scale=0.5]{\begin{scope}[expand style = \stylevect]
  \coordinate (a) at (-0.3,0.3);
  \coordinate (b) at (0.3,0.3);
  \coordinate (c) at (0.3,-0.3);
  \coordinate (d) at (-0.3,-0.3);
  \coordinate (O) at (0,0);
  \coordinate (T) at (-0.3, 0.9);
  \coordinate (B) at (-0.3,-0.9);
  \draw (T) .. controls +(0, -0.2) and +(-0.2, 0.2) .. (a);
  \draw (B) .. controls +(0, +0.2) and +(-0.2, -0.2) .. (d);  
  \draw (a) .. controls (O) .. (b);
  \draw (c) .. controls (O) .. (d);
  %\node at (O) {\singvertex};
  \draw (b) .. controls +(0.4, 0.4) and +(0.4, -0.4) .. (c);
\end{scope}}}}  \\
                                    &= \left(-q([2\myN - 1] + 1) + ([2\myN -2] +[2]) - q^{-1}) \right)\ev{\NB{\tikz[scale=0.5]{}}}\\
                                    &=-q^{2N-1} \ev{\NB{\tikz[scale=0.5]{}}}.                                        \end{align}
  The first line follows from the very definition of a crossing, the second from 
  Proposition~\ref{prop:disjoint-union} and Corollary \ref{cor:curl-sing}. 
\end{proof}

\begin{lem}\label{lem:R2}  \newcommand{\ccc}{
  \coordinate (o) at (0.5,0.5);
  \coordinate (a) at (0,1);
  \coordinate (b) at (1,1);
  \coordinate (c) at (1,0);
  \coordinate (d) at (0,0);
}

\newcommand{\dd}[2][]{
      \NB{\tikz[#1]{
      \begin{scope}[expand style = \stylevect]
        \begin{scope}[xscale=2]
          \ccc
        \end{scope}
        #2
        \end{scope}}}}
        
\newcommand{\ddd}[3][]{
  \NB{\tikz[#1]{
      \begin{scope}[expand style = \stylevect]
      \ccc
      #2
    \begin{scope}[xshift =1cm]
      \ccc
      #3
    \end{scope}
  \end{scope}}}}

\newcommand{\xing}{
  \draw (a) .. controls +(0.3, 0) and +(-0.3, 0) ..  (c);
  \fill[white] (o) circle (1mm);
  \draw (d) .. controls +(0.3, 0) and +(-0.3, 0) .. (b);
}

\newcommand{\xingm}{
  \draw (d) .. controls +(0.3, 0) and +(-0.3, 0) .. (b);
  \fill[white] (o) circle (1mm);
  \draw (a) .. controls +(0.3, 0) and +(-0.3, 0) .. (c);
}
\newcommand{\sing}{
  \draw (d) .. controls +(0.3, 0) and +(-0.3, 0) .. (b);
  \node at (o) {\singvertex};
  \draw (a) .. controls +(0.3, 0) and +(-0.3, 0) .. (c);
}
\newcommand{\horz}{
  \draw (a) -- (b);
  \draw (d) -- (c);
}
\newcommand{\vertz}{
  \draw (c) .. controls +(-0.3, 0) and +(-0.3, 0) .. (b);
  \draw (a) .. controls +(0.3, 0) and +(0.3, 0) .. (d);
}

   The following relation holds:
  \begin{equation}
    \label{eq:35}
    \ev{\ddd[scale=0.5]{\xing}{\xingm}} =
    \ev{\ddd[scale=0.5]{\horz}{\horz}}.
  \end{equation}  
\end{lem}

\begin{proof}
  As before, this is  a simple computation.
  \newcommand{\ccc}{
  \coordinate (o) at (0.5,0.5);
  \coordinate (a) at (0,1);
  \coordinate (b) at (1,1);
  \coordinate (c) at (1,0);
  \coordinate (d) at (0,0);
}

\newcommand{\dd}[2][]{
      \NB{\tikz[#1]{
      \begin{scope}[expand style = \stylevect]
        \begin{scope}[xscale=2]
          \ccc
        \end{scope}
        #2
        \end{scope}}}}
        
\newcommand{\ddd}[3][]{
  \NB{\tikz[#1]{
      \begin{scope}[expand style = \stylevect]
      \ccc
      #2
    \begin{scope}[xshift =1cm]
      \ccc
      #3
    \end{scope}
  \end{scope}}}}

\newcommand{\xing}{
  \draw (a) .. controls +(0.3, 0) and +(-0.3, 0) ..  (c);
  \fill[white] (o) circle (1mm);
  \draw (d) .. controls +(0.3, 0) and +(-0.3, 0) .. (b);
}

\newcommand{\xingm}{
  \draw (d) .. controls +(0.3, 0) and +(-0.3, 0) .. (b);
  \fill[white] (o) circle (1mm);
  \draw (a) .. controls +(0.3, 0) and +(-0.3, 0) .. (c);
}
\newcommand{\sing}{
  \draw (d) .. controls +(0.3, 0) and +(-0.3, 0) .. (b);
  \node at (o) {\singvertex};
  \draw (a) .. controls +(0.3, 0) and +(-0.3, 0) .. (c);
}
\newcommand{\horz}{
  \draw (a) -- (b);
  \draw (d) -- (c);
}
\newcommand{\vertz}{
  \draw (c) .. controls +(-0.3, 0) and +(-0.3, 0) .. (b);
  \draw (a) .. controls +(0.3, 0) and +(0.3, 0) .. (d);
}

   \begin{align}
    \ev{\ddd[scale=0.5]{\xing}{\xingm}} =& q^2\ev{\ddd[scale=0.5]{\vertz}{\horz}}- q \ev{\ddd[scale=0.5]{\sing}{\horz}} + \ev{\ddd[scale=0.5]{\horz}{\horz}} \\
                                         & -q \ev{\ddd[scale=0.5]{\vertz}{\sing}}+ \ev{\ddd[scale=0.5]{\sing}{\sing}} - q^{-1}\ev{\ddd[scale=0.5]{\horz}{\sing}} \\
                                         & + \ev{\ddd[scale=0.5]{\vertz}{\vertz}}- q^{-1} \ev{\ddd[scale=0.5]{\sing}{\vertz}} + q^{-2}\ev{\ddd[scale=0.5]{\horz}{\vertz}} \\
    =&\left(\scriptstyle{q^2 -[2]([2\myN-2] + [2]) + [2\myN -1] +1 + [2\myN-3] +1  +q^{-2}} \right) \ev{\dd[scale=0.5]{\vertz}} \\
                                         & +\left( -q  +[2] -q^{-1}\right)  \ev{\dd[scale=0.5]{\sing}} + \ev{\dd[scale=0.5]{\horz}} \\
    =&\ev{\dd[scale=0.5]{\horz}}.\qedhere
  \end{align}
\end{proof}

\begin{lem}\label{lem:PF}
  The following relations hold:
  \begin{equation}
    \label{eq:PF}
    \ev{\NB{\tikz[scale=0.7]{\input{\imagesfolder/td-pf1}}}} = \ev{\NB{\tikz[scale=0.7]{\input{\imagesfolder/td-pf2}}}}
    \qquad \text{and} \qquad
    \ev{\NB{\tikz[scale=0.7]{\input{\imagesfolder/td-pf3}}}} = \ev{\NB{\tikz[scale=0.7]{\input{\imagesfolder/td-pf4}}}}
    \end{equation}
  \end{lem}

\begin{proof}
  This is a simple computation, we only treat the first identity, the second follows by Remark \ref{rmk:value-knot-inv}.(\ref{it:mirror-image}).
    \newcommand{\ccc}[1]{
  \begin{scope}[xscale=#1]
  \foreach \i in {18, 90, 162, 234, 306}{
    \coordinate (A\i) at (\i:1);
    \coordinate (a\i) at (\i:0.8);
  }
  \coordinate (b90) at (270:-0.3);
  \coordinate (b270) at (270:0.3);
    \coordinate (s) at (-0.1,0);
  \begin{scope}[shift = {(-0.45, -0.077)}, rotate=15]
    \coordinate (l0) at (0,0);
    \foreach \j in {90, 180, 270, 360}{
      \coordinate (l\j) at (\j:0.21);}
  \end{scope}    
  \end{scope}
  }
  \newcommand{\smootha}[2]{
  \draw[expand style= \stylevect] (A162) .. controls (a162) .. (l90);
  \draw[expand style= \stylevect] (b270) .. controls + (s) .. (l270);
  \draw[expand style= #1] (b90) -- (l90) \arr;
  \draw[expand style= #1] (A234) .. controls (a234) .. (l270) \arr ;
  \draw[expand style= #2] (l90) -- (l270) \rarr;
  }
\newcommand{\smoothb}[2]{
    \draw[expand style= \stylevect] (A162) .. controls (a162) .. (l180);
  \draw[expand style= \stylevect] (b270) .. controls + (s) .. (l360);
  \draw[expand style= #1] (b90) -- (l360) \arr;
  \draw[expand style= #1] (A234) .. controls (a234) .. (l180) \rarr ;
  \draw[expand style= #2] (l180) -- (l360) \arr;
}
\newcommand{\ddd}[3][]{\NB{\tikz[#1]{
\begin{scope}
  \ccc{1}
  \draw[expand style= \stylevect] (A90) -- (b90);
  #2{\styleeven}{\styleodd}
  \ccc{-1}
  #3{\styleodd}{\styleeven}
\end{scope}}}}

   \begin{align}
    \ev{\NB{\tikz[scale=0.7]{\input{\imagesfolder/td-pf2}}}} =& -q\ev{\ddd[scale=0.7]{\smootha}{\smoothb}} + \ev{\ddd[scale=0.7]{\smootha}{\smootha}} \\ &+ \ev{\ddd[scale=0.7]{\smoothb}{\smoothb}} - q^{-1}\ev{\ddd[scale=0.7]{\smoothb}{\smootha}} \\
    =& {-\left(q[\myN -2] +[\myN-3] +[\myN-1] -q^{-1}[\myN-2]\right)}\ev{\NB{\tikz[scale=0.7]{\input{\imagesfolder/td-pentagone-2}}}} \\ &- q\ev{\NB{\tikz[scale=0.7]{\input{\imagesfolder/td-pf5}}}} - q^{-1}\ev{\NB{\tikz[scale=0.7]{\input{\imagesfolder/td-pf6}}}} + q \ev{\NB{\tikz[scale=0.7]{\input{\imagesfolder/td-pentagone-1}}}} \\
=& -q\ev{\NB{\tikz[scale=0.7]{\input{\imagesfolder/td-pf5}}}} + \ev{\NB{\tikz[scale=0.7]{\input{\imagesfolder/td-pentagone-1}}}} - q^{-1}\ev{\NB{\tikz[scale=0.7]{\input{\imagesfolder/td-pf6}}}}\\
    =& \ev{\NB{\tikz[scale=0.7]{\input{\imagesfolder/td-pf1}}}}.\qedhere
  \end{align}
\end{proof}
Similarly, one has:

\begin{lem}\label{lem:PF2}
  The following relations hold:
  \begin{align}
    \label{eq:PF2}
    \ev{\NB{\tikz[scale=0.7]{\input{\imagesfolder/td-pf1-2}}}} = \ev{\NB{\tikz[scale=0.7]{\input{\imagesfolder/td-pf2-2}}}},
    \qquad  \qquad
    \ev{\NB{\tikz[scale=0.7]{\input{\imagesfolder/td-pf3-2}}}} = \ev{\NB{\tikz[scale=0.7]{\input{\imagesfolder/td-pf4-2}}}}, \\
    \ev{\NB{\tikz[scale=0.7]{\input{\imagesfolder/td-pf1-3}}}} = \ev{\NB{\tikz[scale=0.7]{\input{\imagesfolder/td-pf2-3}}}},
    \qquad  \qquad
    \ev{\NB{\tikz[scale=0.7]{\input{\imagesfolder/td-pf3-3}}}} = \ev{\NB{\tikz[scale=0.7]{\input{\imagesfolder/td-pf4-3}}}}, \\
    \ev{\NB{\tikz[scale=0.7]{\input{\imagesfolder/td-pf1-4}}}} = \ev{\NB{\tikz[scale=0.7]{\input{\imagesfolder/td-pf2-4}}}},
    \qquad  \qquad
    \ev{\NB{\tikz[scale=0.7]{\input{\imagesfolder/td-pf3-4}}}} = \ev{\NB{\tikz[scale=0.7]{\input{\imagesfolder/td-pf4-4}}}}.
    \end{align}
  \end{lem}

\begin{lem}\label{lem:R2-spin-vect}
  \newcommand{\ccc}[1]{
\begin{scope}[xscale=#1]
  \coordinate (A) at (-1, 0.5);
  \coordinate (B) at (0, 0.5);
  \coordinate (C) at (0, -0.5);
  \coordinate (D) at (-1, -0.5);
  \coordinate (O) at (-0.5, 0);
  \coordinate (t) at (-0.5,0.25);
  \coordinate (b) at (-0.5,-0.25);
  \coordinate (l) at (-0.75,0);
  \coordinate (r) at (-0.25,0);
  \coordinate (sl) at (-0.25,0);
  \coordinate (sr) at ( 0.25,0);
\end{scope}}

\newcommand{\unknotted}[2]{
\begin{scope} 
  \draw[expand style = #1] (A) -- (B) \arr;
  \draw[expand style = \stylevect] (D) -- (C);
\end{scope}}

\newcommand{\xing}[2]{
\begin{scope}
  \draw[expand style = #1] (A) .. controls +(sr) and +(0,0) .. (O) \arr  .. controls +(0,0) and +(sl) .. (C) \arr;
  \fill[white] (O) circle (1mm);
  \draw[expand style =\stylevect] (D) .. controls +(sr) and +(sl) .. (B);
\end{scope}}

\newcommand{\xingm}[2]{
  \begin{scope}
  \draw[expand style =\stylevect] (D) .. controls +(sr) and +(sl) .. (B);
  \fill[white] (O) circle (1mm);
  \draw[expand style = #1] (A) .. controls +(sr) and +(sl) .. (C) \arr ;
\end{scope}
}

\newcommand{\III}[3][\rarr]{
\begin{scope}
  \draw[expand style = #2] (A) .. controls +(sr) and +(0,0) .. (t) \arr;
  \draw[expand style = #2] (C) .. controls +(sl) and +(0,0) .. (b) \arr;
  \draw[expand style = \stylevect] (D) .. controls +(sr) and +(0,0) .. (b) ;
  \draw[expand style = \stylevect] (B) .. controls +(sl) and +(0,0) .. (t) ;
  \draw[expand style = #3] (b) -- (t) #1;
\end{scope}}

\newcommand{\IIIr}[2]{\III[\arr]{#1}{#2}}

\newcommand{\HHH}[2]{
\begin{scope}
  \draw[expand style = #1] (A) .. controls +(sr) and +(0,0) .. (l) \arr;
  \draw[expand style = #1] (C) .. controls +(sl) and +(0,0) .. (r) \arr;
  \draw[expand style = \stylevect] (D) .. controls +(sr) and +(0,0) .. (l) ;
  \draw[expand style = \stylevect] (B) .. controls +(sl) and +(0,0) .. (r) ;
  \draw[expand style = #2] (l) -- (r) \arr;
\end{scope}}

\newcommand{\ddd}[3][]{\NB{\tikz[#1]{
\begin{scope}
  \ccc{1}
  #2{\styleeven}{\styleodd}
  \ccc{-1}
  #3{\styleeven}{\styleodd}
\end{scope}}}}

\newcommand{\dddodd}[3][]{\NB{\tikz[#1]{
\begin{scope}
  \ccc{1}
  #2{\styleodd}{\styleodd}
  \ccc{-1}
  #3{\styleodd}{\styleodd}
\end{scope}}}}

\newcommand{\dd}[2][]{\NB{\tikz[#1]{
\begin{scope}
  \ccc{2}
  #2{\styleeven}{\styleodd}
\end{scope}}}}

\newcommand{\ddodd}[2][]{\NB{\tikz[#1]{
\begin{scope}
  \ccc{2}
  #2{\styleodd}{}
\end{scope}}}}

   The following relations hold:
  \begin{align}
    \label{eq:37}
    &\ddd[scale=0.7]{\xingm}{\xingm}=\dd[scale=0.7]{\unknotted}=\ddd[scale=0.7]{\xing}{\xing}\\
    &\dddodd[scale=0.7]{\xingm}{\xingm}=\ddodd[scale=0.7]{\unknotted}=\dddodd[scale=0.7]{\xing}{\xing}
  \end{align}
\end{lem}

\begin{proof}
  \newcommand{\ccc}[1]{
\begin{scope}[xscale=#1]
  \coordinate (A) at (-1, 0.5);
  \coordinate (B) at (0, 0.5);
  \coordinate (C) at (0, -0.5);
  \coordinate (D) at (-1, -0.5);
  \coordinate (O) at (-0.5, 0);
  \coordinate (t) at (-0.5,0.25);
  \coordinate (b) at (-0.5,-0.25);
  \coordinate (l) at (-0.75,0);
  \coordinate (r) at (-0.25,0);
  \coordinate (sl) at (-0.25,0);
  \coordinate (sr) at ( 0.25,0);
\end{scope}}

\newcommand{\unknotted}[2]{
\begin{scope} 
  \draw[expand style = #1] (A) -- (B) \arr;
  \draw[expand style = \stylevect] (D) -- (C);
\end{scope}}

\newcommand{\xing}[2]{
\begin{scope}
  \draw[expand style = #1] (A) .. controls +(sr) and +(0,0) .. (O) \arr  .. controls +(0,0) and +(sl) .. (C) \arr;
  \fill[white] (O) circle (1mm);
  \draw[expand style =\stylevect] (D) .. controls +(sr) and +(sl) .. (B);
\end{scope}}

\newcommand{\xingm}[2]{
  \begin{scope}
  \draw[expand style =\stylevect] (D) .. controls +(sr) and +(sl) .. (B);
  \fill[white] (O) circle (1mm);
  \draw[expand style = #1] (A) .. controls +(sr) and +(sl) .. (C) \arr ;
\end{scope}
}

\newcommand{\III}[3][\rarr]{
\begin{scope}
  \draw[expand style = #2] (A) .. controls +(sr) and +(0,0) .. (t) \arr;
  \draw[expand style = #2] (C) .. controls +(sl) and +(0,0) .. (b) \arr;
  \draw[expand style = \stylevect] (D) .. controls +(sr) and +(0,0) .. (b) ;
  \draw[expand style = \stylevect] (B) .. controls +(sl) and +(0,0) .. (t) ;
  \draw[expand style = #3] (b) -- (t) #1;
\end{scope}}

\newcommand{\IIIr}[2]{\III[\arr]{#1}{#2}}

\newcommand{\HHH}[2]{
\begin{scope}
  \draw[expand style = #1] (A) .. controls +(sr) and +(0,0) .. (l) \arr;
  \draw[expand style = #1] (C) .. controls +(sl) and +(0,0) .. (r) \arr;
  \draw[expand style = \stylevect] (D) .. controls +(sr) and +(0,0) .. (l) ;
  \draw[expand style = \stylevect] (B) .. controls +(sl) and +(0,0) .. (r) ;
  \draw[expand style = #2] (l) -- (r) \arr;
\end{scope}}

\newcommand{\ddd}[3][]{\NB{\tikz[#1]{
\begin{scope}
  \ccc{1}
  #2{\styleeven}{\styleodd}
  \ccc{-1}
  #3{\styleeven}{\styleodd}
\end{scope}}}}

\newcommand{\dddodd}[3][]{\NB{\tikz[#1]{
\begin{scope}
  \ccc{1}
  #2{\styleodd}{\styleodd}
  \ccc{-1}
  #3{\styleodd}{\styleodd}
\end{scope}}}}

\newcommand{\dd}[2][]{\NB{\tikz[#1]{
\begin{scope}
  \ccc{2}
  #2{\styleeven}{\styleodd}
\end{scope}}}}

\newcommand{\ddodd}[2][]{\NB{\tikz[#1]{
\begin{scope}
  \ccc{2}
  #2{\styleodd}{}
\end{scope}}}}

   As usual, we only need to prove the first one.
  \begin{align}
    \ddd[scale=0.7]{\xingm}{\xingm}=&-q \ddd[scale=0.7]{\HHH}{\IIIr} + \ddd[scale=0.7]{\III}{\IIIr} +\ddd[scale=0.7]{\HHH}{\HHH} - q^{-1} \ddd[scale=0.7]{\III}{\HHH} \\
    =&  \scriptstyle{({-q[\myN -1] +[\myN-2] + [\myN] - q^{-1}[\myN-1]})} \dd[scale=0.7]{\III}\, +\, \dd[scale=0.7]{\unknotted} \\
    =& \dd[scale=0.7]{\unknotted}. \qedhere
  \end{align}
\end{proof}

From Lemmas~\ref{lem:PF} and \ref{lem:R2-spin-vect} and the definition of a singular crossing, one immediately obtains the following.

\begin{cor}\label{cor:R3sing}
  The following local relations hold:
  \begin{equation}
    \label{eq:R3sing}
    \ev{\NB{\tikz[scale=0.7]{\input{\imagesfolder/td-R3sing-0}}}} =     \ev{\NB{\tikz[scale=-0.7]{\input{\imagesfolder/td-R3sing-0}}}} \qquad \text{and} \qquad
    \ev{\NB{\tikz[scale=0.7]{\input{\imagesfolder/td-R3sing-1}}}} =     \ev{\NB{\tikz[scale=-0.7]{\input{\imagesfolder/td-R3sing-1}}}}.
  \end{equation}
\end{cor}

\begin{lem}\label{lem:R3}
  The following local relation holds:
  \begin{equation}
    \label{eq:R3sing}
    \ev{\NB{\tikz[scale=0.7]{\input{\imagesfolder/td-R3-0}}}} = \ev{\NB{\tikz[scale=-0.7]{\input{\imagesfolder/td-R3-0}}}}.
  \end{equation}  
\end{lem}

\begin{proof}
  Using Definition \ref{dfn:Xings}, Corollary \ref{cor:R3sing} and Lemma \ref{lem:R2}, we find:
  \begin{align}
    \ev{\NB{\tikz[scale=0.7]{\input{\imagesfolder/td-R3-0}}}}&= -q \ev{\NB{\tikz[scale=0.7]{\input{\imagesfolder/td-R3-2}}}} + \ev{\NB{\tikz[scale=0.7]{\input{\imagesfolder/td-R3sing-1}}}} -q^{-1} \ev{\NB{\tikz[scale=0.7]{\input{\imagesfolder/td-R3-1}}}} \\
                                  &= -q \ev{\NB{\tikz[scale=-0.7]{\input{\imagesfolder/td-R3-2}}}} + \ev{\NB{\tikz[scale=-0.7]{\input{\imagesfolder/td-R3sing-1}}}} -q^{-1} \ev{\NB{\tikz[scale=-0.7]{\input{\imagesfolder/td-R3-1}}}} \\
   &= \ev{\NB{\tikz[scale=-0.7]{\input{\imagesfolder/td-R3-0}}}}.\qedhere
  \end{align}
\end{proof}

From the very definition of $\ev{\cdot}$ on
crossings \eqref{eq:dfn:xingVV}, we obtain that it satisfies the following skein relation:
\begin{equation}
  \label{eq:skein-ev-link}
\ev{\,\,\NB{\tikz[rotate=90]{}}\,\,} -
   \ev{\,\,\NB{\tikz[rotate=00]{}}\,\,}
   = (q-q^{-1})\left(\ev{\,\,\NB{\tikz[rotate=90]{}}\,\,} - \ev{\,\,\NB{\tikz[rotate=00]{}}\,\,} \right).
\end{equation}

We now can prove Theorem~\ref{theo:A} from the introduction.

\begin{thm}\label{thm:restated-thm-A}
    There is of a family of link invariants satisfying the following skein relations:
\begin{gather}
     \evN{\,\,\NB{\tikz[rotate=90]{}}\,\,} -
   \evN{\,\,\NB{\tikz[rotate=00]{}}\,\,}
   = (q-q^{-1})\left(\evN{\,\,\NB{\tikz[rotate=90]{}}\,\,} -
     \evN{\,\,\NB{\tikz[rotate=00]{}}\,\,} \right), \\
       \evN{\NB{\tikz[scale=0.5]{}}} = -q^{2\myN -1} \evN{\,\NB{\tikz[scale=0.5]{}}\,}, \qquad \qquad
    \evN{\NB{\tikz[scale=0.5, yscale=-1]{}}} = -q^{1-2\myN}
    \evN{\,\NB{\tikz[scale=0.5]{}}\,}, \\ \text{and}\qquad  \evN{\NB{\tikz[]{}}} = \left([2\myN-1] +1\right).  
\end{gather}    
\end{thm}

\begin{proof}
Proposition~\ref{prop:link-invariant} says $\ev{\cdot}$ is a link invariant, while equation \eqref{eq:skein-ev-link}, Lemma~\ref{lem:R1}, and Example \ref{ex:vect} imply the invariant satisfies the skein relations.
\end{proof}

\coloreddiagrams
\section{Type \texorpdfstring{$\mathsf{D}$}{D} intertwiners}
\label{sec:intertwiners}

The purpose of this section is relate the link invariant $\ev{\cdot}$ from Theorem~\ref{thm:restated-thm-A}, with the
Reshetikhin--Turaev link invariant \cite{MR1036112} for
$U_q(\sotN)$. In order to accomplish this task, we explicitly describe
homomorphisms between tensor products of minuscule representations for
$U_q(\sotN)$. These morphisms can be represented in the usual graphical calculus for $U_q(\sotN)$ which is the inspiration for the braided webs studied above, analogous to
\cite[Appendix A]{RW2}. We also derive enough relations between these
morphisms to characterize the link invariant and deduce
Theorem~\ref{thm:ev-vs-alg} which relates it to the link invariant
$\evN{\cdot}$ defined in Section~\ref{sec:linkinvariants}. In this
section, we assume that $\myN\geq 3$.

\subsection{Background on quantum group}

    Let $\mathfrak{g}$ be a semisimple Lie algebra over $\mathbb{C}$ and let $\Phi$ be the associated root system. Fix a choice of simple roots $\Pi$. Write $(-,-)$ for the $W$ invariant bilinear form on $\mathbb{Z}\Phi$ such that $(\alpha, \alpha)=2$ for all short roots. Given $\alpha \in \Phi$, define $\alpha^{\vee} = \frac{2\alpha}{(\alpha, \alpha)}$. Let $q$ be an indeterminant and set $q_{\alpha}:=q^{\frac{(\alpha, \alpha)}{2}}$. We write $X$ for the set of integral weights and $X_+\subset X$ for the set of dominant integral weights.

\begin{exa}
    In the present work, we are only interested in the case of $\mathfrak{g} = \sotN$. In this case,
    \[
    \Phi = \{\pm(\epsilon_i \pm \epsilon_j)\}\subset \oplus_{i=1}^{\myN}\mathbb{R}\epsilon_i 
    \]
    and
    \[
    \Pi = \{\alpha_1 = \epsilon_1 - \epsilon_2, \dots, \alpha_{n-2} = \epsilon_{n-1} - \epsilon_{n-1}, \alpha_{n-1} = \epsilon_{n-1} - \epsilon_n, \alpha_{n} = \epsilon_{n-1} + \epsilon_n\}.
    \]
    Also, the form $(-,-)$ is simply defined by $(\epsilon_i, \epsilon_j) = \delta_{i,j}$. The corresponding weights (respectively dominanat weights) are the $\mathbb{Z}$ (respectively $\mathbb{Z}_{\ge 0}$) span of the fundamental weights
    \[
    \varpi_1 = \epsilon_1 , \varpi_2 = \epsilon_1 + \epsilon_2, \dots, \varpi_{n-2} = \epsilon_1 + \dots + \epsilon_{n-2},
    \]
    \[
    \varpi_{n-1} = \frac{1}{2}(\epsilon_1 + \dots + \epsilon_{n-1} + \epsilon_n), \quad \text{and} \quad  \varpi_n = \frac{1}{2}(\epsilon_1 + \dots + \epsilon_{n-1} - \epsilon_n).
    \]
\end{exa}
    
    Associated to this data is a $\mathbb{C}(q)$-algebra, denoted $U_q(\mathfrak{g})$ and referred to as the Drinfeld-Jimbo quantum group, which is defined by generators
    \[
    e_i, f_i, k_i^{\pm 1}, \quad \text{for $\alpha_i\in \Pi$},
    \]
    and relations \cite[Section 4.3]{Jantzenbook}. 
    
    The $\mathbb{C}(q)$-algebra $U_q(\mathfrak{g})$ is a Hopf algebra. The structure maps are given on generators as follows:
    \begin{enumerate}
        \item $\Delta(e_i) = e_i \otimes k_i + 1\otimes e_i$, 
	$\Delta(f_i) = f_i\otimes 1 + k_i^{-1}\otimes f_i$, 
	and $\Delta(k_i^{\pm}) = k_i ^{\pm}\otimes k_i^{\pm}$.
        \item $S(e_i) = -e_i k_i^{-1}$, 
	$S(f_i) = - k_i f_i$, 
	and $S(k_i^{\pm})= k_i^{\mp}$.
        \item $\epsilon(e_i)= 0$, 
	$\epsilon(f_i) = 0$, 
	and $\epsilon(k_i) = 1$.
    \end{enumerate}
    We write $\mathbb{C}(q)$ for the trivial representation of $U_q(\sotN)$, i.e. for $X\in U_q(\sotN)$ and $\alpha\in \mathbb{C}(q)$, we set $u\cdot \alpha := \epsilon(u)\alpha$. Given a representation $W$ of $U_q(\mathfrak{g})$, we have an action of $U_q(\mathfrak{g})$ on $W^*= \Hom_{\mathbb{C}(q)}(W, \mathbb{C}(q))$ by $(u\cdot f)(w) = f(S(u)\cdot w)$ for all $f\in W^*$ and $w\in W$.

\subsubsection{Background on representations}

    We restrict to finite dimensional representations such that for each $\alpha_i\in \Pi$, $k_i$ is diagonalizable with eigenvalues contained in $\{q^n\}_{n\in \mathbb{Z}}$. This collection of modules is called finite dimensional type-$\textbf{1}$. Since the $k_i$ all commute, they are simultaneously diagonalizable on such representations. Let $W$ be a finite dimensional type$-\textbf{1}$ representation and define weight spaces for $\lambda \in X$ as
    \[
    W[\lambda]:=\{w\in W \ | \ k_iw = q_{\alpha_i}^{(\alpha_i^{\vee}, \lambda)}w\}. 
    \]
    Then $W$ is a direct sum of its weight spaces \cite[Proposition 5.1]{Jantzenbook}. Type-$\textbf{1}$ representations are also closed under tensor product and duals \cite[Section 5.4]{Jantzenbook}.

    Finite dimensional type-$\textbf{1}$ modules are direct sums of irreducible modules \cite[Theorem 5.17]{Jantzenbook}. For each $\lambda\in X_+$, there is a finite dimensional type-$\textbf{1}$ irreducible representation $V(\lambda)$ such that the weight spaces are the same dimension as the corresponding irreducible representation of $\mathfrak{g}$. Such representations are mutually non-isomorphic and any irreducible finite dimensional type-$\textbf{1}$ representation of $U_q(\mathfrak{g})$ is isomorphic to one \cite[Theorem 5.11]{Jantzenbook}.

    For more discussion about the justification for restricting to type-$\textbf{1}$ representations, see \cite[Section 5.2]{Jantzenbook}.

\subsection{Explicit calculations with minuscule representations}

    In the case that $\varpi\in X_+$ is minuscule, the action of $U_q(\mathfrak{g})$ on $V(\varpi)$ can be determined explicitly \cite[Section 5A.1]{Jantzenbook}. For $\sotN$, the (non-zero) dominant minuscule weights are $\varpi_1$, $\varpi_{n-1}$, and $\varpi_n$. We will now recall how the representations $V(\varpi_1)$, $V(\varpi_{n-1})$, and $V(\varpi_n)$ correspond to the vector and half spinor representations. 

\subsubsection{Vector representation}

    \begin{dfn}
    Let $V$ be the $\mathbb{C}(q)$-vector space with basis $\{v_{\pm i}\}_{i\in \setN}$.
    \end{dfn}

    \begin{dfn}
        For $i\in \setN$, define $\wt (v_i)= \epsilon_i$ and $\wt (v_{-i}) = -\epsilon_i$. 
    \end{dfn}

    \begin{prop}\label{P:vectorrepn}
        The operators
        \begin{align*}
        f_i\cdot v_k &= \begin{cases}
        v_{i+1} \qquad& \text{if $k=i$,} \\
        v_{-i} \qquad &\text{if $k= -(i+1)$,} \\
        0 \qquad &\text{otherwise,}\end{cases} && \qquad f_\myN\cdot v_k &= \begin{cases}
            v_{-(\myN-1)} \qquad &\text{if $k=\myN$,} \\
            v_{-\myN} \qquad &\text{if $k= \myN-1$,} \\
            0 \qquad &\text{otherwise,}\end{cases} \\
e_i\cdot v_k &= \begin{cases}
        v_{i} \qquad &\text{if $k=i+1$,} \\
        v_{-(i+1)} \qquad & \text{if $k= -i$,} \\
        0 \qquad &\text{otherwise,}\end{cases} && \qquad e_\myN\cdot v_k &= \begin{cases}
            v_{\myN-1} \qquad &\text{if $k=-\myN$,} \\
            v_{\myN} \qquad &\text{if $k= -(\myN-1)$,} \\
            0 \qquad   &\text{otherwise,}\end{cases} \\ 
k_j\cdot v_k &= q_{\alpha_j}^{(\alpha_j^{\vee}, \wt (v_k))}v_k, \end{align*}
        for $i=1,\dots, \myN -1$ and $j=1, \dots, \myN$,
        define an action of $U_q(\sotN)$ on $V$. Moreover, $V\cong V(\varpi_1)$ so in particular $V$ is a simple $U_q(\sotN)$-module.  
    \end{prop}

    \begin{proof}
        Using the known character of $V(\varpi_1)$, it follows that the vector spaces $V$ and $V(\varpi_1)$ are isomorphic as ``weight graded" vector spaces. The result then follows from a generators and relations check \cite[Section 5A.1]{Jantzenbook}.
    \end{proof}

    \begin{lem} \label{lem:autodual-vect}
        The $\mathbb{C}(q)$-linear map $\varphi_V:V\longrightarrow V^*$, defined by
        \begin{align*}
            &v_i \mapsto (-q^{-1})^{i-1}v_{-i}^* \\
            &v_{-i}\mapsto (-q^{-1})^{\myN-1}(-q^{-1})^{\myN-i}v_i^*, 
        \end{align*}
        for $i=1, \dots, \myN$, is an isomorphism of $U_q(\sotN)$-modules. 
      \end{lem}
\begin{proof}
        Direct check using the definition of the antipode $S$.
    \end{proof}

\subsubsection{Spin representation}

        \begin{dfn}
            Let $S$ be the $\mathbb{C}(q)$-vector space with basis $\{x_I\}_{I\subset \setN}$.
        \end{dfn}

        \begin{notation}
            When $\myN$ is understood, we write $I^c$ for $\setN\setminus I$. 
        \end{notation}
        
        \begin{dfn}
            For $I\subset \setN$, let 
            \[
            \varepsilon_k(I) = \begin{cases}
                +1 \quad \text{if} \quad k\notin I \\
                -1 \quad \text{if} \quad k\in I.
            \end{cases}
            \]
            Define $\wt(x_I):= \sum_{i=1}^{\myN} \varepsilon_i(I)\cdot \frac{\epsilon_i}{2}\in X(\sotN)$.
        \end{dfn}

        \begin{rmk}
            The only dominant weights one obtains as $\wt(x_I)$ are 
            \[
            \wt(x_{\emptyset}) = \frac{\epsilon_1}{2} + \dots + \frac{\epsilon_{\myN-1}}{2} + \frac{\epsilon_\myN}{2} = \varpi_{\myN} \quad \text{and} \quad \wt(x_{\emptyset}) = \frac{\epsilon_1}{2} + \dots + \frac{\epsilon_{\myN-1}}{2} - \frac{\epsilon_\myN}{2} = \varpi_{\myN-1}.
            \]
        \end{rmk}

        \begin{prop}\label{P:U-action-spin}
            The operators
            \begin{align*}
            f_i\cdot x_I &= \begin{cases}
                x_{I\setminus \{i+1\}\cup \{i\}} \quad \text{if $i+1\in I$ and $i\notin I$} \\
                0 \quad \text{otherwise}
            \end{cases}\\
f_\myN\cdot x_I &= \begin{cases}
                x_{I\cup\{\myN-1, \myN\}} \quad \text{if $\myN-1, \myN\notin I$} \\
                0 \quad \text{otherwise},
            \end{cases}\\
            e_i\cdot x_I &= \begin{cases}
                x_{I\setminus\{i\}\cup\{i+1\}} \quad \text{if $i\in I$ and $i+1\notin I$} \\
                0 \quad \text{otherwise}
            \end{cases}\\
e_\myN\cdot x_I &= \begin{cases}
                x_{I\setminus\{\myN-1, \myN\}} \quad \text{if $\myN-1, \myN\in I$} \\ 
                0 \quad \text{otherwise},
            \end{cases}\\
            k_j\cdot x_I &= q^{(\alpha_j^{\vee}, \wt x_I)}x_I
\end{align*}
            for $i= 1, \dots, \myN-1$ and $j= 1,\dots, \myN$, define
            an action of $U_q(\sotN)$ on $S$.
\end{prop}
        
        \begin{proof}
            Same as proof of Proposition \ref{P:vectorrepn}
        \end{proof}
        
        \begin{rmk}
            It is useful to note that $(\alpha_i^{\vee}, \wt x_I) \in \{1, 0, -1\}$, so $k_i$ has eigenvalues $\{q, 1, q^{-1}\}$ in $S$.
        \end{rmk}

        \begin{notation}
            For $I\subset \setN$, write $(-q)^I:=\prod_{i\in
              I}(-q)^{\myN-i}$, e.g. if $\myN=4$, then $(-q)^{\{2,3\}} =
            (-q)^{2+1}$ and $(-q)^{\{1,3,4\}} =
            (-q)^{3+1+0}$. Similarly, write
            $(-q^{-1})^I:=\left((-q)^I\right)^{-1}$.
        \end{notation}

    \begin{lem}\label{lem:autodual-spin}
        The $\mathbb{C}(q)$-linear map $\varphi_S:S\rightarrow S^*$ defined by
        \[
        x_I\mapsto (-q^{-1})^{I} x_{I^c}^*
      \]
is an isomorphism of $U_q(\sotN)$-modules. 
    \end{lem}
    \begin{proof}
        Direct check.
    \end{proof}

Since the algebra $U_q(\sotN)$ is a Hopf algebra, its category of finite dimension modules is monoidal and rigid. It is convenient to setup a graphical calculus. Diagrams should be read from bottom to top. Bottom to top juxtoposition of diagrams corresponds to composition of morphisms. Left to right juxtaposition of diagrams correspond to tensor product. Upward oriented strands labelled by $W$ represent the identity of $W$ while downward oriented strands labelled by $W$ represent $W^*$. Counterclockwise caps and clockwise cups labelled by $W$ correspond to $\eval_W$ and $\coeval_W$ respectively.

The relation with the diagrams drawn in the previous sections is not
completely transparent hence we'll use different diagrams.    

\graydiagrams
    \begin{notation}
      The identities of
      the vector and the spin representation are respectively given
      by
      \[
        \NB{\tikz[]{\begin{scope}
  \draw[expand style=\stylevect] (0,0) -- (0,1);
\end{scope}}} \qquad \text{and} \qquad \NB{\tikz[]{\begin{scope}
  \draw[expand style=\stylespin] (0,0) -- (0,1);
\end{scope}}}.
      \] 

      The careful reader may have noticed that in this notation,
      neither of the strands carries an arrow. We established that the vector and spin representation are self-dual in Lemma \ref{lem:autodual-vect} and Lemma \ref{lem:autodual-spin}, and will give more explanation why we are justified to ignore arrows on these diagrams in Section \ref{sec:cupscaps}. The half spin representation will be self dual if and only if $N$ is even, so we will use arrows in the graphical calculus when working with half spin representations.
    \end{notation}
    
\subsection{Intertwiners}

    It is well-known that the category of finite dimensional type-$1$ $U_q(\sotN)$-modules is a ribbon category. To be precise, this requires we specify an $R$ matrix and a ribbon element. We work with the conventions from \cite{SnyderTingley}. In particular, we use the non-standard ribbon element which has the property that the Frobenius-Schur indicators of all self-dual representations are $+1$ \cite[Lemma 5.7]{SnyderTingley}.

\subsubsection{Braiding}

It follows from semisimplicity of the category of finite dimensional type-$\textbf{1}$ $U_q(\sotN)$-modules, and the fact that each $V(\lambda)$ is a direct sum of its weight spaces, that we can describe operators on all finite dimensional type-$\textbf{1}$ representations by specifying their action on weight vectors in $V(\lambda)$, for all $\lambda \in X_+$. 

For a weight vector $v\in V(\lambda)$, the operator $T_{w_0}$ is the longest element in the quantum Weyl group, with simple reflection generators 
\[
T_{\alpha_i}(v) = \sum_{\substack{a,b\ge 0 \\ b-a = (\alpha_i^{\vee}, \mu)}}(-q_{\alpha_i})^be_{\alpha_i}^{(a)}f_{\alpha_i}^{(b)},
\]
and $J$ is specified by its action on irreducible representations as
\[
J\cdot v = q^{\frac{1}{2}(\wt (v), \wt (v) + 2\rho)}.
\]
Note that to define $J$ for $U_q(\sotN)$, we technically need to adjoin $q^{1/4}$. This is a minor modification, which we ignore in what follows. 

We define $X:=JT_{w_0}$. Following \cite[Section 3]{SnyderTingley}, we define the braiding as $R:=\text{flip}\circ \hat{R}$, with
\[
\hat{R}:=(X^{-1}\otimes X^{-1})\circ \Delta(X).
\]
This agrees with the usual expression 
\[
\hat{R} = \sum a_i\otimes b_i\in 1\otimes 1 + U_q(\sotN)^{> 0}\otimes U_q(\sotN)^{< 0},
\]
where $U_q(\sotN)^{>0}$ is the subalgebra generated by the $e_i$'s and $U_q(\sotN)^{<0}$ is the subalgebra generated by the $f_i$'s. 

We also work with Snyder-Tingley's non-standard ribbon element $X^{-2}$. This means in particular, that the category of finite dimensional type-$\textbf{1}$ $U_q(\sotN)$-modules is pivotal. Let $W$ be such a module, then the pivotal structure is $p_W:W\rightarrow W^{**}$, $p_W(w)(f) = f(g\cdot w)$, where $g= X^{-2}u$ and $u = \sum_i S(b_i)a_i$.

\subsubsection{Cups and Caps}
\label{sec:cupscaps}

    \begin{notation}
        Fix a finite dimensional vector space $W$, with basis $\mathbb{B}_W$. We write $\{b^*\}_{b\in \mathbb{B}}$ for the dual basis of $W^*$. Also, we will write $\mathrm{C}_W:=\sum_{b\in \mathbb{B}}b\otimes b^*\in W\otimes W^*$. It turns out that this element does not depend on the choice of basis. 
    \end{notation}
    
    Let $W$ be a finite dimensional $U_q(\sotN)$-module. The $\mathbb{C}(q)$-linear maps 
        \begin{align*}
            \eval_W&:W^*\otimes W\longrightarrow \mathbb{C}(q) \\
            &g\otimes w\mapsto g(w)
        \end{align*}
        and
        \begin{align*}
            \coeval_W&: \mathbb{C}(q)\longrightarrow W\otimes W^* \\
            &1\mapsto \mathrm{C}_W
        \end{align*}
        are $U_q(\sotN)$-module homomorphisms. These maps satisfy the zig-zag identities:
      \[
        \label{eq:33}
        (\id_W\otimes \eval_W)\circ (\coeval_W \otimes \id_W)
        = \id_W \quad \text{and} \quad \id_{W^*}=
        (\eval_W\otimes\id_{W^*})\circ (\id_{W^*} \otimes \coeval_W).
      \]
    \begin{dfn}\label{D:cap/cup-V-and-S}
        Let $W\in \{V, S\}$. Define $U_q(\sotN)$-module homomorphisms $\capp_W$ and $\cupp_W$ as the following compositions. 
        \[
        \capp_W:= \eval_W\circ (\varphi_W\otimes \id_W)\qquad \text{and} \qquad \cupp_W:=(\id_W\otimes \varphi_W^{-1})\circ \coeval_W. 
        \]
    \end{dfn}

    \begin{rmk}\label{R:explicit-cap-cup}
    We have the following explicit formulas:
    \[
    \capp_V(v_i\otimes v_{-i}) = (-q^{-1})^{i-1} \qquad \text{and} \qquad \capp_V(v_{-i}\otimes v_i) = (-q^{-1})^{\myN-1}(-q^{-1})^{\myN-i},
  \]
for $i=1, \dots, \myN$, while $\capp_V(v_r\otimes v_t) = 0$ otherwise, and
    \[
    \cupp_V(1) = \sum_{i=1}^\myN (-q)^{\myN-1}(-q)^{\myN-i} v_i\otimes v_{-i} + \sum_{i=1}^\myN (-q)^{i-1}v_{-i}\otimes v_i.
  \]
For the spin representation $S$ we have 
    \[
    \capp_S(x_I\otimes x_{I^c}) = (-q^{-1})^I \qquad \text{while} \qquad  \text{$\capp_S(x_R\otimes x_T) = 0$ otherwise}, 
    \]
    and 
    \[
    \cupp_S(1) = \sum_{I\subset \setN}(-q)^{I^c}x_I\otimes x_{I^c}.
    \]
    \end{rmk}

\begin{rmk}
    The pivotal structure associated to the ribbon element $X^{-2}$ then gives rise to $\widetilde{\eval}_W:= \eval_{W^*}\circ (p_W\otimes \id_{W^*})$ and $\widetilde{\coeval}_W:= (\id_{W^*}\otimes p_W^{-1})\circ \coeval_{W^*}$. These maps also satisfy the zig-zag identities. The pairs of evaluation and coevaluation maps allow morphisms to be ``rotated" 180 degrees clockwise or counterclockwise. In a pivotal category both rotations are equal.

    Moreover, since the Frobenius-Schur indicator of a self-dual irreducible is $+1$ \cite[Lemma 5.7]{SnyderTingley}, it follows that $\capp_W = \widetilde{\eval}_W\circ (\id_W\otimes \varphi_W)$ and $\cupp_W = (\varphi_W^{-1}\otimes \id_W)\circ \widetilde{\coeval}_W$. Thus, we can unambiguously work with unoriented cups and caps colored by self-dual irreducible representations.
\end{rmk}

Although the spin representation is not irreducible, it is self-dual. The following lemma justifies using unoriented cups and caps in the graphical calculus for $\capp_S$ and $\cupp_S$.

\begin{lem}
    We have the following equality of morphisms
    \begin{equation}\label{FPS1}
    \eval_S\circ (\varphi_S\otimes \id_S) = \capp_S = \widetilde{\eval}_S\circ (\id_S \otimes \varphi_S) 
    \end{equation}
    and
      \begin{equation}\label{FPS2}
     (\id_S\otimes \varphi_S^{-1})\circ \coeval_S = \cupp_S =(\varphi_S^{-1}\otimes \id_S)\circ \widetilde{\coeval}_S
    \end{equation}
\end{lem}

\begin{proof}
    It is a standard exercise to deduce equation \eqref{FPS2} from equation \eqref{FPS1} using the zig-zag identities. It suffices to prove that 
    \[
    \eval_S\circ (\varphi_S\otimes \id_S) = \eval_{S^*}\circ (p_S\otimes \varphi_S).
    \]
    Weight space considerations imply we can restict to checking equality on simple tensors of the form $x_I\otimes x_{I^c}$, which is implied by the following argument adapted from the proof of \cite[Lemma 7.4]{SnyderTingley} which crucially relies on the identity $S(X) = gX$ \cite[Equation 41]{SnyderTingley} . Note that since conjugation by $X$ acts as in \cite[Lemma 3.10(iv)]{SnyderTingley}, it follows from weight space considerations that $X\cdot x_I = \xi_Ix_{I^c}$ for some $\xi_I\in \mathbb{C}(q)$. Since $X$ is invertible, $\xi_I\ne 0$. Then, we compute
    \begin{align*}
    \eval_S\circ (\varphi_S\otimes \id_S)(x_{I^c}\otimes x_I) &= \varphi_S(x_{I^c})(x_I) \\
    &= \varphi_S(\xi_I^{-1}X\cdot x_I)(x_I) \\
    &= \xi_I^{-1}\varphi_S(x_I)(S(X)\cdot x_I) \\
    &= \xi_I^{-1} \varphi_S(x_I)(gX\cdot x_I) \\
    &= \varphi_S(x_I)(gx_{I^c}) \\
    &= \widetilde{\eval}_S\circ (\id_S\otimes \varphi_S)(x_{I^c}\otimes x_I).
    \end{align*}
\end{proof}

    \begin{notation}
      The morphisms $\cupp_V$, $\capp_V$,  $\cupp_S$
      and  $\capp_S$ are depicted by:
      \[
        \NB{\tikz[]{\begin{scope}
  \draw[expand style=\stylevect] (0,0) arc(0:-180:0.5);
\end{scope}}}, \quad \NB{\tikz[yscale=-1]{}}, \quad
        \NB{\tikz[]{\begin{scope}
  \draw[expand style=\stylespin] (0,0) arc(0:-180:0.5);
\end{scope}}}, \quad \text{and} \quad \NB{\tikz[yscale=-1]{}}.
      \]
    \end{notation}
    
    \begin{lem}\label{L:S-zig-zag}
      These morphisms satisfy the usual zig-zag relations:
      \begin{equation}
        \label{eq:28}
        \NB{\tikz[expand style=\stylevect,xscale=-1]{\begin{scope}
  \draw (0, +0.5) -- (0,0) arc (0:-180:0.25cm) arc (0:180:0.25cm) -- +(0, -0.5);
\end{scope}}} =
        \NB{\tikz[]{}} =
        \NB{\tikz[expand style=\stylevect]{}}
        \qquad \text{and}\qquad
        \NB{\tikz[expand style=\stylespin,xscale=-1]{}} =
        \NB{\tikz[]{}} =
        \NB{\tikz[expand style=\stylespin]{}}.
      \end{equation}
    \end{lem}

\subsection{Twists}\label{subsec:twists}

Using the braiding, cups, and caps, we obtain endomorphisms of objects called ``twists" \cite[Section 2]{MR1797619}. When the object is irreducible, Schur's lemma implies that the twist is a scalar multiple of the identity. Using our conventions, the scalar for the twist of $V(\lambda)$ is $(-1)^{(2\lambda, \rho^{\vee})}q^{(\lambda, \lambda + 2\rho)}$ \cite[Lemma 3.10]{SnyderTingley}. 

\begin{exa}
    For $\sotN$, we have $\rho = (\myN-1)\epsilon_1 + \dots + 2\epsilon_{\myN-2} + \epsilon_{\myN-1}$. Moreover, since $\sotN$ is simply laced, we have $\rho^{\vee} = \rho$. Thus, for the vector representation, with highest weight $\varpi_1 = \epsilon_1$, we have the scalar for the twist is $q^{2\myN-1}$, and for both half spin representations, with highest weights
    \[
    \varpi_{\myN-1} = \frac{1}{2}\epsilon_1 + \dots + \frac{1}{2}\epsilon_{\myN-1} - \frac{1}{2}\epsilon_\myN \qquad \text{and} \qquad \varpi_\myN = \frac{1}{2}\epsilon_1 + \dots + \frac{1}{2}\epsilon_{\myN-1} + \frac{1}{2}\epsilon_\myN,
    \]
    we have the scalar for the twist is $(-q)^{\binom{\myN}{2}}q^{\frac{\myN}{4}}$.
\end{exa}

\subsubsection{Trivalent vertices}

    \begin{notation}
        For $i\in \setN$, let $\sigma_i:\setN\setminus\{i\} \rightarrow \{2, \dots, \myN\}$ be the order preserving bijection, e.g. when $\myN=4$, then $\sigma_3$ is the map: $1\mapsto 2$, $2\mapsto 3$, and $4\mapsto 4$.
    \end{notation}
    
    \begin{dfn}\label{D:YVSS}
        Define the $\mathbb{C}(q)$-linear map $\Ya_{V}^{S, S}:V\longrightarrow S\otimes S$ by
        \[
        v_i\mapsto \sum_{\substack{I,J\subset \setN \\ I\cup J = \setN\setminus \{i\} \\ I \cap J = \emptyset}}(-q)^{\sigma_i(J)}x_I\otimes x_J \quad \text{and} \quad v_{-i} \mapsto \sum_{\substack{I,J\subset \setN \\ I\cup J =\setN \setminus \{i\} \\ I \cap J = \emptyset}}(-q)^{\sigma_i(J)}x_{I\cup \{i\}}\otimes x_{J\cup \{i\}}
        \]
    \end{dfn}

    \begin{prop}
        The map $\Ya_V^{S, S}$ is a $U_q(\sotN)$ intertwiner. 
    \end{prop}
    \begin{proof}
        Note that $V\cong V(\varpi_1)$, and $\wt \Ya_V^{S,S}(v_1)=\epsilon_1=\varpi_1$, It follows from the universal property of $V(\varpi_1)$, that if we show that $e_k\cdot \Ya_V^{S,S}(v_1)=0$, for $k\in \setN$, then there is a unique $U_q(\sotN)$-module homomorphism such that $v_1\mapsto \Ya_V^{S, S}(v_1)$. We then leave it to the reader to use the definition of the action of $f_i$ on $S$ to verify that this implies the intertwiner is equal to $\Ya_V^{S, S}$ as in Definition \ref{D:YVSS}. 
        
        One easily checks, using the definition of $\Delta$, that for $e_1\cdot \Ya_V^{S,S}(v_1)= 0$. Next, for $i=2, \dots, N-1$, we compute the following:
        \begin{align*}
            e_i\cdot \Ya_V^{S,S}(v_1) &= \sum_{\substack{I,J\subset \setN \\ I\cup J = \{2, \dots, \myN\} \\ I \cap J = \emptyset}}(-q)^J(e_ix_I)\otimes (k_ix_J) + \sum_{\substack{I,J\subset \setN \\ I\cup J = \{2, \dots, \myN\} \\ I \cap J = \emptyset}}(-q)^Jx_I\otimes e_i(x_J) \\
            & = \sum_{\substack{I,J\subset \setN \\ I\cup J = \{2, \dots, \myN\} \\ I \cap J = \emptyset \\ i\in I, i+1\notin I}}(-q)^Jq^{(\alpha_i^{\vee}, \wt x_J)}x_{I\setminus\{i\}\cup\{i+1\}}\otimes x_J \\
            &+ \sum_{\substack{I,J\subset \setN \\ I\cup J = \{2, \dots, \myN\} \\ I \cap J = \emptyset \\ i\in J, i+1\notin J}}(-q)^Jx_I\otimes x_{J\setminus \{i\}\cup \{i+1\}} \\
            & = \sum_{\substack{I\subset \{2, \dots, \myN\} \\ i\in I, i+1\notin I}}(-q)^{\{2, \dots, \myN\}\setminus I}q^{(\alpha_i^{\vee}, \wt x_{\{2, \dots, \myN\}\setminus I})}x_{I\setminus\{i\}\cup\{i+1\}}\otimes x_{\{2, \dots, \myN\}\setminus I} \\
            &+ \sum_{\substack{J \subset \{2, \dots, \myN\} \\ i\in J, i+1\notin J}}(-q)^Jx_{\{2, \dots, \myN\}\setminus J}\otimes x_{J\setminus \{i\}\cup \{i+1\}}. 
        \end{align*}
        Reindexing the second sum by $I\subset \{2, \dots, \myN\}$ such that 
        \[
        I\setminus \{i\}\cup\{i+1\} = \{2, \dots, \myN\}\setminus J, \qquad \text{so} \qquad  J= \left(\{2, \dots, n\}\setminus I\right) \setminus \{i+1\}\cup \{i\},
        \]
        we have
        \begin{align*}
            e_i\cdot \Ya_V^{S,S}(v_1)& = \sum_{\substack{I\subset \{2, \dots, \myN\} \\ i\in I, i+1\notin I}}(-q)^{\{2, \dots, \myN\}\setminus I}q^{(\alpha_i^{\vee}, \wt x_{\{2, \dots, \myN\}\setminus I})}x_{I\setminus\{i\}\cup\{i+1\}}\otimes x_{\{2, \dots, \myN\}\setminus I} \\
            &+ \sum_{\substack{I \subset \{2, \dots, \myN\} \\ i\in I, i+1\notin I}}(-q)^{\left(\{2, \dots, n\}\setminus I\right) \setminus \{i+1\}\cup \{i\}}x_{I\setminus \{i\}\cup\{i+1\}}\otimes x_{\{2, \dots, \myN\}\setminus I}.
        \end{align*}
        Since
        \[
        (-q)^{\{2, \dots, \myN\}\setminus I}q^{(\alpha_i^{\vee}, \wt x_{\{2, \dots, \myN\}\setminus I})} + (-q)^{\left(\{2, \dots, n\}\setminus I\right) \setminus \{i+1\}\cup \{i\}} 
        \]
        is equal to 
        \[
        (-q)^{\{2, \dots, \myN\}\setminus I}\left(q + \frac{(-q)^{\myN-i}}{(-q)^{\myN-(i+1)}}\right) = 0,
        \]
        we deduce that $e_i\cdot\Ya_V^{S,S}(v_1) = 0$. A similar calculation shows $e_m\cdot \Ya_V^{S,S}(v_1)= 0$ as well. We leave the details to the reader. 
    \end{proof}

    \begin{notation}
      In graphical calculus, the morphism $\Ya_V^{S,S}$ is depicted by
      \[
        \NB
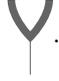
{\tikz[]{\begin{scope}
  \draw[expand style= \stylespin] (-0.25, 1) .. controls +(0, -0.25) and +(0,0) .. (0, 0.5)
  .. controls +(0,0) and +(0, -0.25) .. (0.25, 1);
  \draw[expand style =\stylevect] (0,0.5) -- (0,0);
\end{scope}}}.
      \]
    \end{notation}
    
    \begin{dfn}\label{D:PVSS}
        Define the $\mathbb{C}(q)$-linear map $\Pa_{S,S}^V:S\otimes S\rightarrow V$ by
        \[
        x_I\otimes x_J \mapsto \begin{cases}
            (-q^{-1})^{\sigma_i(I)} v_i \qquad \text{if $I\cup J = \setN\setminus \{i\}$ and $I\cap J = \emptyset$}, \\
            (-q^{-1})^{\sigma_i(I\setminus \{i\})} v_{-i} \qquad \text{if $I\cup J = \setN$ and $I \cap J = \{i\}$}, \qquad \text{and} \\
            0 \qquad \text{otherwise}.
        \end{cases}
        \]
    \end{dfn}

        \begin{notation}
      In graphical calculus, the morphism $\Pa_{S,S}^V$ is depicted
      \[
        \NB
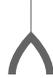
{\tikz[yscale =-1]{}}
      \]
    \end{notation}

    \begin{prop}
        The map $\Pa_{S,S}^V$ is equal to the composition of intertwiners
        \[
          \NB
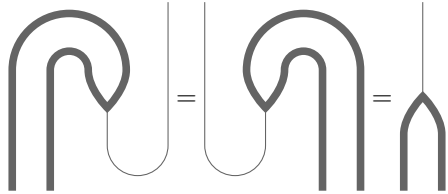
{\tikz[]{\begin{scope}[scale=1]
  \draw[expand style= \stylespin] (-0.75,-0.6) -- (-0.75,1) arc
  (180:0:0.25cm)  .. controls +(0, -0.25) and +(0,0) .. (0, 0.5) 
  .. controls +(0,0) and +(0, -0.25) .. (0.25, 1) arc (0:180:0.75cm)
  -- +(0, -1.6);
  \draw[expand style =\stylevect] (0,0.5) -- (0,0) arc (-180:0:0.4) --
  +(0,1.9);
\end{scope}}} = \NB{\tikz[xscale=-1]{}} =  \NB{\tikz[yscale =-1]{\begin{scope}
  \draw[expand style= \stylespin] (-0.25, 1.75) -- (-0.25, 1)
  .. controls +(0, -0.25) and +(0,0) .. (0, 0.5)
  .. controls +(0,0) and +(0, -0.25) .. (0.25, 1) -- (0.25, 1.75);
  \draw[expand style =\stylevect] (0,0.5) -- (0,-0.75);
\end{scope}}}
        \]
        In particular, $\Pa_{S,S}^V$ is a $U_q(\sotN)$ intertwiner. 
    \end{prop}
    \begin{proof}
        The equality of the left and right ``rotations" of $\Ya_{V}^{S,S}$ is a consequence of working in a pivotal category.

        We will show that the right twist of $\Ya_{V}^{S,S}$, which by definition is the composition of intertwiners
        \begin{equation}\label{eqn:triv-righttwist}
        (\id_V\otimes \capp_S)\circ (\id_V\otimes \id_S\otimes \capp_S\otimes \id_S)\circ (\id_V\otimes \Ya_{V}^{S,S}\otimes \id_S\otimes \id_S)\circ (\cupp_{V}\otimes \id_{S}\otimes \id_{S}),
        \end{equation}
        agrees with $\Pa_{S,S}^{V}$ on the vector $x_I\otimes x_J$, when $I$ and $J$ are disjoint and $I\cup J = \setN\setminus \{k\}$, leaving the other cases to the reader. To this end, fix $I$ and $J$ as above. Using the formulas for cups and caps in Remark \ref{R:explicit-cap-cup} and Definition \ref{D:YVSS}, we find that the map in equation \eqref{eqn:triv-righttwist} sends $x_I\otimes x_J$ to 
        \begin{align*}
        &(-q)^{\myN-1}(-q)^{\myN-k}(-q^{-1})^{I^c}(-q^{-1})^{J^c}(-q)^{\sigma_k(I^c\setminus \{k\})}v_k \\
        =&(-q)^{\myN-1}(-q)^{\myN-k}(-q^{-1})^{J}(-q^{-1})^{\myN-k}(-q^{-1})^{I}(-q^{-1})^{\myN-k}(-q)^{\sigma_k(J)}v_k \\
        =&(-q)^{k-1}(-q^{-1})^I(-q^{-1})^J(-q)^{J}(-q^{-1})^{|J\cap \{1, \dots, k-1\}|}v_k \\
        =&(-q^{-1})^I(-q)^{|I\cap \{1, \dots, k\}|}v_k \\
        =&(-q^{-1})^{\sigma_k(I)}v_k \\
        =&\Pa_{S,S}^{V}(x_I\otimes x_J).\qedhere
        \end{align*}
    \end{proof}
    \begin{notation}\label{N:I-and-H}
        We will write $\Ya_{S}^{V,S}$ and $\Pa_{S,V}^{S}$ etc. for the appropriate rotations of $\Ya_{V}^{S,S}$ and/or $\Pa_{S,S}^{V}$. We also use the letters $\II$ to denote $\Ya^{S,S}_{V}\circ \Pa_{S,S}^{V}$ and $\HH$ to denote $\II$'s ninety degree rotation.
    \end{notation}

\subsection{Even and odd spin}

\newcommand{\macroeven}{\ensuremath{\mathbf{e}}}
\newcommand{\macroodd}{\ensuremath{\mathbf{o}}}
\newcommand{\evenodd}{\ensuremath{\macroeven/\macroodd}}

      \begin{dfn}
        Let $S_{\macroeven}$ (respectively $S_{\macroodd}$) be the span of the $x_I$ such that the $|I|$ is even (respectively odd). 
          
        Write $\iota_{\macroeven}$ and $\iota_{\macroodd}$ for the inclusions of $S_{\macroeven}$ and $S_{\macroodd}$ into $S$. Let $\pi_{\macroeven}:S\rightarrow S_{\macroeven}$ and $\pi_{\macroodd}:S\rightarrow S_{\macroodd}$ be the following projections: 
            \[
             \pi_{\macroeven}(x_I) = \begin{cases}
                 x_I\qquad \text{if $|I|$ is even} \\
                 0 \qquad \text{if $|I|$ is odd},
             \end{cases}
             \qquad \text{and} \qquad \pi_{\macroodd}(x_I) = \begin{cases}
                 0 \qquad \text{if $|I|$ is even} \\ x_I \qquad \text{if $|I|$ is odd}.
             \end{cases}
           \]
         \end{dfn}

               \begin{rmk}
          It is easy to see from Proposition \ref{P:U-action-spin}, that the action of $U_q(\sotN)$ on $S$ respects the decomposition $S= S_{\macroeven} \oplus S_{\macroodd}$. Moreover,
            \[
            S_{\macroeven}\cong V(\varpi_{\myN}) \quad \text{and} \quad S_{\macroodd}\cong V(\varpi_{\myN-1}).
            \]
        Note that the maps $\pi_{\evenodd}$ and $\iota_{\evenodd}$ are $U_q(\sotN)$ intertwiners. 
      \end{rmk}

      \begin{notation}
           The identity morphisms of $S_{\macroeven}$ and $S_{\macroodd}$, the projections $\pi_\evenodd$
           and the inclusion maps $\iota_{\evenodd}$ will be depicted by:
           \[
             \NB{\tikz[]{\begin{scope}
  \draw[expand style=\styleeven] (0,0) -- (0,1) \arreven{};
\end{scope}}}, \quad
             \NB{\tikz[]{\begin{scope}
  \draw[expand style=\styleodd] (0,0) -- (0,1) \arrodd;
\end{scope}}}, \quad
             \NB{\tikz[]{\begin{scope}
  \draw[expand style=\styleeven] (0,0.5) -- (0,1) \arreven;
  \draw[expand style=\stylespin] (0,0) -- (0,0.5); % \arreven;
\end{scope}}}, \quad
             \NB{\tikz[]{\begin{scope}
  \draw[expand style=\styleodd] (0,0.5) -- (0,1) \arrodd;
  \draw[expand style=\stylespin] (0,0) -- (0,0.5); % \arreven;
\end{scope}}}, \quad
             \NB{\tikz[yscale=-1]{}}, \quad \text{and} \quad
             \NB{\tikz[yscale=-1]{}}.
           \]
           In particular, one has
           \begin{gather}
             \NB{\tikz[yscale=1.2]{}}= \NB{\tikz[]{\begin{scope}
%  \draw[expand style=\styleeven] (0,0.8) -- (0,1.2) \arreven;
  \draw[expand style=\stylespin] (0,0.4) -- (0,0.8); % \arreven;
  \draw[expand style=\styleeven] (0,0.0) -- (0,0.4) \arreven;
  \draw[expand style=\stylespin] (0,-0.4) -- (0,0);
\end{scope}}} +
             \NB{\tikz[]{\begin{scope}
%  \draw[expand style=\styleodd] (0,0.8) -- (0,1.2) \arrodd;
  \draw[expand style=\stylespin] (0,0.4) -- (0,0.8); % \arreven;
  \draw[expand style=\styleodd] (0,0.0) -- (0,0.4) \arrodd;
  \draw[expand style=\stylespin] (0,-0.4) -- (0,0);
\end{scope}}},\qquad \NB{\tikz[]{\begin{scope}
  \draw[expand style=\styleeven] (0,0.8) -- (0,1.2) \arreven;
  \draw[expand style=\stylespin] (0,0.4) -- (0,0.8); % \arreven;
  \draw[expand style=\styleodd] (0,0.0) -- (0,0.4) \arrodd;
%  \draw[expand style=\stylespin] (0,-0.4) -- (0,0);
\end{scope}}} =
             \NB{\tikz[yscale=-1]{}} =0, \qquad
             \NB{\tikz[]{\begin{scope}
  \draw[expand style=\styleeven] (0,0.8) -- (0,1.2) \arreven;
  \draw[expand style=\stylespin] (0,0.4) -- (0,0.8); % \arreven;
  \draw[expand style=\styleeven] (0,0.0) -- (0,0.4) \arreven;
%  \draw[expand style=\stylespin] (0,-0.4) -- (0,0);
\end{scope}}} = \NB{\tikz[yscale=1.2]{}},
             \qquad \text{and,} \qquad
             \NB{\tikz[]{\begin{scope}
  \draw[expand style=\styleodd] (0,0.8) -- (0,1.2) \arrodd;
  \draw[expand style=\stylespin] (0,0.4) -- (0,0.8); % \arreven;
  \draw[expand style=\styleodd] (0,0.0) -- (0,0.4) \arrodd;
%  \draw[expand style=\stylespin] (0,-0.4) -- (0,0);
\end{scope}}} = \NB{\tikz[yscale=1.2]{}}.
           \end{gather}
\end{notation}
         \begin{lem}\label{lem:odd-even-duality}
           The isomorphism $\varphi_S$ of Lemma~\ref{lem:autodual-spin} restricts
           to isomorphisms $\varphi_{S_{\macroeven}, S_{\macroeven}}: S_{\macroeven}^* \to S_{\macroeven}$ and
           $\varphi_{S_{\macroodd}, S_{\macroodd}}: S_{\macroodd}^* \to S_{\macroodd}$ if $\myN$ is even and to isomorphisms $\varphi_{S_{\macroeven}, S_{\macroodd}}: S_{\macroeven}^* \to S_{\macroodd}$ and
           $\varphi_{S_{\macroodd}, S_{\macroeven}}: S_{\macroodd}^* \to S_{\macroeven}$ if $\myN$ is odd.
         \end{lem}
         The impact of the parity of $\myN$, should be compared with the
         results of Section~\ref{sec:chang-parity-orient}.
         \begin{notation}
           In graphical calculus, Lemma~\ref{lem:odd-even-duality}
           translates into:
           \begin{equation}
                                       \label{eq:34}
             \NB{\tikz[]{\begin{scope}
  \draw[expand style=\styleeven] (0,0) -- (0,1) \rarreven;
\end{scope}}} = \NB{\tikz[]{}}
             \quad\text{and}\quad
             \NB{\tikz[]{\begin{scope}
  \draw[expand style=\styleodd] (0,0) -- (0,1) \rarrodd;
\end{scope}}} = \NB{\tikz[]{}}
             \quad\text{if $\myN$ is even,}\quad
             \NB{\tikz[]{}} = \NB{\tikz[]{}}
             \quad\text{and}\quad
             \NB{\tikz[]{}} = \NB{\tikz[]{}}
             \quad\text{if $\myN$ is odd;}\quad
           \end{equation}
           When $\myN$ is even, we could get rid of orientations of
           strands representing $S_{\macroeven}$, $S_{\macroodd}$ and their duals. However
           to keep the presentation homogeneous among the values of
           $\myN$, we prefer to keep them. 
         \end{notation}

      Using the inclusion and projections we can define various maps as follows. 
      \begin{dfn}
          Let $\epsilon_1,\epsilon_2\in \{\macroeven, \macroodd\}$. Define 
          \[
          \Ya_V^{{\epsilon_1},{\epsilon_2}}:= (\pi_{\epsilon_1}\otimes \pi_{\epsilon_2})\circ \Ya_V^{S,S}, \qquad \Pa_{{\epsilon_1},{\epsilon_2}}^V:=\Pa_{S,S}^V\circ (\iota_{\epsilon_1}\otimes \iota_{\epsilon_2}), 
          \]
          \[
          \capp_{{\epsilon_1},{\epsilon_2}} :=\capp_{S}\circ (\iota_{\epsilon_1}\otimes \iota_{\epsilon_2}), \qquad \text{and} \qquad \cupp_{{\epsilon_1},{\epsilon_2}}:= (\pi_{\epsilon_1}\otimes \pi_{\epsilon_2})\circ \cupp_{S}
          \]
      \end{dfn}
      \begin{notation}
          Following the notation from the previous definition and Notation \ref{N:I-and-H}, we write $\II_{V,\macroeven}^{V,\macroodd}$ and $\HH_{V,\macroeven}^{\macroeven,V}$ etc. We will also write $\id_{\macroeven}:=\id_{S_{\macroeven}}$, $\id_{\macroodd}:=\id_{S_{\macroodd}}$, and $R_{\epsilon_1,\epsilon_1}:=R_{S_{\epsilon_1},S_{\epsilon_2}}$.
      \end{notation}

      Depending on the parity of $\myN$, different maps from the previous definition are easily seen to be zero. 

      \begin{lem}\label{L:N-when-is-zero}
          If $\myN$ is even, then
          \[
          \Ya_{V}^{\macroeven,\macroeven}, \Ya_{V}^{\macroodd,\macroodd},\Pa_{\macroeven,\macroeven}^V,\Pa_{\macroodd,\macroodd}^V,
          \]
          \[
          \capp_{\macroeven,\macroodd}, \capp_{\macroodd,\macroeven},\cupp_{\macroeven,\macroodd}, \text{and} \ \cupp_{\macroodd,\macroeven}
          \]
          are equal to zero.
          
          If $\myN$ is odd, then
          \[
          \Ya_{V}^{\macroeven,\macroodd}, \Ya_{V}^{\macroodd,\macroeven},\Pa_{\macroeven,\macroodd}^V,\Pa_{\macroodd,\macroeven}^V,
          \]
          \[
          \capp_{\macroeven,\macroeven}, \capp_{\macroodd,\macroodd},\cupp_{\macroeven,\macroeven}, \text{and} \ \cupp_{\macroodd,\macroodd}
          \]
          are equal to zero.
      \end{lem}
      \begin{proof}
          Each map lies in a homomorphism space which is zero dimensional.
      \end{proof}
      \begin{notation}
        All other choices will give non-zero intertwiners. In graphical calculus we
        write:
        \begin{gather}
          \label{eq:36}
          \NB{\tikz[]{\begin{scope}
  \draw[expand style= \styleeven] (-0.25, 1) .. controls +(0, -0.25)  .. (0, 0.5) \arreven;
  \draw[expand style= \styleodd]
  (0, 0.5) .. controls +(0.25, 0.25) .. (0.25, 1) \arrodd;
  \draw[expand style =\stylevect] (0,0.5) -- (0,0);
\end{scope}}},\quad \NB{\tikz[]{\begin{scope}
  \draw[expand style= \styleodd] (-0.25, 1) .. controls +(0, -0.25) .. (0, 0.5) \arrodd;
  \draw[expand style= \styleeven] (0, 0.5)
  .. controls +(0.25,0.25) .. (0.25, 1) \arreven;
  \draw[expand style =\stylevect] (0,0.5) -- (0,0);
\end{scope}}},\quad
          \NB{\tikz[]{\begin{scope}
  \draw[expand style= \styleeven] (-0.25, 1) .. controls +(0, -0.25)
.. (0, 0.5) \rarreven;
  \draw[expand style= \styleodd] (0, 0.5)
  .. controls +(0.25, 0.25) .. (0.25, 1) \rarrodd;
  \draw[expand style =\stylevect] (0,0.5) -- (0,0);
\end{scope}}},\quad \NB{\tikz[]{\begin{scope}
  \draw[expand style= \styleodd] (-0.25, 1) .. controls +(0, -0.25) .. (0, 0.5) \rarrodd;
  \draw[expand style= \styleeven] (0, 0.5)
  .. controls  +(0.25, 0.25) .. (0.25, 1) \rarreven;
  \draw[expand style =\stylevect] (0,0.5) -- (0,0);
\end{scope}}},\quad
          \NB{\tikz[rotate=180]{}},\quad \NB{\tikz[rotate=180]{}},\quad
          \NB{\tikz[rotate=180]{}},\quad
          \NB{\tikz[rotate=180]{}},\\
          \NB{\tikz[scale=0.8]{\begin{scope}
  \draw[expand style=\styleeven] (0,0) arc(0:-180:0.5) \arreven;
\end{scope}}}, \quad \NB{\tikz[scale=0.8]{\begin{scope}
  \draw[expand style=\styleeven] (0,0) arc(0:-180:0.5) \rarreven;
\end{scope}}}, \quad
          \NB{\tikz[scale=0.8]{\begin{scope}
  \draw[expand style=\styleodd] (0,0) arc(0:-180:0.5) \arrodd;
\end{scope}}}, \quad \NB{\tikz[scale=0.8]{\begin{scope}
  \draw[expand style=\styleodd] (0,0) arc(0:-180:0.5) \rarrodd;
\end{scope}}}, \quad
          \NB{\tikz[scale=0.8, yscale=-1]{}}, \quad
          \NB{\tikz[scale=0.8, yscale=-1]{}}, \quad
          \NB{\tikz[scale=0.8, yscale=-1]{}}, \quad
          \NB{\tikz[scale=0.8, yscale=-1]{}}. 
        \end{gather}
        Note that if $\myN$ is even, one has:
        \begin{gather}
          \label{eq:38}
          \NB{\tikz[]{}}=\NB{\tikz[]{}}, \qquad
          \NB{\tikz[]{}}=\NB{\tikz[]{}}, \qquad
          \NB{\tikz[rotate=180]{}}=\NB{\tikz[rotate=180]{}}, \qquad
          \NB{\tikz[rotate=180]{}}=\NB{\tikz[rotate=180]{}}, \\
          \NB{\tikz[scale=0.8]{}} =\NB{\tikz[scale=0.8]{}}, \qquad
          \NB{\tikz[scale=0.8]{}} = \NB{\tikz[scale=0.8]{}}, \qquad
          \NB{\tikz[scale=0.8, yscale=-1]{}}=
          \NB{\tikz[scale=0.8, yscale=-1]{}}, \qquad
          \NB{\tikz[scale=0.8, yscale=-1]{}}=
          \NB{\tikz[scale=0.8, yscale=-1]{}};
        \end{gather}
        while if $\myN$ is odd, one has:
        \begin{gather}
          \NB{\tikz[]{}}=\NB{\tikz[]{}}, \qquad
          \NB{\tikz[]{}}=\NB{\tikz[]{}}, \qquad
          \NB{\tikz[rotate=180]{}}=\NB{\tikz[rotate=180]{}}, \qquad
          \NB{\tikz[rotate=180]{}}=\NB{\tikz[rotate=180]{}}, \\
          \NB{\tikz[scale=0.8]{}} =\NB{\tikz[scale=0.8]{}}, \qquad
          \NB{\tikz[scale=0.8]{}} = \NB{\tikz[scale=0.8]{}}, \qquad
          \NB{\tikz[scale=0.8, yscale=-1]{}}=
          \NB{\tikz[scale=0.8, yscale=-1]{}}, \qquad
          \NB{\tikz[scale=0.8, yscale=-1]{}}=
          \NB{\tikz[scale=0.8, yscale=-1]{}}.
        \end{gather}
\end{notation}
      \begin{lem}\label{lem:dual-even-odd}
           The following relations holds:
           \begin{gather}
             \label{eq:39}
             \NB{\tikz[xscale=-1]{\begin{scope}
  \draw[expand style=\styleeven] (0, +0.5) -- (0,0) \rarreven arc
  (0:-180:0.25cm) arc (0:180:0.25cm) -- +(0, -0.5) \rarreven;
\end{scope}}} =
             \NB{\tikz[]{}} =
             \NB{\tikz[]{}},\qquad \qquad
             \NB{\tikz[xscale=-1]{\begin{scope}
  \draw[expand style=\styleeven] (0, +0.5) -- (0,0) \arreven arc
  (0:-180:0.25cm) arc (0:180:0.25cm) -- +(0, -0.5) \arreven;
\end{scope}}} =
             \NB{\tikz[]{}} =
             \NB{\tikz[]{}},\\
             \label{eq:40}
             \NB{\tikz[xscale=-1]{\begin{scope}
  \draw[expand style=\styleodd] (0, +0.5) -- (0,0) \rarrodd arc
  (0:-180:0.25cm) arc (0:180:0.25cm) -- +(0, -0.5) \rarrodd;
\end{scope}}} =
             \NB{\tikz[]{}} =
             \NB{\tikz[]{}},\qquad \qquad
             \NB{\tikz[xscale=-1]{\begin{scope}
  \draw[expand style=\styleodd] (0, +0.5) -- (0,0) \arrodd arc
  (0:-180:0.25cm) arc (0:180:0.25cm) -- +(0, -0.5) \arrodd;
\end{scope}}} =
             \NB{\tikz[]{}} =
             \NB{\tikz[]{}}.
           \end{gather}

      \end{lem}
      \begin{proof}
          This follows from combining Lemmas \ref{lem:odd-even-duality}, \ref{L:N-when-is-zero}, Lemma \ref{L:S-zig-zag}.
        \end{proof}
        \begin{rmk}\label{rmk:two-by-two}
           For each $\myN$, the relations of Lemma~\ref{lem:dual-even-odd} identify ``two-by-two". When $\myN$ is even, we erase orientation, e.g. the two relations in \eqref{eq:39} become one relation. When $\myN$ is odd, we ignore color but reverse orientation on the odd strands, e.g. the first relation in \eqref{eq:39} and the second relation in \eqref{eq:40} become one relation.
        \end{rmk}

      Remark \ref{rmk:two-by-two} applies to most of the forthcoming
      lemmas. In addition, we are not exhaustive, in terms of color and orientation, when displaying relations.
        
      \begin{lem}
        The following relations hold:
        \begin{align}
          \label{eq:unknot-v-diag}
          \NB{\tikz[]{}} &= ([2\myN-1]
 +1)\mathrm{id}_{\mathbb{C}(q)} \\
                    \label{eq:unknot-v-spins}
          {\NB{\tikz[]{}}}&={\NB{\tikz[]{}}}={\NB{\tikz[]{}}}={\NB{\tikz[]{}}}= (-1)^{\binom{\myN}{2}}[2]^{[\myN-1]}\mathrm{id}_{\mathbb{C}(q)}.
        \end{align}
      \end{lem}

      \begin{proof}
The proofs are direct calculations. For instance, let us
          prove compute the first circle on the second line in the
          case when $\myN$ is even:
          Using formulas in \ref{R:explicit-cap-cup} we find
          \[
          \capp_{\macroeven,\macroeven}\circ \cupp_{\macroeven,\macroeven}(1) =  \sum_{\substack{I\subset \setN \\ |I|\in 2\mathbb{Z}}}(-q^{-1})^I(-q)^{\setN\setminus I} = (-1)^{\binom{\myN}{2}}[2]^{[\myN-1]}.
          \]
          Since $\End_{U_q(\sotN)}(\mathbb{C}(q)$ is spanned by $\id_{\mathbb{C}(q)}$, it follows that
          \[
          \capp_{\macroeven,\macroeven}\circ \cupp_{\macroeven,\macroeven} =
          (-1)^{\binom{\myN}{2}}[2]^{[\myN-1]} \id_{\mathbb{C}(q)}. \qedhere
          \]
      \end{proof}

      \begin{lem}
        The following relations hold:
        \begin{equation}
          \label{eq:42}
          \NB{\tikz[]{\begin{scope}
  \draw[expand style= \stylevect] (0,-0.5) -- (0,-0.3);
  \draw[expand style= \stylevect] (0, 0.5) -- (0, 0.3);
  \draw[expand style= \styleodd] (0,-0.3) ..  controls +( 0.3,0.2) and +( 0.3,-0.2) .. (0,0.3) \rarrodd;
  \draw[expand style= \styleeven] (0,-0.3) ..  controls +( -0.3,0.2) and +( -0.3,-0.2) .. (0,0.3) \arreven;
\end{scope}
}} = \NB{\tikz[xscale=-1]{}} =
          \NB{\tikz[]{\begin{scope}
  \draw[expand style= \stylevect] (0,-0.5) -- (0,-0.3);
  \draw[expand style= \stylevect] (0, 0.5) -- (0, 0.3);
  \draw[expand style= \styleodd] (0,-0.3) ..  controls +( 0.3,0.2) and +( 0.3,-0.2) .. (0,0.3) \arrodd;
  \draw[expand style= \styleeven] (0,-0.3) ..  controls +( -0.3,0.2) and +( -0.3,-0.2) .. (0,0.3) \rarreven;
\end{scope}
}} = \NB{\tikz[xscale=-1]{}} =
          (-1)^{\binom{\myN-1}{2}}[2]^{[\myN-2]} \NB{\tikz[]{}}.
        \end{equation}
      \end{lem}
      \begin{proof}
Proofs are direct calculations using the formulas in
        Definition \ref{D:YVSS} and Definition \ref{D:PVSS}. For the
        first bigon, in the case, when $\myN$ is even, this  gives
          \[
          \Pa_{\macroeven,\macroodd}^{V}\circ \Ya_{V}^{\macroeven,\macroodd}(v_1) = \sum_{\substack{I\subset \{2, \dots, \myN\} \\ |I|\in 2\mathbb{Z}}}(-q^{-1})^{I}(-q)^{\{2, \dots, \myN\}\setminus I}\cdot v_1 = (-1)^{\binom{\myN-1}{2}}[2]^{[\myN-2]}v_1.
          \]
          Since $\End_{U_q(\sotN)}(V)$ is spanned by $\id_V$, it follows that
          \[
          \Pa_{\macroeven,\macroodd}^V\circ \Ya_{V}^{\macroeven,\macroodd} = (-1)^{\binom{\myN-1}{2}}[2]^{[\myN-2]}\id_V.
          \qedhere\]
      \end{proof}

      \begin{lem}The following relations hold:
        \begin{align}
          \NB{\tikz[]{\begin{scope}
  \draw[expand style= \styleodd] (0,-0.7) -- (0,-0.3) \arrodd;
  \draw[expand style= \styleodd] (0, 0.7) -- (0, 0.3) \rarrodd;
  \draw[expand style= \styleeven] (0,-0.3) ..  controls +( 0.3,0.2) and +( 0.3,-0.2) .. (0,0.3) \arreven;
  \draw[expand style= \stylevect] (0,-0.3) ..  controls +( -0.3,0.2) and +( -0.3,-0.2) .. (0,0.3);
\end{scope}
}}=
          \NB{\tikz[xscale=-1]{}}=(-1)^{\myN-1}[\myN]
          \NB{\tikz[yscale=1.4]{}}, \qquad
          \NB{\tikz[]{\begin{scope}
  \draw[expand style= \styleodd] (0,-0.7) -- (0,-0.3) \rarrodd;
  \draw[expand style= \styleodd] (0, 0.7) -- (0, 0.3) \arrodd;
  \draw[expand style= \styleeven] (0,-0.3) ..  controls +( 0.3,0.2) and +( 0.3,-0.2) .. (0,0.3) \rarreven;
  \draw[expand style= \stylevect] (0,-0.3) ..  controls +( -0.3,0.2) and +( -0.3,-0.2) .. (0,0.3);
\end{scope}
}}=
          \NB{\tikz[xscale=-1]{}}=(-1)^{\myN-1}[\myN]
          \NB{\tikz[yscale=1.4]{}}, \\
          \NB{\tikz[]{\begin{scope}
  \draw[expand style= \styleeven] (0,-0.7) -- (0,-0.3) \arreven;
  \draw[expand style= \styleeven] (0, 0.7) -- (0, 0.3) \rarreven;
  \draw[expand style= \styleodd] (0,-0.3) ..  controls +( 0.3,0.2) and +( 0.3,-0.2) .. (0,0.3) \arrodd;
  \draw[expand style= \stylevect] (0,-0.3) ..  controls +( -0.3,0.2) and +( -0.3,-0.2) .. (0,0.3);
\end{scope}
}}=
          \NB{\tikz[xscale=-1]{}}=(-1)^{\myN-1}[\myN]
          \NB{\tikz[yscale=1.4]{}}, \qquad
          \NB{\tikz[]{\begin{scope}
  \draw[expand style= \styleeven] (0,-0.7) -- (0,-0.3) \rarreven;
  \draw[expand style= \styleeven] (0, 0.7) -- (0, 0.3) \arreven;
  \draw[expand style= \styleodd] (0,-0.3) ..  controls +( 0.3,0.2) and +( 0.3,-0.2) .. (0,0.3) \rarrodd;
  \draw[expand style= \stylevect] (0,-0.3) ..  controls +( -0.3,0.2) and +( -0.3,-0.2) .. (0,0.3);
\end{scope}
}}=
          \NB{\tikz[xscale=-1]{}}=(-1)^{\myN-1}[\myN]
          \NB{\tikz[yscale=1.4]{}}.
        \end{align}
      \end{lem}
      \begin{proof}
        These are direct computations, we only prove the first one, in
        the case $\myN$ is even.
Since $\End_{U_q(\sotN)}(S_{\macroodd})$ is spanned by $\id_{S_{\macroodd}}$, it follows that there is some $\xi\in \mathbb{C}(q)$ such that
          \[
          \Pa_{V,\macroeven}^{\macroodd}\circ \Ya_{\macroodd}^{V,\macroeven} = \xi\id_{\macroodd}.
          \]
          Using graphical calculus, it is easy to check that 
          \[
          \xi(-1)^{\binom{\myN}{2}}[2]^{[\myN-1]} = (-1)^{\binom{\myN-1}{2}}[2]^{[\myN-2]}\frac{[\myN][2(\myN-1)]}{[\myN-1]}, 
          \]
          so $\xi = (-1)^{\myN-1}[\myN]$.
      \end{proof}

      \begin{lem} The following relations hold:
        \begin{align}
        \label{triangle1}
          {\NB{\tikz[scale=0.5]{\input{\imagesfolder/td-diag-triangle-up-o-e}}}}
          &= (-1)^\myN[\myN-1] {\NB{\tikz[yscale=0.85, xscale=1]{}}}, 
            &&&\qquad 
            {\NB{\tikz[xscale=1, scale=0.5]{\input{\imagesfolder/td-diag-triangle-up-e-o}}}}
          &= (-1)^\myN[\myN-1] {\NB{\tikz[yscale=0.85, xscale=1]{}}},  \\               {\NB{\tikz[scale=0.5]{\input{\imagesfolder/td-diag-triangle-down-o-e}}}}
          &= (-1)^\myN[\myN-1] {\NB{\tikz[yscale=0.85, xscale=1]{}}}, 
            &&&\qquad 
            {\NB{\tikz[xscale=1, scale=0.5]{\input{\imagesfolder/td-diag-triangle-down-e-o}}}}
          &= (-1)^\myN[\myN-1] {\NB{\tikz[yscale=0.85, xscale=1]{}}}.
        \end{align}
\end{lem}
      \begin{proof}
        We prove in detail the first identity in \eqref{triangle1}. The morphism represented in graphical calculus by the ``triangle" is equal to the composition
          \begin{equation}\label{eqn:composition-triangle}
          (\capp_{\macroeven}\otimes\id_{\macroodd}\otimes\id_{\macroeven}\otimes\capp_{\macroodd})\circ(\id_{\macroeven}\otimes \Ya_V^{\macroeven,\macroodd}\otimes \Ya_V^{\macroeven,\macroodd}\otimes \id_{\macroodd})\circ(\id_{\macroeven}\otimes \cupp_V\otimes \id_{\macroodd})\circ\Ya_V^{\macroeven,\macroodd},
        \end{equation}
        This is easier seen diagrammatically:
        \[
          \NB{\tikz[scale=1.5]{\input{\imagesfolder/td-diag-triangle-up-o-e}}}:=\NB{\tikz[scale=1.5]{\input{\imagesfolder/td-diag-triangle-big}}}
        \]
        
          The map in equation \eqref{eqn:composition-triangle} is an element of the $1$ dimensional vector space $\Hom_{U_q(\sotN)}(V, S_{\macroodd}\otimes S_{\macroeven})$, which is spanned by $\Ya_{V}^{\macroodd,\macroeven}$. Thus, there is some scalar, say $\xi$, such that equation \eqref{eqn:composition-triangle} is equal to $\xi\cdot \Ya_{V}^{\macroodd,\macroeven}$. 
          
          Using Definition \ref{D:YVSS}, it is easy to see that $\Ya_V^{\macroodd,\macroeven}$ maps $v_1$ to $1\cdot x_{\{2, \dots, \myN\}}\otimes x_{\emptyset}$ plus terms of the form $x_I\otimes x_J$ where $J\ne \emptyset$. Thus, $\xi$ is the coefficient of the image of $v_1$ under the map in equation \eqref{eqn:composition-triangle}. Using the explicit formulas in Remark \ref{R:explicit-cap-cup} and Definition \ref{D:YVSS}, it follows that $\xi$ is equal to 
          \[
          \sum_{k=2}^\myN (-q)^{k-1}(-q)^{\{k\}}\frac{(-q)^{\sigma_k\{1, \dots, \hat{k}, \dots, \myN\}}(-q)^{\sigma_k\{2, \dots,\hat{k}, \dots, \myN\}}}{(-q)^{\{1, \dots, \hat{k}, \dots, \myN\}}(-q)^{\{2, \dots, \hat{k}, \dots, \myN\}}}
          \]
          which simplifies to be
          \[
          \sum_{k=2}^\myN(-q)^{\myN-1}(-q)^{(\myN-k)-(\myN-1)}(-q)^{(\myN-k)-(\myN-2)} = \sum_{k=1}^{\myN-1} (-q)^{\myN-2k} = (-1)^\myN[\myN-1].
          \]
      \end{proof}

      \begin{lem}\label{lem:braid-and-trivalent}
        The following relations hold:\begin{align}
          (-q)^{\binom{\myN-1}{2}}q^{\frac{\myN-2}{4}}\NB{\tikz[]{\begin{scope}[scale=1]
  \draw[expand style= \styleeven] (-0.25, 1.5) .. controls +(0, -0.2)
  and +(-0.1,0.1)
  .. (0, 1.25)\rarreven .. controls +(0.1,-0.1) and (0.25, 1.1) .. (0.25, 1) .. controls +(0, -0.25)
  .. (0, 0.5) \arreven;
  \fill[white] (0, 1.25) circle (1mm);
  \draw[expand style= \styleodd] (0.25, 1.5) .. controls +(0, -0.2)
  and +(0.1,0.1)
  .. (0, 1.25)\rarrodd .. controls +(-0.1,-0.1) and (-0.25, 1.1) .. (-0.25, 1) .. controls +(0, -0.25)
  .. (0, 0.5) \arrodd;
  % \draw[expand style= \styleeven] (0, 0.5)
  % .. controls +(0.25,0.25) .. (0.25, 1) \arreven;
  \draw[expand style =\stylevect] (0,0.5) -- (0,0);
\end{scope}}} &=\NB{\tikz[yscale=1.5]{}}=(-q^{-1})^{\binom{\myN-1}{2}}q^{-\frac{\myN-2}{4}}\NB{\tikz[]{\begin{scope}[scale=1]
  \draw[expand style= \styleodd] (0.25, 1.5) .. controls +(0, -0.2)
  and +(0.1,0.1)
  .. (0, 1.25)\rarrodd .. controls +(-0.1,-0.1) and (-0.25, 1.1) .. (-0.25, 1) .. controls +(0, -0.25)
  .. (0, 0.5) \arrodd;
  \fill[white] (0, 1.25) circle (1mm);
  \draw[expand style= \styleeven] (-0.25, 1.5) .. controls +(0, -0.2)
  and +(-0.1,0.1)
  .. (0, 1.25)\rarreven .. controls +(0.1,-0.1) and (0.25, 1.1) .. (0.25, 1) .. controls +(0, -0.25)
  .. (0, 0.5) \arreven;
  % \draw[expand style= \styleeven] (0, 0.5)
  % .. controls +(0.25,0.25) .. (0.25, 1) \arreven;
  \draw[expand style =\stylevect] (0,0.5) -- (0,0);
\end{scope}}}, \label{eq:YXing}
                                 \\ (-q)^{\binom{\myN-1}{2}}q^{\frac{\myN-2}{4}}\NB{\tikz[]{\begin{scope}[scale=1]
  \draw[expand style= \styleodd] (-0.25, 1.5) .. controls +(0, -0.2)
  and +(-0.1,0.1)
  .. (0, 1.25)\rarrodd .. controls +(0.1,-0.1) and (0.25, 1.1) .. (0.25, 1) .. controls +(0, -0.25)
  .. (0, 0.5) \arrodd;
  \fill[white] (0, 1.25) circle (1mm);
  \draw[expand style= \styleeven] (0.25, 1.5) .. controls +(0, -0.2)
  and +(0.1,0.1)
  .. (0, 1.25)\rarreven .. controls +(-0.1,-0.1) and (-0.25, 1.1) .. (-0.25, 1) .. controls +(0, -0.25)
  .. (0, 0.5) \arreven;
  % \draw[expand style= \styleeven] (0, 0.5)
  % .. controls +(0.25,0.25) .. (0.25, 1) \arreven;
  \draw[expand style =\stylevect] (0,0.5) -- (0,0);
\end{scope}}} &=\NB{\tikz[yscale=1.5]{}}=(-q^{-1})^{\binom{\myN-1}{2}}q^{-\frac{\myN-2}{4}}\NB{\tikz[]{\begin{scope}[scale=1]
  \draw[expand style= \styleeven] (0.25, 1.5) .. controls +(0, -0.2)
  and +(0.1,0.1)
  .. (0, 1.25)\rarreven .. controls +(-0.1,-0.1) and (-0.25, 1.1) .. (-0.25, 1) .. controls +(0, -0.25)
  .. (0, 0.5) \arreven;
  \fill[white] (0, 1.25) circle (1mm);
  \draw[expand style= \styleodd] (-0.25, 1.5) .. controls +(0, -0.2)
  and +(-0.1,0.1)
  .. (0, 1.25)\rarrodd .. controls +(0.1,-0.1) and (0.25, 1.1) .. (0.25, 1) .. controls +(0, -0.25)
  .. (0, 0.5) \arrodd;
  % \draw[expand style= \styleeven] (0, 0.5)
  % .. controls +(0.25,0.25) .. (0.25, 1) \arreven;
  \draw[expand style =\stylevect] (0,0.5) -- (0,0);
\end{scope}}}, \\
          (-q)^{\binom{\myN-1}{2}}q^{\frac{\myN-2}{4}}\NB{\tikz[]{\begin{scope}[scale=1]
  \draw[expand style= \styleeven] (-0.25, 1.5) .. controls +(0, -0.2)
  and +(-0.1,0.1)
  .. (0, 1.25)\arreven .. controls +(0.1,-0.1) and (0.25, 1.1) .. (0.25, 1) .. controls +(0, -0.25)
  .. (0, 0.5) \rarreven;
  \fill[white] (0, 1.25) circle (1mm);
  \draw[expand style= \styleodd] (0.25, 1.5) .. controls +(0, -0.2)
  and +(0.1,0.1)
  .. (0, 1.25)\arrodd .. controls +(-0.1,-0.1) and (-0.25, 1.1) .. (-0.25, 1) .. controls +(0, -0.25)
  .. (0, 0.5) \rarrodd;
  % \draw[expand style= \styleeven] (0, 0.5)
  % .. controls +(0.25,0.25) .. (0.25, 1) \arreven;
  \draw[expand style =\stylevect] (0,0.5) -- (0,0);
\end{scope}}} &=\NB{\tikz[yscale=1.5]{}}=(-q^{-1})^{\binom{\myN-1}{2}}q^{-\frac{\myN-2}{4}}\NB{\tikz[]{\begin{scope}[scale=1]
  \draw[expand style= \styleodd] (0.25, 1.5) .. controls +(0, -0.2)
  and +(0.1,0.1)
  .. (0, 1.25)\arrodd .. controls +(-0.1,-0.1) and (-0.25, 1.1) .. (-0.25, 1) .. controls +(0, -0.25)
  .. (0, 0.5) \rarrodd;
  \fill[white] (0, 1.25) circle (1mm);
  \draw[expand style= \styleeven] (-0.25, 1.5) .. controls +(0, -0.2)
  and +(-0.1,0.1)
  .. (0, 1.25)\arreven .. controls +(0.1,-0.1) and (0.25, 1.1) .. (0.25, 1) .. controls +(0, -0.25)
  .. (0, 0.5) \rarreven;
  % \draw[expand style= \styleeven] (0, 0.5)
  % .. controls +(0.25,0.25) .. (0.25, 1) \arreven;
  \draw[expand style =\stylevect] (0,0.5) -- (0,0);
\end{scope}}}, 
                                 \\ (-q)^{\binom{\myN-1}{2}}q^{\frac{\myN-2}{4}}\NB{\tikz[]{\begin{scope}[scale=1]
  \draw[expand style= \styleodd] (-0.25, 1.5) .. controls +(0, -0.2)
  and +(-0.1,0.1)
  .. (0, 1.25)\arrodd .. controls +(0.1,-0.1) and (0.25, 1.1) .. (0.25, 1) .. controls +(0, -0.25)
  .. (0, 0.5) \rarrodd;
  \fill[white] (0, 1.25) circle (1mm);
  \draw[expand style= \styleeven] (0.25, 1.5) .. controls +(0, -0.2)
  and +(0.1,0.1)
  .. (0, 1.25)\arreven .. controls +(-0.1,-0.1) and (-0.25, 1.1) .. (-0.25, 1) .. controls +(0, -0.25)
  .. (0, 0.5) \rarreven;
  % \draw[expand style= \styleeven] (0, 0.5)
  % .. controls +(0.25,0.25) .. (0.25, 1) \arreven;
  \draw[expand style =\stylevect] (0,0.5) -- (0,0);
\end{scope}}} &= \NB{\tikz[yscale=1.5]{}}=(-q^{-1})^{\binom{\myN-1}{2}}q^{-\frac{\myN-2}{4}}\NB{\tikz[]{\begin{scope}[scale=1]
  \draw[expand style= \styleeven] (0.25, 1.5) .. controls +(0, -0.2)
  and +(0.1,0.1)
  .. (0, 1.25)\arreven .. controls +(-0.1,-0.1) and (-0.25, 1.1) .. (-0.25, 1) .. controls +(0, -0.25)
  .. (0, 0.5) \rarreven;
  \fill[white] (0, 1.25) circle (1mm);
  \draw[expand style= \styleodd] (-0.25, 1.5) .. controls +(0, -0.2)
  and +(-0.1,0.1)
  .. (0, 1.25)\arrodd .. controls +(0.1,-0.1) and (0.25, 1.1) .. (0.25, 1) .. controls +(0, -0.25)
  .. (0, 0.5) \rarrodd;
  % \draw[expand style= \styleeven] (0, 0.5)
  % .. controls +(0.25,0.25) .. (0.25, 1) \arreven;
  \draw[expand style =\stylevect] (0,0.5) -- (0,0);
\end{scope}}}.
        \end{align}
      \end{lem}
      \begin{proof}
        We leave it as an easy exercise to adapt the proof of the analogous result for $U_q(\mathfrak{so}_{2\myN+1})$ \cite[Lemma 4.11]{spinlinkhomology} to prove that e.g.
          \begin{align*}
          R_{\macroeven,\macroodd}\circ \Ya_{V}^{\macroeven,\macroodd} &= (-q^{-1})^{\{2, \dots, \myN\}}q^{-\frac{\myN-2}{4}}\Ya_{V}^{\macroodd,\macroeven}. \\
          &=(-q^{-1})^{\binom{\myN-1}{2}}q^{-\frac{\myN-2}{4}}\Ya_{V}^{\macroodd,\macroeven},
          \end{align*}
          which implies the second equality in \eqref{eq:YXing}. The other identities are similar.
      \end{proof}
      \begin{lem}
        The following relations hold: \allowdisplaybreaks
        \begin{gather}
(-1)^{\myN-1}q^{\frac{2\myN-1}{2}}\NB{\tikz[]{\begin{scope}[scale=1]
  \draw[expand style= \styleeven] (-0.25, 1.5) .. controls +(0, -0.2)
  and +(-0.1,0.1)
  .. (0, 1.25)\rarreven .. controls +(0.1,-0.1) and (0.25, 1.1) .. (0.25, 1) .. controls +(0, -0.25)
  .. (0, 0.5) \arreven;
  \fill[white] (0, 1.25) circle (1mm);
  \draw[expand style= \stylevect] (0.25, 1.5) .. controls +(0, -0.2)
  and +(0.1,0.1)
  .. (0, 1.25) .. controls +(-0.1,-0.1) and (-0.25, 1.1) .. (-0.25, 1) .. controls +(0, -0.25)
  .. (0, 0.5);
  % \draw[expand style= \styleeven] (0, 0.5)
  % .. controls +(0.25,0.25) .. (0.25, 1) \arreven;
  \draw[expand style =\styleodd] (0,0.5) -- (0,0) \rarrodd;
\end{scope}}}   =\NB{\tikz[yscale=1.5]{\begin{scope}
  \draw[expand style= \styleeven] (-0.25, 1) .. controls +(0, -0.25)  .. (0, 0.5) \rarreven;
  \draw[expand style= \stylevect]
  (0, 0.5) .. controls +(0.25, 0.25) .. (0.25, 1);
  \draw[expand style =\styleodd] (0,0.5) -- (0,0) \rarrodd;
\end{scope}}}   =  (-1)^{\myN-1}q^{\frac{1-2\myN}{2}}\NB{\tikz[]{\begin{scope}[scale=1]
  \draw[expand style= \stylevect] (0.25, 1.5) .. controls +(0, -0.2)
  and +(0.1,0.1)
  .. (0, 1.25) .. controls +(-0.1,-0.1) and (-0.25, 1.1) .. (-0.25, 1) .. controls +(0, -0.25)
  .. (0, 0.5);
  \fill[white] (0, 1.25) circle (1mm);
  \draw[expand style= \styleeven] (-0.25, 1.5) .. controls +(0, -0.2)
  and +(-0.1,0.1)
  .. (0, 1.25)\rarreven .. controls +(0.1,-0.1) and (0.25, 1.1) .. (0.25, 1) .. controls +(0, -0.25)
  .. (0, 0.5) \arreven;
  % \draw[expand style= \styleeven] (0, 0.5)
  % .. controls +(0.25,0.25) .. (0.25, 1) \arreven;
  \draw[expand style =\styleodd] (0,0.5) -- (0,0) \rarrodd;
\end{scope}}}   ,\\
(-1)^{\myN-1}q^{\frac{2\myN-1}{2}}\NB{\tikz[]{\begin{scope}[scale=1]
  \draw[expand style= \styleodd] (-0.25, 1.5) .. controls +(0, -0.2)
  and +(-0.1,0.1)
  .. (0, 1.25)\rarrodd .. controls +(0.1,-0.1) and (0.25, 1.1) .. (0.25, 1) .. controls +(0, -0.25)
  .. (0, 0.5) \arrodd;
  \fill[white] (0, 1.25) circle (1mm);
  \draw[expand style= \stylevect] (0.25, 1.5) .. controls +(0, -0.2)
  and +(0.1,0.1)
  .. (0, 1.25) .. controls +(-0.1,-0.1) and (-0.25, 1.1) .. (-0.25, 1) .. controls +(0, -0.25)
  .. (0, 0.5);
  % \draw[expand style= \styleeven] (0, 0.5)
  % .. controls +(0.25,0.25) .. (0.25, 1) \arreven;
  \draw[expand style =\styleeven] (0,0.5) -- (0,0) \rarreven;
\end{scope}}}   =\NB{\tikz[yscale=1.5]{\begin{scope}
  \draw[expand style= \styleodd] (-0.25, 1) .. controls +(0, -0.25)  .. (0, 0.5) \rarrodd;
  \draw[expand style= \stylevect]
  (0, 0.5) .. controls +(0.25, 0.25) .. (0.25, 1);
  \draw[expand style =\styleeven] (0,0.5) -- (0,0) \rarreven;
\end{scope}}}   =  (-1)^{\myN-1}q^{\frac{1-2\myN}{2}}\NB{\tikz[]{\begin{scope}[scale=1]
  \draw[expand style= \stylevect] (0.25, 1.5) .. controls +(0, -0.2)
  and +(0.1,0.1)
  .. (0, 1.25) .. controls +(-0.1,-0.1) and (-0.25, 1.1) .. (-0.25, 1) .. controls +(0, -0.25)
  .. (0, 0.5);
  \fill[white] (0, 1.25) circle (1mm);
  \draw[expand style= \styleodd] (-0.25, 1.5) .. controls +(0, -0.2)
  and +(-0.1,0.1)
  .. (0, 1.25)\rarrodd .. controls +(0.1,-0.1) and (0.25, 1.1) .. (0.25, 1) .. controls +(0, -0.25)
  .. (0, 0.5) \arrodd;
  % \draw[expand style= \styleeven] (0, 0.5)
  % .. controls +(0.25,0.25) .. (0.25, 1) \arreven;
  \draw[expand style =\styleeven] (0,0.5) -- (0,0) \rarreven;
\end{scope}}}   ,\\
(-1)^{\myN-1}q^{\frac{2\myN-1}{2}}\NB{\tikz[]{\begin{scope}[scale=1]
  \draw[expand style= \stylevect] (-0.25, 1.5) .. controls +(0, -0.2)
  and +(-0.1,0.1)
  .. (0, 1.25) .. controls +(0.1,-0.1) and (0.25, 1.1) .. (0.25, 1) .. controls +(0, -0.25)
  .. (0, 0.5);
  \fill[white] (0, 1.25) circle (1mm);
  \draw[expand style= \styleeven] (0.25, 1.5) .. controls +(0, -0.2)
  and +(0.1,0.1)
  .. (0, 1.25)\arreven .. controls +(-0.1,-0.1) and (-0.25, 1.1) .. (-0.25, 1) .. controls +(0, -0.25)
  .. (0, 0.5)\rarreven;
  % \draw[expand style= \styleeven] (0, 0.5)
  % .. controls +(0.25,0.25) .. (0.25, 1) \arreven;
  \draw[expand style =\styleodd] (0,0.5) -- (0,0) \rarrodd;
\end{scope}}}   =\NB{\tikz[yscale=1.5]{\begin{scope}
  \draw[expand style= \stylevect] (-0.25, 1) .. controls +(0, -0.25)  .. (0, 0.5); 
  \draw[expand style= \styleeven]
  (0, 0.5) .. controls +(0.25, 0.25) .. (0.25, 1) \arreven;
  \draw[expand style =\styleodd] (0,0.5) -- (0,0) \rarrodd;
\end{scope}}}   =  (-1)^{\myN-1}q^{\frac{1-2\myN}{2}}\NB{\tikz[]{\begin{scope}[scale=1]
  \draw[expand style= \styleeven] (0.25, 1.5) .. controls +(0, -0.2)
  and +(0.1,0.1)
  .. (0, 1.25)\arreven .. controls +(-0.1,-0.1) and (-0.25, 1.1) .. (-0.25, 1) .. controls +(0, -0.25)
  .. (0, 0.5)\rarreven;
  \fill[white] (0, 1.25) circle (1mm);
  \draw[expand style= \stylevect] (-0.25, 1.5) .. controls +(0, -0.2)
  and +(-0.1,0.1)
  .. (0, 1.25) .. controls +(0.1,-0.1) and (0.25, 1.1) .. (0.25, 1) .. controls +(0, -0.25)
  .. (0, 0.5);
  % \draw[expand style= \styleeven] (0, 0.5)
  % .. controls +(0.25,0.25) .. (0.25, 1) \arreven;
  \draw[expand style =\styleodd] (0,0.5) -- (0,0) \rarrodd;
\end{scope}}}   ,\\
(-1)^{\myN-1}q^{\frac{2\myN-1}{2}}\NB{\tikz[]{\begin{scope}[scale=1]
  \draw[expand style= \stylevect] (-0.25, 1.5) .. controls +(0, -0.2)
  and +(-0.1,0.1)
  .. (0, 1.25) .. controls +(0.1,-0.1) and (0.25, 1.1) .. (0.25, 1) .. controls +(0, -0.25)
  .. (0, 0.5);
  \fill[white] (0, 1.25) circle (1mm);
  \draw[expand style= \styleodd] (0.25, 1.5) .. controls +(0, -0.2)
  and +(0.1,0.1)
  .. (0, 1.25)\arrodd .. controls +(-0.1,-0.1) and (-0.25, 1.1) .. (-0.25, 1) .. controls +(0, -0.25)
  .. (0, 0.5)\rarrodd;
  % \draw[expand style= \styleeven] (0, 0.5)
  % .. controls +(0.25,0.25) .. (0.25, 1) \arreven;
  \draw[expand style =\styleeven] (0,0.5) -- (0,0) \rarreven;
\end{scope}}}   =\NB{\tikz[yscale=1.5]{\begin{scope}
  \draw[expand style= \stylevect] (-0.25, 1) .. controls +(0, -0.25)  .. (0, 0.5); 
  \draw[expand style= \styleodd]
  (0, 0.5) .. controls +(0.25, 0.25) .. (0.25, 1) \arrodd;
  \draw[expand style =\styleeven] (0,0.5) -- (0,0) \rarreven;
\end{scope}}}   =  (-1)^{\myN-1}q^{\frac{1-2\myN}{2}}\NB{\tikz[]{\begin{scope}[scale=1]
  \draw[expand style= \styleodd] (0.25, 1.5) .. controls +(0, -0.2)
  and +(0.1,0.1)
  .. (0, 1.25)\arrodd .. controls +(-0.1,-0.1) and (-0.25, 1.1) .. (-0.25, 1) .. controls +(0, -0.25)
  .. (0, 0.5)\rarrodd;
  \fill[white] (0, 1.25) circle (1mm);
  \draw[expand style= \stylevect] (-0.25, 1.5) .. controls +(0, -0.2)
  and +(-0.1,0.1)
  .. (0, 1.25) .. controls +(0.1,-0.1) and (0.25, 1.1) .. (0.25, 1) .. controls +(0, -0.25)
  .. (0, 0.5);
  % \draw[expand style= \styleeven] (0, 0.5)
  % .. controls +(0.25,0.25) .. (0.25, 1) \arreven;
  \draw[expand style =\styleeven] (0,0.5) -- (0,0) \rarreven;
\end{scope}}}   ,\\
(-1)^{\myN-1}q^{\frac{2\myN-1}{2}}\NB{\tikz[]{\begin{scope}[scale=1]
  \draw[expand style= \styleeven] (-0.25, 1.5) .. controls +(0, -0.2)
  and +(-0.1,0.1)
  .. (0, 1.25)\arreven .. controls +(0.1,-0.1) and (0.25, 1.1) .. (0.25, 1) .. controls +(0, -0.25)
  .. (0, 0.5) \rarreven;
  \fill[white] (0, 1.25) circle (1mm);
  \draw[expand style= \stylevect] (0.25, 1.5) .. controls +(0, -0.2)
  and +(0.1,0.1)
  .. (0, 1.25) .. controls +(-0.1,-0.1) and (-0.25, 1.1) .. (-0.25, 1) .. controls +(0, -0.25)
  .. (0, 0.5);
  % \draw[expand style= \styleeven] (0, 0.5)
  % .. controls +(0.25,0.25) .. (0.25, 1) \arreven;
  \draw[expand style =\styleodd] (0,0.5) -- (0,0) \arrodd;
\end{scope}}} =\NB{\tikz[yscale=1.5]{\begin{scope}
  \draw[expand style= \styleeven] (-0.25, 1) .. controls +(0, -0.25)  .. (0, 0.5) \arreven;
  \draw[expand style= \stylevect]
  (0, 0.5) .. controls +(0.25, 0.25) .. (0.25, 1);
  \draw[expand style =\styleodd] (0,0.5) -- (0,0) \arrodd;
\end{scope}}} =  (-1)^{\myN-1}q^{\frac{1-2\myN}{2}}\NB{\tikz[]{\begin{scope}[scale=1]
  \draw[expand style= \stylevect] (0.25, 1.5) .. controls +(0, -0.2)
  and +(0.1,0.1)
  .. (0, 1.25) .. controls +(-0.1,-0.1) and (-0.25, 1.1) .. (-0.25, 1) .. controls +(0, -0.25)
  .. (0, 0.5);
  \fill[white] (0, 1.25) circle (1mm);
  \draw[expand style= \styleeven] (-0.25, 1.5) .. controls +(0, -0.2)
  and +(-0.1,0.1)
  .. (0, 1.25)\arreven .. controls +(0.1,-0.1) and (0.25, 1.1) .. (0.25, 1) .. controls +(0, -0.25)
  .. (0, 0.5) \rarreven;
  % \draw[expand style= \styleeven] (0, 0.5)
  % .. controls +(0.25,0.25) .. (0.25, 1) \arreven;
  \draw[expand style =\styleodd] (0,0.5) -- (0,0) \arrodd;
\end{scope}}} ,\\
(-1)^{\myN-1}q^{\frac{2\myN-1}{2}}\NB{\tikz[]{\begin{scope}[scale=1]
  \draw[expand style= \styleodd] (-0.25, 1.5) .. controls +(0, -0.2)
  and +(-0.1,0.1)
  .. (0, 1.25)\arrodd .. controls +(0.1,-0.1) and (0.25, 1.1) .. (0.25, 1) .. controls +(0, -0.25)
  .. (0, 0.5) \rarrodd;
  \fill[white] (0, 1.25) circle (1mm);
  \draw[expand style= \stylevect] (0.25, 1.5) .. controls +(0, -0.2)
  and +(0.1,0.1)
  .. (0, 1.25) .. controls +(-0.1,-0.1) and (-0.25, 1.1) .. (-0.25, 1) .. controls +(0, -0.25)
  .. (0, 0.5);
  % \draw[expand style= \styleeven] (0, 0.5)
  % .. controls +(0.25,0.25) .. (0.25, 1) \arreven;
  \draw[expand style =\styleeven] (0,0.5) -- (0,0) \arreven;
\end{scope}}} =\NB{\tikz[yscale=1.5]{\begin{scope}
  \draw[expand style= \styleodd] (-0.25, 1) .. controls +(0, -0.25)  .. (0, 0.5) \arrodd;
  \draw[expand style= \stylevect]
  (0, 0.5) .. controls +(0.25, 0.25) .. (0.25, 1);
  \draw[expand style =\styleeven] (0,0.5) -- (0,0) \arreven;
\end{scope}}} =  (-1)^{\myN-1}q^{\frac{1-2\myN}{2}}\NB{\tikz[]{\begin{scope}[scale=1]
  \draw[expand style= \stylevect] (0.25, 1.5) .. controls +(0, -0.2)
  and +(0.1,0.1)
  .. (0, 1.25) .. controls +(-0.1,-0.1) and (-0.25, 1.1) .. (-0.25, 1) .. controls +(0, -0.25)
  .. (0, 0.5);
  \fill[white] (0, 1.25) circle (1mm);
  \draw[expand style= \styleodd] (-0.25, 1.5) .. controls +(0, -0.2)
  and +(-0.1,0.1)
  .. (0, 1.25)\arrodd .. controls +(0.1,-0.1) and (0.25, 1.1) .. (0.25, 1) .. controls +(0, -0.25)
  .. (0, 0.5) \rarrodd;
  % \draw[expand style= \styleeven] (0, 0.5)
  % .. controls +(0.25,0.25) .. (0.25, 1) \arreven;
  \draw[expand style =\styleeven] (0,0.5) -- (0,0) \arreven;
\end{scope}}} ,\\
(-1)^{\myN-1}q^{\frac{2\myN-1}{2}}\NB{\tikz[]{\begin{scope}[scale=1]
  \draw[expand style= \stylevect] (-0.25, 1.5) .. controls +(0, -0.2)
  and +(-0.1,0.1)
  .. (0, 1.25) .. controls +(0.1,-0.1) and (0.25, 1.1) .. (0.25, 1) .. controls +(0, -0.25)
  .. (0, 0.5);
  \fill[white] (0, 1.25) circle (1mm);
  \draw[expand style= \styleeven] (0.25, 1.5) .. controls +(0, -0.2)
  and +(0.1,0.1)
  .. (0, 1.25)\rarreven .. controls +(-0.1,-0.1) and (-0.25, 1.1) .. (-0.25, 1) .. controls +(0, -0.25)
  .. (0, 0.5)\arreven;
  % \draw[expand style= \styleeven] (0, 0.5)
  % .. controls +(0.25,0.25) .. (0.25, 1) \arreven;
  \draw[expand style =\styleodd] (0,0.5) -- (0,0) \arrodd;
\end{scope}}} =\NB{\tikz[yscale=1.5]{\begin{scope}
  \draw[expand style= \stylevect] (-0.25, 1) .. controls +(0, -0.25)  .. (0, 0.5); 
  \draw[expand style= \styleeven]
  (0, 0.5) .. controls +(0.25, 0.25) .. (0.25, 1) \rarreven;
  \draw[expand style =\styleodd] (0,0.5) -- (0,0) \arrodd;
\end{scope}}} =  (-1)^{\myN-1}q^{\frac{1-2\myN}{2}}\NB{\tikz[]{\begin{scope}[scale=1]
  \draw[expand style= \styleeven] (0.25, 1.5) .. controls +(0, -0.2)
  and +(0.1,0.1)
  .. (0, 1.25)\rarreven .. controls +(-0.1,-0.1) and (-0.25, 1.1) .. (-0.25, 1) .. controls +(0, -0.25)
  .. (0, 0.5)\arreven;
  \fill[white] (0, 1.25) circle (1mm);
  \draw[expand style= \stylevect] (-0.25, 1.5) .. controls +(0, -0.2)
  and +(-0.1,0.1)
  .. (0, 1.25) .. controls +(0.1,-0.1) and (0.25, 1.1) .. (0.25, 1) .. controls +(0, -0.25)
  .. (0, 0.5);
  % \draw[expand style= \styleeven] (0, 0.5)
  % .. controls +(0.25,0.25) .. (0.25, 1) \arreven;
  \draw[expand style =\styleodd] (0,0.5) -- (0,0) \arrodd;
\end{scope}}} ,\\
(-1)^{\myN-1}q^{\frac{2\myN-1}{2}}\NB{\tikz[]{\begin{scope}[scale=1]
  \draw[expand style= \stylevect] (-0.25, 1.5) .. controls +(0, -0.2)
  and +(-0.1,0.1)
  .. (0, 1.25) .. controls +(0.1,-0.1) and (0.25, 1.1) .. (0.25, 1) .. controls +(0, -0.25)
  .. (0, 0.5);
  \fill[white] (0, 1.25) circle (1mm);
  \draw[expand style= \styleodd] (0.25, 1.5) .. controls +(0, -0.2)
  and +(0.1,0.1)
  .. (0, 1.25)\rarrodd .. controls +(-0.1,-0.1) and (-0.25, 1.1) .. (-0.25, 1) .. controls +(0, -0.25)
  .. (0, 0.5)\arrodd;
  % \draw[expand style= \styleeven] (0, 0.5)
  % .. controls +(0.25,0.25) .. (0.25, 1) \arreven;
  \draw[expand style =\styleeven] (0,0.5) -- (0,0) \arreven;
\end{scope}}} =\NB{\tikz[yscale=1.5]{\begin{scope}
  \draw[expand style= \stylevect] (-0.25, 1) .. controls +(0, -0.25)  .. (0, 0.5); 
  \draw[expand style= \styleodd]
  (0, 0.5) .. controls +(0.25, 0.25) .. (0.25, 1) \rarrodd;
  \draw[expand style =\styleeven] (0,0.5) -- (0,0) \arreven;
\end{scope}}} =  (-1)^{\myN-1}q^{\frac{1-2\myN}{2}}\NB{\tikz[]{\begin{scope}[scale=1]
  \draw[expand style= \styleodd] (0.25, 1.5) .. controls +(0, -0.2)
  and +(0.1,0.1)
  .. (0, 1.25)\rarrodd .. controls +(-0.1,-0.1) and (-0.25, 1.1) .. (-0.25, 1) .. controls +(0, -0.25)
  .. (0, 0.5)\arrodd;
  \fill[white] (0, 1.25) circle (1mm);
  \draw[expand style= \stylevect] (-0.25, 1.5) .. controls +(0, -0.2)
  and +(-0.1,0.1)
  .. (0, 1.25) .. controls +(0.1,-0.1) and (0.25, 1.1) .. (0.25, 1) .. controls +(0, -0.25)
  .. (0, 0.5);
  % \draw[expand style= \styleeven] (0, 0.5)
  % .. controls +(0.25,0.25) .. (0.25, 1) \arreven;
  \draw[expand style =\styleeven] (0,0.5) -- (0,0) \arreven;
\end{scope}}} .
        \end{gather}
\allowdisplaybreaks[0]
      \end{lem}
      \begin{proof}
        Using Lemma \ref{lem:braid-and-trivalent} and the twists as described in Section \ref{subsec:twists} the proof is an easy application of graphical calculus. Compare to the proof in \cite[Lemma 4.33]{spinlinkhomology}.
      \end{proof}
      \begin{notation}
          We define
          \[
          \II_{\macroeven,\macroodd}^{\macroeven,\macroodd} := \Ya_{V}^{\macroeven,\macroodd} \circ \Pa^{V}_{\macroeven,\macroodd}
          \]
          and similarly define $\II_{\macroodd,\macroeven}^{\macroodd,\macroeven}$, $\II_{\macroeven,V}^{V,\macroeven}$ etc. We also define $\HH_{\macroeven,\macroeven}^{\macroodd,\macroodd}$ etc. to be the ninety degree rotations of the $\II$'s.
      \end{notation}

      \begin{prop}\label{prop:Xing:spin-v} 
        The following identities hold:\begin{align}
            R_{V,\macroeven} &= q^{\frac{1}{2}}\HH_{V,\macroeven}^{\macroeven,V} +
            q^{-\frac{1}{2}} \II_{V,\macroeven}^{\macroeven,V},\\
            R_{V,\macroeven}^{-1} &= q^{-\frac{1}{2}}\HH_{\macroeven,V}^{V,\macroeven} +
            q^{\frac{1}{2}} \II_{\macroeven,V}^{V,\macroeven}, \\
            R_{V,\macroodd} &= q^{\frac{1}{2}}\HH_{V,\macroodd}^{\macroodd,V} +
            q^{-\frac{1}{2}} \II_{V,\macroodd}^{\macroodd,V},\\
            R_{V,\macroodd}^{-1} &= q^{-\frac{1}{2}}\HH_{\macroodd,V}^{V,\macroodd} +
            q^{\frac{1}{2}} \II_{\macroodd,V}^{V,\macroodd}.
          \end{align}
          Graphically, this reads:
          \begin{align}
          \NB{\tikz[yscale=1]{\begin{scope}
  \draw[expand style=\styleeven]     (0, 0) .. controls +(0,0.2) and
  +(0.05, -0.1) .. (-0.25, 0.5) \rarreven  .. controls  +(-0.05, 0.1) and +(0,-0.2)
   .. (-0.5, 1)\rarreven;
  \fill[white] (-0.25, 0.5) circle (1mm);
  \draw[expand style=\stylevect]     (-0.5, 0) .. controls +(0,0.2) and
  +(-0.05, -0.1) .. (-0.25, 0.5)  .. controls  +(0.05, 0.1) and +(0,-0.2)
   .. (0, 1);
\end{scope}}} &= q^{\frac12}\NB{\tikz[yscale=1]{\begin{scope}
  \draw[expand style=\styleeven]     (0, 0) -- +(0, 0.5) \arreven;
  \draw[expand style=\styleeven](-0.5, 0.5) -- +(0, 0.5) \arreven;
  \draw[expand style= \stylevect] (-0.5, 0) -- +(0, 0.5);
  \draw[expand style= \stylevect]  (0, 0.5) -- +(0, 0.5);
  \draw[expand style= \styleodd]   (0, 0.5) -- +(-0.5, 0) \rarrodd;  
\end{scope}}} +
                                      q^{-\frac12}\NB{\tikz[yscale=1]{\input{\imagesfolder/td-I-v-e}}},\\
          \NB{\tikz[yscale=1]{\begin{scope}
  \draw[expand style=\stylevect]     (0.5, 0) .. controls +(0,0.2) and
  +(0.05, -0.1) .. (0.25, 0.5)  .. controls  +(-0.05, 0.1) and +(0,-0.2)
   .. (0, 1);
  \fill[white] (0.25, 0.5) circle (1mm);
  \draw[expand style=\styleeven]     (0, 0) .. controls +(0,0.2) and
  +(-0.05, -0.1) .. (0.25, 0.5) \arreven  .. controls  +(0.05, 0.1) and +(0,-0.2)
   .. (0.5, 1)\arreven;
\end{scope}}} &= q^{-\frac12}\NB{\tikz[yscale=1]{\begin{scope}[xscale=-1]
  \draw[expand style=\styleeven]     (0, 0) -- +(0, 0.5) \arreven;
  \draw[expand style=\styleeven](-0.5, 0.5) -- +(0, 0.5) \arreven;
  \draw[expand style= \stylevect] (-0.5, 0) -- +(0, 0.5);
  \draw[expand style= \stylevect]  (0, 0.5) -- +(0, 0.5);
  \draw[expand style= \styleodd]   (0, 0.5) -- +(-0.5, 0) \arrodd;  
\end{scope}}} +
                                      q^{\frac12}\NB{\tikz[yscale=1]{\input{\imagesfolder/td-I-e-v}}},\\
          \NB{\tikz[yscale=1]{\begin{scope}
  \draw[expand style=\styleodd]     (0, 0) .. controls +(0,0.2) and
  +(0.05, -0.1) .. (-0.25, 0.5) \rarrodd  .. controls  +(-0.05, 0.1) and +(0,-0.2)
   .. (-0.5, 1)\rarrodd;
  \fill[white] (-0.25, 0.5) circle (1mm);
  \draw[expand style=\stylevect]     (-0.5, 0) .. controls +(0,0.2) and
  +(-0.05, -0.1) .. (-0.25, 0.5)  .. controls  +(0.05, 0.1) and +(0,-0.2)
   .. (0, 1);
\end{scope}}} &= q^{\frac12}\NB{\tikz[yscale=1]{\begin{scope}
  \draw[expand style=\styleodd]     (0, 0) -- +(0, 0.5) \arrodd;
  \draw[expand style=\styleodd](-0.5, 0.5) -- +(0, 0.5) \arrodd;
  \draw[expand style= \stylevect] (-0.5, 0) -- +(0, 0.5);
  \draw[expand style= \stylevect]  (0, 0.5) -- +(0, 0.5);
  \draw[expand style= \styleeven]   (0, 0.5) -- +(-0.5, 0) \rarreven;  
\end{scope}}} +
                                      q^{-\frac12}\NB{\tikz[yscale=1]{\input{\imagesfolder/td-I-v-o}}},\\
            \NB{\tikz[yscale=1]{\begin{scope}
  \draw[expand style=\stylevect]     (0.5, 0) .. controls +(0,0.2) and
  +(0.05, -0.1) .. (0.25, 0.5)  .. controls  +(-0.05, 0.1) and +(0,-0.2)
   .. (0, 1);
  \fill[white] (0.25, 0.5) circle (1mm);
  \draw[expand style=\styleodd]     (0, 0) .. controls +(0,0.2) and
  +(-0.05, -0.1) .. (0.25, 0.5) \arrodd  .. controls  +(0.05, 0.1) and +(0,-0.2)
   .. (0.5, 1)\arrodd;
\end{scope}}} &= q^{-\frac12}\NB{\tikz[yscale=1]{\begin{scope}[xscale=-1]
  \draw[expand style=\styleodd]     (0, 0) -- +(0, 0.5) \arrodd;
  \draw[expand style=\styleodd](-0.5, 0.5) -- +(0, 0.5) \arrodd;
  \draw[expand style= \stylevect] (-0.5, 0) -- +(0, 0.5);
  \draw[expand style= \stylevect]  (0, 0.5) -- +(0, 0.5);
  \draw[expand style= \styleeven]   (0, 0.5) -- +(-0.5, 0) \arreven;  
\end{scope}}} +
                                      q^{\frac12}\NB{\tikz[yscale=1]{\input{\imagesfolder/td-I-o-v}}}.
          \end{align}
      \end{prop}
      \begin{proof}
            Given what we have established so far, the proof of \cite[Proposition 4.35]{spinlinkhomology} can easily be adapted to prove this result. The idea is that one can first adapt the proof of \cite[Lemma 4.34]{spinlinkhomology} to argue the $\II$ and $\HH$ diagrams are a basis for this homomorphism space. Thus, there is some expression for the braiding as a $\mathbb{C}(q)$-linear combination of $\II$ and $\HH$. The coefficients can then be determined by placing trivalent vertices on the top and side of the relation, and applying relations (we have already derived) for computing triangles, bigons, and the compositions of the braiding with trivalent vertices.
      \end{proof}

      \begin{lem}\label{L:square}
        The following relations hold:\begin{align}
            \label{eq:inter-square}
            \HH_{V,\macroeven}^{\macroeven,V}\circ \HH_{\macroeven,V}^{V,\macroeven} &= \id_{\macroeven}\otimes \id_{V} - (-1)^{\myN-2}[\myN-2] \II_{\macroeven,V}^{\macroeven,V},\quad\text{and}\\
            \HH_{V,\macroodd}^{\macroodd,V}\circ \HH_{\macroodd,V}^{V,\macroodd} &= \id_{\macroodd}\otimes \id_{V} - (-1)^{\myN-2}[\myN-2] \II_{\macroodd,V}^{\macroodd,V}.            
          \end{align}
        Graphically, they read:
        \begin{align}
          \label{eq:inter-square-diag}
          \NB{\tikz[]{\input{\imagesfolder/td-HH-e-v}}}&=\NB{\tikz[]{\begin{scope}
  \draw[expand style=\styleeven]     (0, 0) -- +(0, 1.5) \arreven;
  \draw[expand style= \stylevect]  (0.5,0) -- +(0, 1.5);
\end{scope}}}- (-1)^{\myN-2}[\myN-2]\NB{\tikz[yscale=1.5, xscale=1]{\input{\imagesfolder/td-Ip-e-v}}}, \quad\text{and}\\
          \NB{\tikz[]{\input{\imagesfolder/td-HH-o-v}}}&=\NB{\tikz[]{\begin{scope}
  \draw[expand style=\styleodd]     (0, 0) -- +(0, 1.5) \arrodd;
  \draw[expand style= \stylevect]  (0.5,0) -- +(0, 1.5);
\end{scope}}}- (-1)^{\myN-2}[\myN-2]\NB{\tikz[yscale=1.5, xscale=1]{\input{\imagesfolder/td-Ip-o-v}}}.
        \end{align}
      \end{lem}
      \begin{proof}
          Combine 
          \[
          R_{V,\macroeven} = q^{1/2}\HH_{V,\macroeven}^{\macroeven,V} + q^{-1/2}\II_{V,+}^{\macroeven,V} \qquad \text{and} \qquad R_{\macroeven,V}^{-1} = q^{-1/2}\HH_{\macroeven,V}^{V,\macroeven} + q^{1/2}\II_{\macroeven,V}^{V,\macroeven}.\qedhere
          \]
      \end{proof}

      \begin{lem}
        The following relations hold if $\myN\geq 3$:
        \begin{equation}
          \label{eq:45}
          \NB{\tikz[]{\input{\imagesfolder/td-diag-square-1}}}=\NB{\tikz[]{\input{\imagesfolder/td-diag-square-2}}}=\NB{\tikz[]{\input{\imagesfolder/td-diag-square-3}}}=\NB{\tikz[]{\input{\imagesfolder/td-diag-square-4}}}.
        \end{equation}
      \end{lem}
      \begin{proof}
          Using Lemma \ref{L:square}, resolve the following diagrams
          \begin{equation}
            \NB{\tikz[]{\input{\imagesfolder/td-diag-dble-square-1}}}, \NB{\tikz[]{\input{\imagesfolder/td-diag-dble-square-2}}}, \NB{\tikz[]{\input{\imagesfolder/td-diag-dble-square-3}}}, \NB{\tikz[]{\input{\imagesfolder/td-diag-dble-square-4}}}
          \end{equation}   
          in two different ways. The claim the follows since
          $[\myN-2]\ne 0$. 
      \end{proof}

      Note that the equality of the various square diagrams in the previous Lemma imply that the diagrams are rotation invariant.

      \begin{notation}
        We introduce notation for this rotation invariant endomorphism
        \begin{equation}
          \XX:=(-1)^{1+\binom{\myN}{2}}\frac{1}{[2]^{[\myN-3]}}\HH_{\macroeven,\macroeven}^{V,V}\circ \HH_{V,V}^{\macroeven,\macroeven}.
        \end{equation}
        Diagrammatically, we write:
        \begin{equation}
          \label{eq:46}
          \NB{\tikz[yscale= 1.5]{\begin{scope}
  \draw[expand style=\stylevect]     (0, 0) .. controls +(0,0.2) and
  +(0.05, -0.1) .. (-0.25, 0.5)  .. controls  +(-0.05, 0.1) and +(0,-0.2)
   .. (-0.5, 1);
  \draw[expand style=\stylevect]     (-0.5, 0) .. controls +(0,0.2) and
  +(-0.05, -0.1) .. (-0.25, 0.5)  .. controls  +(0.05, 0.1) and +(0,-0.2)
  .. (0, 1);
    \node at (-0.25, 0.5) {\singvertex};
\end{scope}}}=
          (-1)^{1+\binom{\myN}{2}}\frac{1}{[2]^{[\myN-3]}} \NB{\tikz[]{\input{\imagesfolder/td-diag-square-1}}}.
        \end{equation}
        
          The choice of sign is to make the formula for $R_{V,V}$ particularly simple.
      \end{notation}

      \begin{prop}\label{prop:braidingVV}
        The following relation holds:
          \begin{equation}
          R_{V,V} = q \id_{V\otimes V} + \XX + q^{-1} \cupp_{V}\circ \capp_{V}.
        \end{equation}
        Diagrammatically, this reads:
        \begin{equation}
          \label{eq:47}
          \NB{\tikz[yscale=1]{\begin{scope}
  \draw[expand style=\stylevect]     (0, 0) .. controls +(0,0.2) and
  +(0.05, -0.1) .. (-0.25, 0.5)  .. controls  +(-0.05, 0.1) and +(0,-0.2)
   .. (-0.5, 1);
  \fill[white] (-0.25, 0.5) circle (1mm);
  \draw[expand style=\stylevect]     (-0.5, 0) .. controls +(0,0.2) and
  +(-0.05, -0.1) .. (-0.25, 0.5)  .. controls  +(0.05, 0.1) and +(0,-0.2)
   .. (0, 1);
\end{scope}}}= q\, \NB{\tikz[yscale=0.667]{\begin{scope}
  \draw[expand style=\stylevect]     (0, 0) -- +(0, 1.5);
  \draw[expand style= \stylevect]  (0.5,0) -- +(0, 1.5);
\end{scope}}}\, +\,
          \NB{\tikz[]{}} \,+ \,q^{-1}\, \NB{\tikz[yscale=0.667]{\begin{scope}
  \draw[expand style=\stylevect]   (0, 0) .. controls +(0, 0.7) and
  +(0,0.7) .. +(0.5,0);
  \draw[expand style=\stylevect]   (0, 1.5) .. controls +(0, -0.7) and
  +(0,-0.7) .. +(0.5,0);
\end{scope}}}
        \end{equation}
      \end{prop}
      \begin{proof}
        First remark that:
        \begin{equation}
          \label{eq:50}
          (-1)^{\binom{\myN-1}{2}}[2]^{[\myN -2]}\NB{\tikz[scale=1.5, yscale=1.5]{}}=  \NB{\tikz[scale=1.5]{\begin{scope}
  \draw[expand style=\stylevect]     (0, 0) .. controls +(0,0.2) and
  +(0.05, -0.1) .. (-0.25, 0.5)  .. controls  +(-0.05, 0.1) and +(0,-0.2)
  .. (-0.5, 1);
  \draw[expand style=\styleodd] (-0.5, 1) .. controls +(0.2, 0.2)
  .. (-0.5, 1.4) \arrodd; 
  \draw[expand style=\styleeven] (-0.5, 1) .. controls +(-0.2, 0.2)
  .. (-0.5, 1.4) \rarreven;
  \draw[expand style=\stylevect] (-0.5, 1.4) -- (-0.5, 1.5);
  \fill[white] (-0.25, 0.5) circle (1mm);
  \draw[expand style=\stylevect]     (-0.5, 0) .. controls +(0,0.2) and
  +(-0.05, -0.1) .. (-0.25, 0.5)  .. controls  +(0.05, 0.1) and +(0,-0.2)
   .. (0, 1) -- +(0, 0.5);
\end{scope}}}
           = \NB{\tikz[scale=1.5, yscale=1.5]{\begin{scope}[/pgf/fpu/install only={reciprocal}]%[scale=3]
  \draw[expand style=\stylevect] (0, 0) -- (0,0.1);
  \draw[expand style=\stylevect] (-0.5, 0.9) -- (-0.5,1);
  \draw[expand style=\styleodd]     (0, 0.1) .. controls +(0,0.8)  and
  +(0, 0)..  (-0.5, 0.9) \rarrodd ; 
  \draw[expand style=\styleeven]     (0, 0.1) .. controls +(-0.0,0) and
  +(0, -0.8) ..  (-0.5, 0.9)  \arreven;
  % \coordinate[pos=0.4] (b)
   \fill[white] (-0.37, 0.35) circle (0.8mm and 0.6mm);
   \fill[white] (-0.12, 0.65) circle (0.8mm and 0.6mm);
  \draw[expand style=\stylevect]     (-0.5, 0) .. controls +(0,0.2) and
  +(-0.05, -0.1) .. (-0.25, 0.5)  .. controls  +(0.05, 0.1) and +(0,-0.2)
   .. (0, 1);
\end{scope}}}.
         \end{equation}
         Using the Proposition~\ref{prop:Xing:spin-v}, one can
         express the braidings between the vectorial representation
         and the spin ones. Finally, one uses planar relations, to
         obtain the result. This last step is left as an exercise to the reader. 
\end{proof}

      \begin{cor}\label{cor:skein-diag}
        The following relation holds:
        \begin{equation}
          \label{eq:skein-diag}
          \NB{\tikz[yscale=0.75]{}}\,- \,\NB{\tikz[xscale=-1,yscale=0.75]{}} \,= (q- q^{-1})\left(
            \, \NB{\tikz[yscale=0.5]{}} \,- \,
            \NB{\tikz[yscale=0.5]{}}\,\right).
        \end{equation}
      \end{cor}

      \begin{lem}
        The following relations hold:
        \begin{equation}
          \label{eq:R1s-diag}
          {\NB{\tikz[scale=0.5]{}}} = q^{2\myN -1} {\,\NB{\tikz[scale=0.5]{}}\,} \qquad \qquad  \text{and} \qquad \qquad
          {\NB{\tikz[scale=0.5, yscale=-1]{}}} = q^{1-2\myN}
          {\,\NB{\tikz[scale=0.5]{}}\,}. 
        \end{equation}
      \end{lem}

      \begin{proof}
          To compute the first equality in \eqref{eq:R1s-diag}, observe that the left hand side is the twist, so the discussion in Section \ref{subsec:twists} implies the scalar on the right hand side must be $(-1)^{(2\varpi_1, \rho^{\vee})}q^{(\varpi_1 , \varpi_1 + 2\rho)} = q^{2\myN - 1}$. A diagrammatic argument establishes the second equality by verifying that the left hand side of the second equality is the inverse of the twist endomorphism.
      \end{proof}

If $D$ is a diagram of a framed unoriented link (thought of as
      colored with the vector representation), let us denote by $\RT{D}$ the scalar in $\mathbb{C}[q, q^{-1}]$ obtained by viewing $D$ as an endomorphism of the
      trivial representation of $U_q(\sotN)$. From the
      Reshetikhin--Turaev framework, one knows that this quantity is
      independent of the choice of diagram, so that for a framed
      unoriented link $L$, the quantity $\RT{L}$ is well-defined. Moreover, we can relate $\RT{-}$ to the link invariant $\evN{\cdot}$
      developed in Section~\ref{sec:linkinvariants}, proving Theorem \ref{theo:C} from the introduction.
      \begin{thm}\label{thm:ev-vs-alg}
        For any framed link $L$, $\RT{L} =
        \left(\evN{L}\right)_{q\mapsto -q}$.
      \end{thm}

        \begin{proof}
      The invariant $\RT{-}$ satisfies the
      skein relations \eqref{eq:skein-diag}, \eqref{eq:unknot-v-diag},
      \eqref{eq:R1s-diag}, and the invariant $\left(\evN{L}\right)_{q\mapsto -q}$ satisfies the same skein relations, see \eqref{eq:skein-ev-link}, \eqref{eq:17}, and \eqref{eq:R1s}. Moreover, these skein relations completely determine each invariant
       \cite{MR958895}.
        \end{proof}

        Note that Theorem \ref{thm:ev-vs-alg} implies $\RT{-}$ is $\ZZ[q, q^{-1}]$-valued.

    \begin{rmk}
            Observe that we did not need to derive the analogue of the pentagon relation satisfied by $\ev{\cdot}$ from Proposition \ref{prop:pentagon}. This pentagon relation was used to verify the braid relation for the state sum invariant $\ev{\cdot}_{\myN}$. However, we already know that the representation category for $U_q(\sotN)$ is braided (and have used this fact extensively to derive relations between intertwiners). We leave it as an exercise to derive $\HH_{V,V}^{\macroeven,\macroeven}\circ R_{V,V}$, following the rubric from the proof of \cite[Corollary 4.36]{spinlinkhomology}, and then deriving the pentagon relation for trivalent intertwiners using graphical calculus.
    \end{rmk}

    \begin{rmk}
      We did not prove that there is a re-normalization of
      the algebraic evaluation of webs which matches the
      combinatorial one. We believe that considering pivotal
      structures should shed some light on whether there is such a renormalization.
    \end{rmk}

\bibliographystyle{alphaurl}
\bibliography{biblio}

\end{document}